\documentclass[final,letterpaper]{siamart171218}

\usepackage{amsmath,amssymb}
\usepackage{mathtools}

\usepackage{picture,slashbox}
\usepackage{graphicx}
\usepackage{subfigure}
\usepackage{color}
\usepackage{xcolor}
\usepackage{comment}
\usepackage{url}
\usepackage{enumitem}

\graphicspath{{./figures/}}

\usepackage{tabularx}

\usepackage{array,booktabs}
\usepackage{siunitx}
\usepackage{float}
\usepackage{algorithm}
\usepackage{algpseudocode}
\usepackage{caption}

\newtheorem{remark}[theorem]{Remark}

\DeclareMathOperator\erf{erf}
\DeclareMathOperator\erfc{erfc}

\newcommand{\be}{\begin{equation}}
\newcommand{\ee}{\end{equation}}
\newcommand{\ba}{\begin{aligned}}
\newcommand{\ea}{\end{aligned}}

\newcommand{\x}{\boldsymbol{x}}
\newcommand{\y}{\boldsymbol{y}}

\newcommand{\bc}{\boldsymbol{c}}

\newcommand{\bk}{\boldsymbol{k}}
\newcommand{\bm}{\boldsymbol{m}}
\newcommand{\bn}{\boldsymbol{n}}

\newcommand{\bq}{\boldsymbol{q}}

\newcommand{\psic}{{\psi_0^c}}
\newcommand{\hpsic}{{\hat{\psi}_0^c}}
\newcommand{\bxi}{\boldsymbol{\xi}}
\newcommand{\cf}{\mathcal{F}}
\newcommand{\cd}{\mathcal{D}}
\newcommand{\bfa}{\boldsymbol{a}}

\newcommand{\bff}{\boldsymbol{f}}
\newcommand{\bi}{\boldsymbol{i}}
\newcommand{\bj}{\boldsymbol{j}}
\newcommand{\br}{\boldsymbol{r}}
\newcommand{\acron}{DMK }
\newcommand{\cR}{r}



\def\ccm{Center for Computational Mathematics, Flatiron Institute, Simons Foundation,
  New York, New York 10010}

\def\nyu{Courant Institute of Mathematical Sciences,
  New York University, New York, New York 10012}

\def\papertitle{A dual-space multilevel kernel-splitting framework
for discrete and continuous convolution}

\title{\papertitle}
\author{Shidong Jiang\thanks{\ccm\,
    ({\tt sjiang@flatironinstitute.org})}
  \and
Leslie Greengard%
  \thanks{\ccm\, \& \nyu\,
    ({\tt greengard@courant.nyu.edu}).}}
  
\begin{document}

\maketitle

\begin{abstract}
We introduce a new class of multilevel, adaptive, dual-space 
methods for computing fast convolutional transforms.
These methods can be applied to a broad class of 
kernels, from the Green's functions for classical 
partial differential equations (PDEs) to power functions and radial basis
functions such as those used in statistics and machine learning. 
The \acron ({\em dual-space multilevel kernel-splitting}) framework 
uses a hierarchy of grids, computing a smoothed interaction at the 
coarsest level, followed by a sequence of corrections at finer and finer scales
until the problem is entirely local, at which point direct summation is applied. 

Unlike earlier multilevel summation schemes, \acron exploits the fact that
the interaction at each scale
is diagonalized by a short Fourier transform, permitting the use of separation of 
variables, but without relying on the FFT. This requires careful attention to the
discretization of the Fourier transform at each spatial scale.
Like multilevel summation, we make use of a recursive (telescoping) decomposition of the
original kernel into the sum of a smooth far-field kernel, a sequence of difference kernels, and 
a residual kernel, which plays a role only at leaf boxes in the adaptive tree.
At all higher levels in the grid hierarchy,
the interaction kernels are designed to be smooth in both physical 
and Fourier space, admitting efficient Fourier spectral approximations.
The DMK framework substantially simplifies the algorithmic 
structure of the fast multipole method (FMM) and unifies the FMM, Ewald summation,
and multilevel summation, achieving speeds comparable to the FFT in work per
gridpoint, even in a fully adaptive context. 

For continuous source distributions, the evaluation of local interactions is 
further accelerated by approximating the kernel at the finest level 
as a sum of Gaussians with a highly localized remainder. The Gaussian
convolutions are calculated using tensor product transforms, and the remainder
term is calculated using asymptotic methods.  We illustrate the performance 
of \acron for both continuous and discrete sources with extensive numerical examples in two 
and three dimensions.
\end{abstract} 



\pagestyle{myheadings}
\thispagestyle{plain}
\markboth{S. Jiang and L. Greengard}
{A dual-space multilevel kernel-split framework for fast transforms}

%%%%%%%%%%%%%%%%%%%%%%%%%%%%%%%%%%%%%%%%%%%%%%

%
%
%

\section{Introduction}\label{intro}

Many problems in computational science involve the  
evaluation of discrete sums of the form
\be\label{discretesum}
u_i=\sum_{j=1}^{N}K(\x_i-\y_j)\rho_j, \qquad i=1,\ldots,N,
\ee
or continuous convolutions of the form
\be\label{volumepotential}
u(\x)=\int_B K(\x-\y)\rho(\y)d\y,
\ee
where $\x, \y\in \mathbb{R}^d$ and $B$ is a rectangular box in $\mathbb{R}^d$.
In the present paper, we restrict our attention to kernels 
$K(\x)$ which are not highly oscillatory.
Examples of such kernels include many of the Green's functions governing 
classic partial differential equations, 
the Matern kernels in statistics
and machine learning, the general power function $1/|\x-\y|^\alpha$,
and the radial basis functions used for function
approximation from point clouds of data.
For both the discrete and continuous cases, we will refer to $\rho$ as the charge 
and to $u$ as the potential. 
Since this class of problems is ubiquitous, the 
corresponding literature is vast.
We will not attempt a comprehensive review here, but simply point out that
there are, broadly speaking, two classes of fast algorithms
for such problems: (a) tree-based methods such as the fast multipole 
method (FMM) and its variants
\cite{carrier1988sisc,fmm2,fmm3,fmm4,greengard1987jcp,fmm6,rokhlin1985jcp,
fmm7,zhang2011jcp}
or multilevel summation
\cite{brandt1990jcp,brandt1998,hardy2009pc,multilevel_summation_2015,
multilevel_summation_bspline,tensor_multilevel_ewald,skeel2002},
and (b) uniform grid-based methods that rely on the 
fast Fourier transform (FFT) such as Ewald summation
(see, for example, 
\cite{darden1993jcp,HE-1981,shamshirgar2021jcp,shaw_gaussian_ewald,TOUKMAJI199673}).
The tree-based methods have the advantage of permitting adaptive discretization
in the case of either continuous or discrete sources, and can achieve linear 
scaling.
The FFT-based methods, on the other hand, achieve $O(N \log N)$ complexity for 
more or less uniform
discretizations with a small constant prefactor implicit in the
$O(N \log N)$ notation. They are typically preferred in applications 
that do not require adaptivity because of their ease of implementation and the wide
availability of high-performance packages such as FFTW \cite{fftw97,fftw98}.
In the discrete setting, the FFT-based methods can be understood in one of two
ways: (a) mollification of the 
given source distribution, using the FFT to diagonalize 
the convolution, followed by a local correction step, or
(b) splitting of the kernel $K(\x)$ into a smooth far-field component and a 
singular near-field component  (Ewald's original approach). 
Both approaches can also be recast as applications of
the nonuniform fast Fourier transform (NUFFT) 
\cite{finufft,nufft2,nufft3,nufft6,potts2004}.
Finally, we should also note that a ``multi-level Ewald" method was proposed in 
\cite{multilevel_ewald}, aimed primarily at achieving better 
parallel FFT performance.

The \acron method developed here
draws on all of the ideas mentioned above, as well as 
multiresolution methods 
\cite{beylkin2009jcp,beylkinmohl2002,beylkin2010acha,harrisonbeylkin2004}.
It combines hierarchical, tree-based kernel-splitting with spatially localized Fourier 
transforms and leads to {\em adaptive} methods with $O(N)$ complexity.
While it is described in detail below, we summarize the novel 
features that make it faster and more general than prior approaches:

\vspace{.2in}

\begin{enumerate}
\item Unlike Ewald methods, \acron {\em does not} depend on the FFT for its asymptotic
complexity.
(The lengths of the transforms used internally are independent of $N$, and 
sufficiently short that all such calculations could be carried out directly.)
\item Unlike Ewald methods, it is fully adaptive.
\item Like Ewald methods, and unlike previous multilevel summation methods, Fourier
analysis is exploited to diagonalize the convolutions needed at each level.
\item Like Ewald methods, at the finest level in the adaptive tree, 
the interaction kernels are local and compactly supported. 
\item Unlike fast multipole methods (FMMs), \acron does not separate the 
near and far fields, and does not rely on properties of the governing partial 
differential equation for compression.
For readers familiar with the details of FMMs, the ``interaction list" processing 
in \acron is much simpler - requiring only communication with near neighbors 
at each level in the tree hierarchy.
\item Like fast multipole methods, \acron involves a single upward pass and
a single downward pass over an adaptive tree hierarchy.
\item Unlike earlier multiresolution approaches, the kernel splitting in \acron
involves only a single convolution kernel at each level in the tree hierarchy.
\item Unlike earlier multilevel summation methods, \acron
uses Fourier analysis to diagonalize interactions at every level.
\end{enumerate}

\vspace{.2in}

In short, \acron blends the algorithmic structure of the FMM and
multilevel methods with that of Ewald summation by combining
hierarchical kernel-splitting with localized, spatially adaptive Fourier
convolution.
At each level, a different degree of smoothing is applied:
the longest range interactions are accounted for 
at the coarsest levels, where the smoothing is greatest. Corrections are then
computed on successively finer levels, corresponding to sharper
but more localized features. Detailed Fourier analysis of these
localized kernels shows that, for any fixed precision, 
a modest number of Fourier modes is needed at each scale, permitting an efficient,
separable representation of the interaction.
It is worth noting that,
from the viewpoint of numerical linear algebra, the FMM and its descendants
(including ${\mathcal H}$-matrices and skeletonization-based schemes) 
\cite{fmm3,hackbusch1999,hackbusch2002,ho2012fast,minden:2017a:multiscale-model-simul}
carry out hierarchical low-rank compression {\em directly} on the given 
interaction matrix with entries $K(\x_i-\y_j)$.
For singular, nonoscillatory kernels, this leads to the requirement that source and
target boxes be ``well-separated" in order for the sub-blocks of the 
interaction matrix to be sufficiently low rank. 
In the \acron framework (and in multilevel summation methods more generally), 
the interaction matrix
is split into a {\em sum} of matrices - one for each spatial scale - 
that are dense (but low-rank)
at coarse levels and sparse at finer levels. 

We refer to \acron as a {\em framework} since the structure of the algorithm
is largely kernel- and dimension-independent,
using only generic, tensor-product Fourier convolution. The details, however, are
kernel-specific, requiring some analysis in both the Fourier transform domain
and physical space.  Numerical experiments show that \acron leads to 
implementations that are as fast or faster than state-of-art codes, such
as the PVFMM~\cite{pvfmm} and FMM3D~\cite{fmm3d} libraries for 
translation-invariant Green's functions,
and that they {\em retain the same speed} for more general kernels, unlike earlier
FMM variants such as those in \cite{fmm3,fmm4}.

The paper is organized as follows. 
In section 2, we summarize some of mathematical prerequisites.
In section 3, we provide a detailed description of an adaptive \acron
algorithm for the Poisson equation in three dimensions.
In section 4, we show how to extend the \acron approach to several other kernels,
and in section 5, we present the results of numerical experiments.
In section 6, we discuss some potential algorithmic improvements 
and extensions of \acron to other problems in computational physics. 
\section{Mathematical preliminaries and notation}

We begin by summarizing the main mathematical tools and definitions to
be used throughout the paper, with some material relegated to
appendix A.
For a point $\x \in \mathbb{R}^d$, we will denote its magnitude by 
$r=|\x|$.  Points in the Fourier transform domain will be denoted by 
$\bk$ with magnitude $k=|\bk|$. When a function $f$ is radially symmetric,
we will use $f(r)$ and $f(\x)$ interchangeably.

We will make use of multi-index notation. That is, for a $d$-tuple
of integers (a multi-index)
of the form
$\bn=(n_1,\ldots,n_d)$, 
a point $\x = (x_1,\dots,x_d)$ in $\mathbb{R}^d$, and a scalar
function of one variable $T_n(x)$, we define
$T_{\bn}(\x)=\prod_{i=1}^d T_{n_i}(x_i)$ and
\[ 
 \sum_{\bn \in [1,\dots,p]^d} C_{\bn} T_{\bn}(\x) \equiv
 \sum_{n_1 = 1}^{p} 
 \dots \sum_{n_d = 1}^{p} 
C_{n_1,\dots,n_d} T_{n_1-1}(x_1) \cdots T_{n_d-1}(x_d).
\]
We will permit the multi-index to take on negative values. For example,
we will make extensive use of formulas such as
\[ 
 \sum_{\bm \in [-p,\dots,p]^d} w_{\bm} e^{i \bk_{\bm} \cdot \x} \equiv
 \sum_{m_1 = -p}^{p} 
 \dots \sum_{m_d = -p}^{p} 
w_{m_1,\dots,m_d} e^{i \bk_{\bm} \cdot \x},
\]
where $\bk_{\bm} = (k_{m_1},\dots,k_{m_d})$.

\subsection{Interpolation} \label{sec:interp}

Let $f(x)$ be a smooth function on the interval $[-1,1]$.
Then it is well-known to have a rapidly converging Chebyshev series
\cite{boyd2000,trefethenatap}
\[ 
f(x) \approx \sum_{n=1}^{p} a_n T_{n-1}(x).
\]
The Chebyshev polynomials \cite{boyd2000,trefethen2000} can be defined by 
the recurrence relations
\[ T_0(x) = 1,\quad T_1(x) = x,\quad T_{n+1}(x) = 2x T_n(x) - T_{n-1}(x).  \]
The {\em Chebyshev nodes of the first kind} are the zeros of $T_p(x)$, given by 
\[ 
\{ r_i =  \cos \left[  \frac{\pi (i-\frac{1}{2})}{p} \right], i = 1,\dots,p \}.
\]
We define the $p \times p$ ``Vandermonde-like'' matrix $V$ by $V(i,j) = T_{j-1}(r_i)$
\cite{higham1990}
Given the vector of function values $\bff = (f(r_1),f(r_2),\dots,f(r_p))$,
the coefficients of the Chebyshev interpolant can be obtained as
\[ \bfa = V^{-1} \bff ,  \]
where $\bfa = (a_1,\dots,a_p)$.

Given a set of $N$ additional points $\{ -1 \leq \xi_k \leq 1 |\ k = 1,\dots,N \}$, we define
the $N \times p$ {\em evaluation} matrix $E$ by $E(k,n) = T_{n-1}(\xi_k)$, so that
$E \, \bfa$ is the value of the interpolant at the additional points.
The mapping from function values 
at the Chebyshev nodes $\{r_j\}$ to function values at the additional 
points $\{ \xi_k \}$ is clearly given
by  the {\em interpolation matrix} $U = E \, V^{-1} \in \mathbb{R}^{N\times p}$. 

In $d$ dimensions, we denote by 
$\{ \br_{\bi} \}$ the tensor product Chebyshev nodes. That is,
$\br_{\bi} = (r_{i_1},\dots,r_{i_d})$ for the multi-index $\bi$.
We will continue to 
denote the corresponding $p^d \times p^d$ matrix by $V$ and
write $V(\bi,\bj) = T_{\bj-{\bf 1}}(\br_{\bi})$ for Chebyshev interpolation.
Here, the multi-index ${\bj-{\bf 1}} = (j_1-1,\dots,j_d-1)$.
Given the function values $\bff = f(\br_{\bj})$ for $\bj \in [1,\dots,p]^d$, the 
coefficients of the Chebyshev interpolant are given by 
\[ \bfa = V^{-1} \bff  \]
where we are omitting the details of ``unrolling" the multi-indices and
$\bfa = \{ a_{\bi} | \, \bi \in [1,\dots,p]^d \}$.
For $N$ additional points $\{ \bxi_k \in [-1,1]^d |\ k = 1,\dots,N \}$, 
the $N \times p^d$ evaluation matrix $E$ is given by $E(k,\bn) = T_{\bn-{\bf 1}}(\bxi_k)$
and the interpolation matrix by 
$U = E \, V^{-1} \in \mathbb{R}^{N\times p^d}$. 

If the targets themselves lie on a tensor product grid with multi-index $\bq$, we will
write $E(\bq,\bn)$ for the corresponding $q^d \times p^d$ evaluation matrix. The 
interpolation matrix is again defined by
\be \label{Udefnd}
U = E \, V^{-1} \in \mathbb{R}^{q^d \times p^d}. 
\ee

\subsection{Anterpolation} \label{sec:anterp}

Suppose now that, in one dimension, 
we have a potential function expressed in terms of some smooth kernel $K(x)$ of the form
\be
u(x) = \sum_{j=1}^N \rho_j K(x-x_j),
\ee
with $x_j \in [-1,1]$.
It is convenient to adopt the language of electrostatics, and we will
refer to $\rho_j$ as a {\em charge} located at $x_j$. 
Since $K(x)$ is a smooth function, viewing $x$ as fixed, we may write
\be
u(x) \approx \sum_{j=1}^N \rho_j \left[ \sum_{i=1}^p K(x-r_i) U(j,i) \right],
\ee
where $U$ is the interpolation matrix mapping from the Chebyshev nodes $\{r_i\}$ to the
points $\{ x_j \}$.
Changing the order of summation, we may write
\be
u(x) \approx \sum_{i=1}^p \tilde{\rho}_i K(x-r_i),
\ee
with 
\be \label{proxy1d} 
\tilde{\rho}_i = \sum_{j=1}^N U^T(i,j) \rho_j.
\ee

\begin{definition}\label{anterpdef}
The transpose of the interpolation matrix, $U^T = V^{-T} \, E^T$ is generally
referred to as 
the ``anterpolation" matrix. The charges $\{ \tilde{\rho}_j \}$ defined
in \eqref{proxy1d} will be referred to as ``proxy" charges.
In higher dimensions, proxy charges are defined in the same way. That is,  
for a smooth potential function such as
\be \label{potfun_nd}
u(\x) = \sum_{j=1}^N \rho_j K(\x-\x_j),
\ee
we have
\be
u(\x) \approx \sum_{\bi \in [1,\dots,p]^d} \tilde{\rho}_{\bi} K(\x-\br_{\bi}),
\ee
where
\be \label{proxynd} 
\tilde{\rho}_{\bi} = \sum_{j=1}^N U^T(\bi,j) \rho_j.
\ee
\end{definition}

Anterpolation is a key step in accelerating the ``fine-to-coarse" transition
in a variety of tree-based methods  
\cite{brandt1990jcp,brandt1998,skeel2002,hardy2009pc,multilevel_summation_2015,
multilevel_summation_bspline,tensor_multilevel_ewald} and, in $d$ dimensions,
the cost of computing \eqref{proxynd} is clearly $O(N p^d)$. 
If the sources $\rho_j$ in \eqref{potfun_nd} lie on a tensor-product grid, however, 
that cost can be reduced.

\begin{lemma} \label{ctoplemma}
Suppose that the charges $\rho_{j}$ in \eqref{potfun_nd} 
lie on a tensor product Chebyshev grid contained within $[-1,1]^d$. Then the mapping
\be \label{proxytensor} 
\tilde{\rho}_{\bi} = \sum_{\bj \in [1,\dots,p]^d} U^T(\bi,\bj) \rho_{\bj}
\ee
can be computed in $O(p^{d+1})$ operations, where 
$U$ is defined in \eqref{Udefnd}.
\end{lemma}

The proof follows from a straightforward application of 
separation of variables, since the 
Chebyshev nodes $\{\br_{\bi} \}$ are themselves on a tensor product grid. 
(\Cref{ctoplemma} is used in
\cref{sec:afemmodifications} below
to merge the proxy charges from child boxes to a smaller set of proxy
charges for their parent.)

Given a smooth function on the box $B = [-1,1]^d$, we will also need to shift the
corresponding Chebyshev expansion to a box $C$ of width $1$ centered at
$\bc = (\pm \frac{1}{2},\dots,\pm \frac{1}{2})$. (In the context of hierarchical
tree-based solvers, one can think of $B$ as a parent box and $C$ as one of its 
$2^d$ children.)

\begin{lemma} \label{ptoclemma}
Suppose that we are given a tensor-product polynomial 
on a box $B$ of the form
\[ u(\x) =
\sum_{\bn \in [1,\dots,p]^d} \alpha_{\bn} T_{\bn-{\bf 1}}(\x),
\]
where the Chebyshev polynomials are centered at the center of $B$ and scaled to the
box size.  Then, in a child box $C \subset B$,  
\[ u(\x) =
\sum_{\bn' \in [1,\dots,p]^d} \beta_{\bn'} T'_{\bn'-{\bf 1}}(\x),
\]
where the Chebyshev polynomials $T'_{\bn'-{\bf 1}}(\x)$ are
centered on $C$ and scaled to its size.
This translation is exact and 
\be \label{ptoc} 
\beta_{\bn'} = V^{-1} E \, \alpha_{\bn}
\ee
where $E(\bk,\bn) = T_{\bn-{\bf 1}}(\br'_{\bk})$ is the evaluation matrix
entry for scaled Chebyshev node $\br'_{\bk}$ on $C$ and $V$ is the
(multi-index) Vandermonde matrix. The mapping in
\eqref{ptoc} can be computed in $O(p^{d+1})$ operations.
\end{lemma}

The exactness of the translation is obvious since we are simply shifting the center
of a polynomial in $d$ variables. The reduction in computational cost from 
the naive estimate $O(p^{2d})$ to $O(p^{d+1})$ follows from
separation of variables. 

\begin{remark} \label{polyerror}
The order of polynomial approximation to be used 
depends on the desired precision of the calculation and the smoothness of the kernel
in \eqref{potfun_nd}.
Using Chebyshev interpolation, we have the 
standard error estimate \cite{powell1981,trefethenatap}
\[ |f(x) - P_{n-1}(x)| \leq \frac{1}{2^n \, n!} \left( \frac{b-a}{2} \right)^n 
\| f^{(n)} \|_\infty  \, .
\]
Below, we will focus on the systematic use of band-limited (or approximately 
band-limited) functions, whose Fourier transform is supported in
$[-K_{\rm max},K_{\rm max}]$, which satisfy 
\[ \| f^{(n)}\|_\infty \leq K_{\rm max}^n \| f \|_\infty  \, .
\]
If the approximation is used on an interval of size $1/K_{\rm max}$, then clearly
\[ |f(x) - P_{n-1}(x)| = O \left( \frac{1}{2^n \, n!} \right)
\| f \|_\infty  \, ,
\]
indicating rapid convergence with $n$, and {\em independent of $K_{\rm max}$}. 
\end{remark}

\subsection{The Fourier transform and its properties} \label{sec:ft}

We define the Fourier transform of $F$ by
\be
\label{fourierdef}
\widehat{F}(\bk) = \int e^{-i\bk\cdot \x} F(\x) d\x,
\ee
for $\x, \bk\in \mathbb{R}^d$.
The function $F(\x)$ can be recovered from the inverse transform
\be
\label{fourierinvdef}
F(\x) =\frac{1}{(2\pi)^d} \int e^{i\bk\cdot \x} \widehat{F}(\bk) d\bk.
\ee
Some well-known properties of the Fourier transform are
\cite{dym1975,trefethen2000}:

\vspace{.1in}

\begin{enumerate}[label=(\alph*)]
\item 
If $c\in \mathbb{R}$ with $c\ne 0$, then
  $\mathcal{F}\left\{f\left(\frac{\x}{c}\right)\right\}(\bk)=|c|^d\hat{f}(|c|\bk)$.
\hfill
({\em Dilation}) 
\item 
$\mathcal{F}\left\{\Delta^\alpha f(\x)\right\}(\bk) = -|\bk|^{2\alpha}\hat{f}(\bk)$.
\hfill
({\em Differentiation})
\item 
$f(\x)$ is smooth if and only if $\hat{f}(\bk)$ decays rapidly as $k\rightarrow \infty$.
\hfill ({\em Duality}) 
\item
The Poisson summation formula \cite{dym1975} states that
\end{enumerate}
\be\label{poissonsummation}
\sum_{n=-\infty}^\infty f\left(x+\frac{2\pi n}{h}\right) = \frac{h}{2\pi}\sum_{m=-\infty}^\infty
\hat{f}(mh)e^{imhx}.
\ee
This holds for a broad class of functions, and extends to distributions such
as the Dirac delta function.

\section{\acron for the 3D Laplace kernel}

In this section, we consider the free-space Green's function for the Laplace equation
in three dimensions, and begin with a review
of how classical (single level) Ewald summation
\cite{HE-1981,darden1993jcp,shamshirgar2021jcp})
can be applied to the calculation of the discrete sum
\be\label{laplacesum}
u_i=\sum_{\substack{j=1}}^{N}\frac{1}{|\x_i-\y_j|}\rho_j, \qquad i=1,\ldots,N,
\ee
where (for the moment) we assume the targets $\{ \x_i \}$ are disjoint from
the sources $\{ \y_j \}$.
We then show how to create, using multilevel kernel-splitting, a
fast and fully adaptive variant of the method with 
$O(N \log N)$ complexity. Simple modifications of the method yield an even more
efficient $O(N)$ method. 
In \cref{sec:continuoussources}, we show that significant
acceleration can be obtained for continuous sources at the level of leaf boxes 
by approximating the kernel using a sum-of-Gaussians combined with 
asymptotic analysis.  Finally,
in \cref{sec:betterkernelsplitting}, we show that a
further reduction in cost can be achieved using prolate spheroidal wave functions
in the kernel-splitting framework.

\subsection{Classical Ewald summation}\label{sec:l3dgaussianks}

In order to reduce the cost of the calculation \eqref{laplacesum},
Ewald \cite{ewald1921ap} began by rewriting
the 3D Laplace kernel in the form
\be\label{ewaldsplitting}
\frac{1}{r} = M(r)+R(r)\coloneqq \frac{\erf({r}/\sigma)}{r}
+ \frac{\erfc({r}/\sigma)}{r},
\ee
where $\erf$ and $\erfc$ are the error and complementary error functions
\be\label{erfdef}
\erf(x)=\frac{2}{\sqrt{\pi}}\int_0^x e^{-t^2}dt, \quad \erfc(x)=1-\erf(x),
\ee
and $\sigma$ is a free parameter, to be chosen for optimal performance later.
Note that, once $r \geq 6 \sigma$, $\erf({r}/\sigma) \approx 1$ with more than
fifteen digits of accuracy and 
$\erfc({r}/\sigma) \approx 0$.
Thus, $M(r)$ accurately represents the 
$\frac{1}{r}$ potential in the far field and $R(r)$
can be viewed as a local correction that needs to be invoked for 
$r < 6 \sigma$. 
Note also that 
the {\em mollified} far-field kernel $M(r)$ is smooth 
but slowly decaying, while
the {\em residual} kernel $R(r)$ is singular but compactly supported (to a fixed
precision).
The smoothness of $M(r)$ 
follows from the fact that the Taylor series of 
the error function $\erf(r)$ about the origin contains only odd
powers of $r$. 

Combining \cref{ewaldsplitting} and \cref{laplacesum}, we may write
\be\label{laplacesum2}
u_i= u_{i}^{\rm far}+u_{i}^{\rm local}\\
\ee
where
\[
u_{i}^{\rm far} = \sum_{j=1}^{N}M(|\x_i-\y_j|)\rho_j, \quad
u_{i}^{\rm local} = \sum_{j=1}^{N}R(|\x_i-\y_j|)\rho_j.
\]

\begin{remark}
To evaluate the field at one of the source locations, 
say $\y_{i}$,
it is physically sensible to set the 
``self-interaction" to zero. For this, we  may write
\be\label{laplacesum2self}
\ba
u(\y_{i}) &= u_{i}^{\rm self}+u_{i}^{\rm far}+u_{i}^{\rm local}\\
&\coloneqq -\frac{2}{\sqrt{\pi} \sigma}\rho_i +
\sum_{j=1}^{N}M(|\y_i-\y_j|)\rho_j + 
\sum_{\substack{j=1 \\ j\ne i}}^{N}R(|\y_i-\y_j|)\rho_j.
\ea
\ee
The first term is a ``self-interaction" correction, needed to account for
the fact that, in the smoothed far field sum,
\be
\lim_{r\rightarrow 0} M(r) = \frac{2}{\sqrt{\pi} \sigma}.
\ee
\end{remark}

The far field contribution is amenable to Fourier-based methods. 
It is straightforward to check that the Fourier
transform of the mollified far-field kernel $M(|\x|)$ is
given by
\be\label{erfoverrft}
\widehat{M}(\bk) = 4\pi \frac{e^{-k^2\sigma^2/4}}{k^2},
\ee
for $\bk \in R^3$, where $k = |\bk|$. It decays rapidly, but is 
singular at $k=0$.

\begin{remark}
In its original form, Ewald summation was 
developed for problems with periodic boundary conditions
\cite{darden1993jcp,ewald1921ap,HE-1981}.
In that setting, the far field contribution is expressed as a 
rapidly converging Fourier series, avoiding the Fourier transform 
along with issues of quadrature in the Fourier domain.
The singularity of $\widehat{F}(\bk)$ at the origin is ignored, because of 
the requirement of charge neutrality in the unit box for periodic problems. 
In recent years, however, Ewald methods have been extended 
to address problems with either periodic or free-space boundary conditions
in any coordinate direction.
See \cite{bagge2022,klinteberg2017,shamshirgar2021jcp} for a detailed discussion.
\end{remark}

In this paper, we restrict our attention to
the free-space problem and modify the original
Ewald method accordingly, assuming only that
the point locations $\{ \x_i, \y_i | \ i=1,\ldots,N \}$  have been 
rescaled and centered so that they lie 
in the unit box $[-1/2,1/2]^d$.

Since we will compute the far field interactions via the Fourier transform,
it is convenient to replace the 
kernel $M(|\x|)$ with a {\em windowed} kernel, which we define as
\be
\label{windowedkerps}
W(\x)=\frac{1}{(2\pi)^3} \int_{\bk \in \mathbb{R}^3} 
e^{i\bk\cdot \x} \widehat{W}(\bk) \, d\bk, 
\ee
where
\be\label{windowedkernelft}
\widehat{W}(\bk) = 
8\pi \left(\frac{\sin(\tilde{C}|\bk|/2)}{|\bk|}\right)^2 e^{-|\bk|^2\sigma^2/4},
\ee
a $C^{\infty}$ function, depending on a parameter $\tilde{C}$.
This makes discretization very simple, since the 
integrand in \eqref{windowedkerps} is
smooth and exponentially decaying, so that the tensor
product trapezoidal rule yields spectral accuracy 
(see, for example, \cite{trefethentrap}). It is only permissible to do this,
of course, if the substitution of $W(\x)$ for $M(\x)$ yields the same result,
as established in \cref{trunc_lemma}. 

\begin{remark} \label{vicoremark}
The elimination of the singularity in 
$\widehat{M}(\bk)$ at the origin in \cref{erfoverrft} by 
applying a sharp window in physical space was developed systematically in
\cite{vico2016jcp}, where it was applied to a variety of constant coefficient PDEs
with smooth, compactly supported right-hand sides. The basic observation (for the 
Laplace kernel) is that the kernel $1/r$ truncated at a distance $\tilde{C}$
has the smooth Fourier transform
\[
8 \pi \left(\frac{\sin(\tilde{C}|\bk|/2)}{|\bk|}\right)^2 .
\]
Here, we apply this approach to the mollified kernel $M(|\x|)$.
\end{remark}

\begin{lemma}  \label{trunc_lemma}
Let $C = \sqrt{3}$, the diameter of the unit box in three dimensions,
and let the parameter defining $\widehat{W}$ in \eqref{windowedkernelft}
be given by $\tilde{C} = C + b \sigma$. 
Then, for sources and targets in the unit box 
$[-1/2,1/2]^3$, we have
\be\label{farpart}
u_{i}^{\rm far}=\sum_{j=1}^{N}M(|\x_i-\y_j|)\rho_j
\approx \sum_{j=1}^{N}W(\x_i-\y_j)\rho_j,
\ee
with a relative error of the order $\erfc(b)$. (Setting $b=6$ yields better than
fifteen digits of accuracy.) 
\end{lemma}

\begin{proof}
From \eqref{windowedkernelft} and the duality property 
 in
\cref{sec:ft}, we know that the kernel $W(\x) = W(|\x|)$ is a smooth radial
function.  We now show that it is given explicitly by
\be\label{windowedkernel2}
W(r)=
\frac{\erf(r/\sigma)}{r}- \frac{1}{2} 
\left(\frac{\erf((b\sigma+C+r)/\sigma)}{r}-\frac{\erf((b\sigma+C-r)/\sigma)}{r} \right).
\ee
For this, we combine
\eqref{fourierinvdef}
and \eqref{rbfft3d} to obtain
the Fourier transform relation
\be\label{inverserbfft}
F(r)=\frac{1}{2\pi^2}\int_0^\infty \frac{\sin(kr)}{kr}\hat{F}(k)k^2dk
\ee
for any radially symmetric function $F$.
Thus, 
\[
\ba
W(r)&=\frac{1}{2\pi^2}\int_0^\infty \frac{\sin(kr)}{kr}\hat{W}(k)k^2dk\\
&=\frac{4}{\pi}\int_0^\infty \frac{\sin(kr)}{kr}
\sin^2((b\sigma+C)k/2) e^{-k^2 \sigma^2/4} \, dk\\
&=\frac{2}{\pi}\int_0^\infty \frac{\sin(kr)}{kr}
(1-\cos(k(b\sigma+C))) e^{-k^2\sigma^2/4} \, dk\\
&=\frac{2}{\pi}\int_0^\infty \frac{e^{-k^2 \sigma^2/4}}{kr}
(\sin(kr)-\frac{1}{2}\sin(k(b\sigma+C+r))+\frac{1}{2}\sin(k(b\sigma+C-r)))dk.
\ea
\]
Applying \cref{inverserbfft} to the mollified kernel $M(r)$, 
and using \cref{erfoverrft}, it is straightforward to see that
\be
\erf(r/\sigma) = \frac{2}{\pi}\int_0^\infty 
\frac{e^{-k^2 \sigma^2/4}}{k}\sin(kr)dk,
\ee
and the result \eqref{windowedkernel2} follows. 
Combining \cref{ewaldsplitting} and \cref{windowedkernel2}, we have
\be
\ba
\left|W(r)-M(r)\right| &= 
\frac{\erf((b\sigma+C+r)/\sigma)-\erf((b\sigma+C-r)/\sigma)}{2r}\\
&\le \erfc(b) \frac{1}{r}, \quad r\le C,
\ea
\ee
as desired.
\end{proof}

Although we have an explicit expression for $W(\x)$, 
the calculation of \eqref{farpart} is most rapidly carried out using
Fourier convolution. 
Given $\sigma >0$ and $\epsilon < 1$, assuming
\be 
\label{Knodesdef}
K_{\rm max}^2 \sigma^2/4 \geq \log(1/\epsilon),
\ee
 we have
\be
u_{i}^{\rm far} = 
\frac{1}{(2\pi)^3}\int_{-K_{\rm max}}^{K_{\rm max}} e^{i\bk\cdot \x_i}  \widehat{W}(\bk) \,\widehat{g}(\bk) 
\, d\bk \ 
+ O(\epsilon),
\ee
where
\be\label{rhs_trans}
\widehat{g}(\bk)= \sum_{j=1}^{N} e^{-i \bk \cdot y_j} \rho_j.
\ee

It is easy to see that
the number of oscillations of the integrand 
in each dimension is at most $O((C+\tilde{C})K_{\rm max}/(2\pi))
= O(\sqrt{ \log(1/\epsilon)}/\sigma)$.
Thus, once the total number of quadrature nodes  $N_F = (2n+1)^3$, 
with $n \approx K_{\rm max}$ (Nyquist sampling), the error in
\be\label{truncatedkernelfourierrep}
u_{i}^{\rm far} 
\approx \sum_{\bm \in [-n,\dots,n]^d} \frac{1}{\pi^2} 
\left(\frac{\sin(\tilde{C}|\bk_{\bm}|/2)}
{|\bk_{\bm}|}\right)^2
e^{-|\bk_{\bm}|^2 \sigma^2/4} e^{i\bk_{\bm} \cdot \x_i} \,\widehat{g}(\bk_{\bm})
\ee
decays superalgebraically.
Here, $\bk_{\bm} = \frac{K_{\rm max}}{n} \bm = \frac{K_{\rm max}}{n} (m_1,m_2,m_3)$.

It is convenient to write the far field contribution in the form
\be\label{farpart2}
u_{i}^{\rm far}=\sum_{\bm \in [-n,\dots,n]^d} 
w_{\bm} e^{i\bk_{\bm} \cdot \x_i}  \widehat{g}(\bk_{\bm}), \quad i=1,\ldots,N,
\ee
where
\be\label{farpart2form}
w_{\bm}  =  \frac{1}{\pi^2} \left(\frac{\sin(\tilde{C}|\bk_{\bm}|/2)}{|\bk_{\bm}|}\right)^2
e^{-|\bk_{\bm}|^2\sigma^2/4}.
\ee
We may rewrite \cref{farpart2} more explicitly in terms of matrix-vector products 
as follows (omitting the details of unrolling the multi-index $\bm$ 
to a single index $m$ that ranges from 1 to $N_F$):
\be
\label{farpart3}
u^{\rm far}= \cf_2 \, \cd \, \cf_1 \, \rho, 
\ee
where
\be
u^{\rm far}=\begin{bmatrix} u_{1}^{\rm far} \ldots u_{N}^{\rm far}\end{bmatrix}^T, \quad
\rho=\begin{bmatrix} \rho_1 \ldots \rho_N\end{bmatrix}^T,
\ee
$\cf_1$ is an $N_F\times N$ matrix with entries $\cf_{1}(m,j)=e^{-i\bk_{\bm}\cdot \y_j}$,
$\cd$ is an $N_F\times N_F$ diagonal matrix with entries
$\cd(m,m)=w_{\bm}$,
and $\cf_2$ is an $N\times N_F$ matrix with entries $\cf_{2}(i,m)=e^{i\bk_{\bm}\cdot \x_i}$.

The application of $\cf_1$ or $\cf_2$ to a vector requires
only $O(N_F\log N_F + N \log^3(1/\epsilon))$ work using the 
type-1 and type-2 nonuniform fast Fourier transform (NUFFT) 
\cite{finufft,finufftlib,nufft2,nufft3,nufft6,potts2004},
where $\epsilon$ is the desired precision. The application
of $\cd$ clearly requires $O(N_F)$ work.

Let us now assume that we have divided
the computational domain into a grid with $N_B= 8^L$ 
cells (boxes), each of side length $\cR_L=2^{-L}$. (In $d$ dimensions,
there are $2^{dL}$ boxes after $L$ levels of uniform refinement, with each
refinement by a factor of 2.)
Assuming that the particles are uniformly distributed, the 
number of particles in each box is approximately the average value $n_s=N/(8^L)$. 

We now choose $\sigma$ so that, when computing the local part,
\[
u_{i}^{\rm local} =
\sum_{j=1}^{N}R(|\x_i-\y_j|)\rho_j
\]
for a target $\x_i$ in a box $B$ with side length $\cR_L$, the interactions
are negligible beyond the $3^3$ nearest neighbors 
in the uniform grid, including itself.
That is to say, we would like to enforce that
$\erfc(r/\sigma) \le \epsilon$
for $r \geq \cR_L$. A sufficient condition is that
\be\label{femdeltavalue}
\sigma \approx \frac{\cR_L}{\sqrt{\log(1/\epsilon)}}.
\ee
From the discussion above, the number of Fourier modes needed for the 
far field computation is 
\be
N_F = O(K_{\rm max}^3) =  O \left( \left( \frac{\log(1/\epsilon)}{\cR_L} \right)^3 \right)
= O\left(\frac{N}{n_s} \log^3(1/\epsilon)\right).
\ee
With this choice of $\sigma$,
the cost of evaluating all local interactions is bounded
by $3^3 n_s N$ evaluations of the residual kernel $R(|\x_i-\y_j|)$.
In summary, the local interactions require of the order
$O(N n_s)$ work and the far field interactions require of the order 
$O( N\log \frac{N}{n_s})$ work. For $n_s = O(1)$, this
leads to an $O(N \log N)$ method.
The optimal choice of $n_s$, of course, is the value 
that balances the far field and local work,
which depends on the precision $\epsilon$, the precise value of $N_F$
and the cost of residual kernel evaluations.
Finally, we should note that
we have used a grid with $8^{L}$ boxes for convenience only; any uniform grid
on which the FFT is efficient can be used.

\subsection{An $O(N \log N)$ \acron method}\label{sec:afem}

We turn now to the development of an adaptive, hierarchical
extension of Ewald's method, which will also require $O(N \log N)$ work, but does
not rely on the FFT for its asymptotic complexity. 
Recall that the problem with highly nonuniform particle distributions
is that we cannot bound the number of particles that lie
in any single box $B$ in a uniform grid with $O(N)$ boxes. As a result, 
we cannot control the cost of the local interactions. 
Creating a uniform grid with smaller boxes 
would require that $N_F \gg N$, increasing the cost of the far-field 
interactions. Thus, an adaptive data structure is essential, and we
shall rely here on a level-restricted oct-tree in three dimensions
(and on a level-restricted quad-tree in two dimensions) 
\cite{biros2008sisc}. The tree construction
begins with the unit box in dimension $d$ 
and sorts all sources and targets into its $2^d$ children, obtained by bisecting
the box in each coordinate direction. The refinement process is continued
recursively until there are fewer than $n_s$ sources and targets in any 
leaf node, with the usual convention of referring to the box that is subdivided
as the {\em parent} and the $2^d$ subdivisions as its children. 
The free parameter $n_s$ is typically chosen in a precision-dependent manner,
discussed in more detail below.
The requirement of level-restriction is that no two boxes that
share a boundary point in the data structure can be more than one refinement 
level apart, and there are standard algorithms to enforce such a condition
\cite{biros2008sisc}.
In a level-restricted tree, for any box $B$ at level $l$, 
we will refer to the boxes at the same refinement level which share a
boundary point as {\em colleagues}, including
$B$ itself. 
Boxes which share a boundary point with box $B$ at level $l$ but are themselves
leaf nodes at level $l-1$ will be referred to as {\em coarse neighbors}
(see \cref{boxdefs}).
The set of coarse neighbors will be denoted by ${\cal N}(B)$.
We will refer to the unit box itself as level $0$ and to the 
maximum refinement level as $L_{\rm max}$. The linear dimension of a box at level
$l$ will be denoted by $\cR_l = 2^{-l}$.

\begin{figure}[!ht]
\centering
   \includegraphics[width=2.5in]{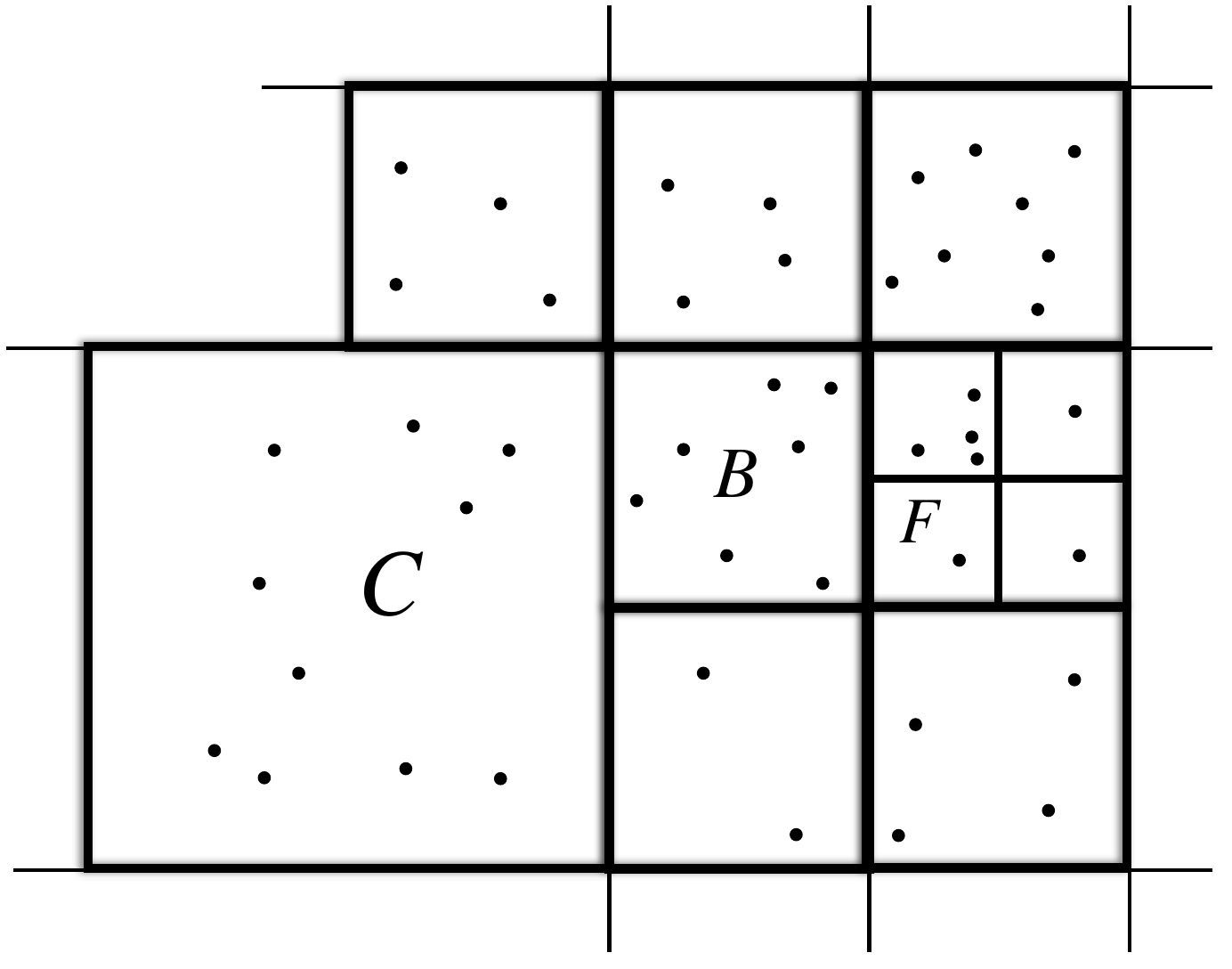}
   \caption{\sf 
A portion of a level-restricted tree is depicted (in two dimensions).
For a leaf node $B$, the boxes at the same level are called {\em colleagues},
while $F$ is a fine neighbor and $C$ is a coarse neighbor.
}
\label{boxdefs}
\end{figure}

Let us now consider the following 
``telescoping" decomposition of the Laplace kernel {for each level}
$L\in \{ 0,\dots,L_{\rm max} \}$:
\be\label{mlks}
\frac{1}{r} = W_0(r)+\sum_{l=0}^{L-1}D_{l}(r)+R_L(r), \quad 
L=0,\ldots,L_{\rm max},
\ee
where 
the windowed kernel $W_0$, the difference kernels $D_l$, and the 
residual kernels $R_L$
are defined by the formulas
\be\label{kerneldef}
\ba
W_0(r)&\approx M_0(r) = \frac{\erf(r/\sigma_0)}{r}, \\
D_l(r)&=\frac{\erf(r/\sigma_{l+1})-\erf(r/\sigma_{l})}{r},\\
R_L(r)&=\frac{\erfc(r/\sigma_L)}{r},
\ea
\ee
for $\sigma_0 > \sigma_1 > \ldots > \sigma_{L_{\rm max}}$.
It is easy to see that \cref{mlks} is indeed an equality at every level
$L = 0,\ldots,L_{\rm max}$. We
define the mollified far-field kernel at level $L\geq 1$ by
\be\label{windowedkernelL}
M_L(r) = \frac{\erf(r/\sigma_L)}{r} =
M_0(r)+\sum_{l=0}^{L-1}D_{l}(|\x|).
\ee
Far field interactions beyond 
the colleagues at level $L$ 
are accounted for in terms of the kernel $M_L(|\x|)$
to the desired precision, and the remaining
corrections involve the residual kernel $R_L(|\x|)$ only within colleagues
and coarse neighbors. (At level $L=1$, of course, the colleagues
cover the entire computational
domain.)

The departure from the classical Ewald approach
\eqref{ewaldsplitting},
is the telescoping decomposition of the free-space Green's function 
and the introduction of the more and more localized
difference kernels $D_l$ as we proceed from level to level in a tree hierarchy 
(hence the nomenclature ``kernel-splitting").
Unlike previous multilevel summation methods 
\cite{brandt1990jcp,brandt1998,skeel2002,hardy2009pc,multilevel_summation_2015,
multilevel_summation_bspline,tensor_multilevel_ewald},
we systematically rely on Fourier analysis to compute convolutions at each scale. 
As for the values  of $\sigma_l$ ($l=0,\ldots,L_{\rm max}$), they are chosen so that
at every scale, the residual kernel $R_l$ is supported within nearest neighbors at the
same level. That is,
we require that $\erfc(r/\sigma_l) \le \epsilon$
for $r\ge \cR_l=2^{-l}$. This is accomplished by letting
\be\label{deltavalue}
\sigma_0 \approx \frac{\cR_0}{\sqrt{\log(1/\epsilon)}},\quad
\sigma_l = \sigma_{l-1}/2. 
\ee
The full $O(N \log N)$ scheme is illustrated in \cref{fig:dmk_adap}.
Before presenting the algorithm in detail, several observations are in order.

\begin{figure}[!ht]
\centering
   \includegraphics[width=5in]{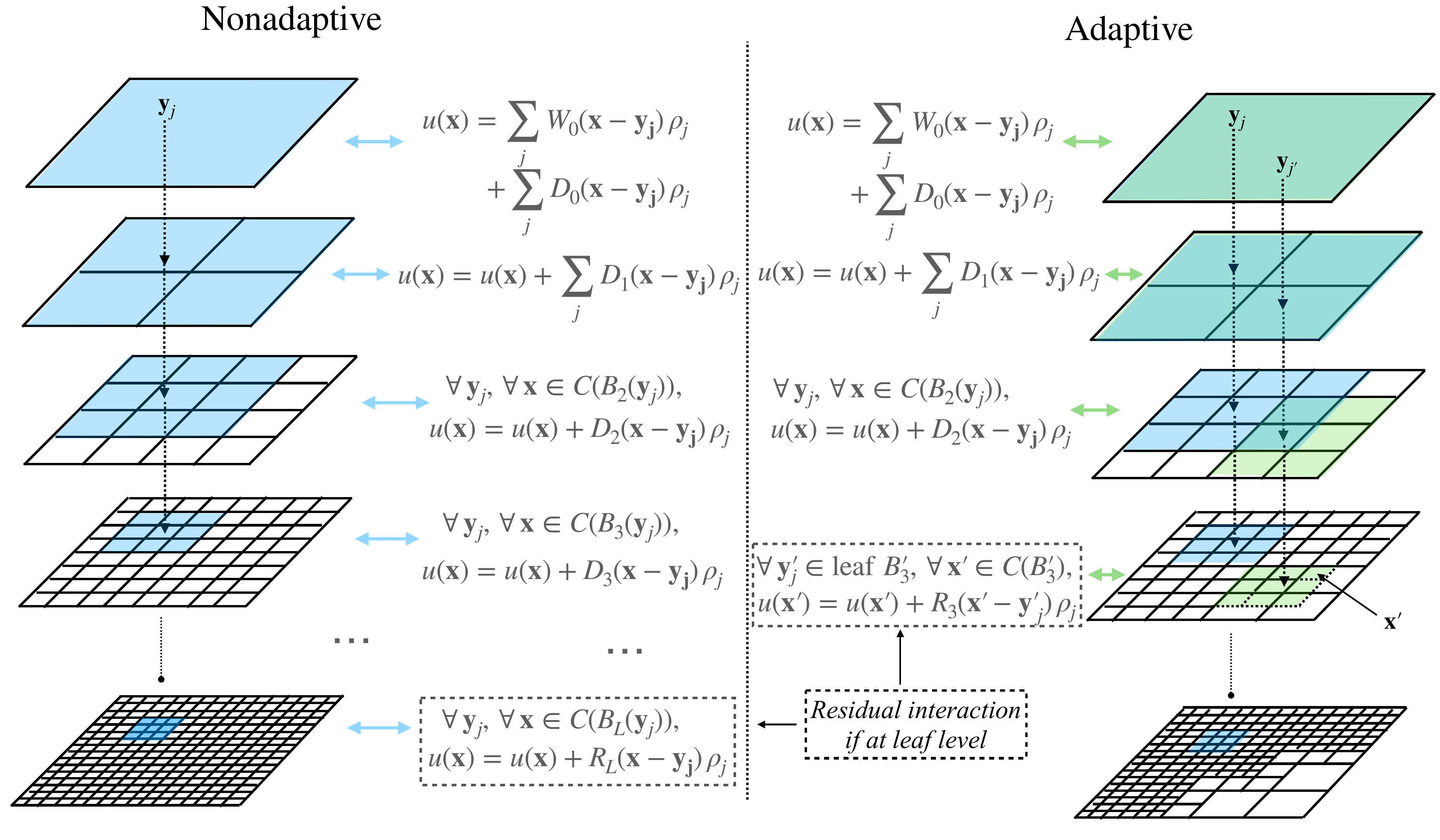}
   \caption{\sf 
The introduction of the telescoping series
\eqref{mlks} allows for a multi-level version of Ewald summation,
illustrated in two dimensions.
Convolution at level 0 with $W_0$ is 
is global, but requires only a short Fourier transform, whose
length is independent of the number of particles.
At successively finer levels, convolution with the difference kernels $D_l$
is carried out for colleagues (nearest neighbors) only. 
In the nonadaptive case (left), each source $\y_j$ lies in some leaf box $B$ 
at the finest level, and direct interactions are computed with the residual 
kernel $R_L$ for targets within $B$'s colleagues. 
In the adaptive case (right), some sources such as ${\bf y}'$ 
may lie in a leaf box at some coarser
level $l$ (here, $l=3$). That leaf box is denoted by $B_3'$ in the figure.
Direct interactions are computed using the
residual kernel $R_{3}$ for the targets within $B_3'$ and its colleagues.
Aside from some care in bookkeeping, implementation of the adaptive
scheme is straightforward.
}
\label{fig:dmk_adap}
\end{figure}

\begin{figure}[!ht]
\centering
\includegraphics[height=50mm]{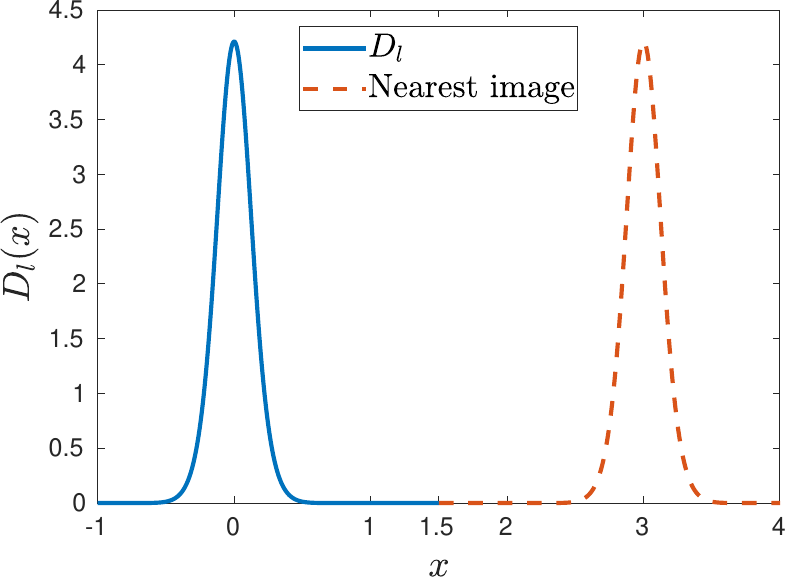}
\caption{\sf The difference kernel $D_l$ is compactly
  supported on a ball of radius $\cR_l$ ($=1$ in the figure) to the desired precision, 
  while its Fourier spectral approximation
  needs to be valid on a ball of radius $2\cR_l$. Since the actual support of $D_l$ is smaller
than the full interval, its periodic
  images only need to be $\frac{3}{2} \, \cR_l$ away in order to 
  avoid the aliasing error which controls the accuracy of its Fourier spectral approximation.
  This accounts for the factor of $\frac{3}{2}$ in \cref{fourierspacing}.
}
\label{diffkernelaliaserror}
\end{figure}

\begin{enumerate}[label=(\alph*).]
\item The windowed kernel $W_0(\x)$ has the same behavior 
as in the classical Ewald approach, with a smooth and 
rapidly decaying Fourier transform.
Here, however, $\sigma_0$ is so large that the number of quadrature nodes 
$N_F$ in the Fourier transform
is independent of $N$ (the number of particles) and is simply 
a function of the precision with
\[ N_F \approx (4 \log(1/\epsilon))^d  \]
from \eqref{Knodesdef}.
\item For a particle in a box  $B$ at level $l$,
the difference kernel $D_l(\x) = D_l(|\x|)$ is negligible
(smaller than the requested precision $\epsilon$) beyond $B$'s colleagues. 
It is a smooth function of $\x\in\mathbb{R}^3$
and its Fourier transform is given by the formula
\be\label{differencekernelft}
\hat{D}_l(\bk) =\left\{\begin{aligned}& 4\pi
\frac{e^{-|\bk|^2 \sigma_{l+1}^2/4}-e^{-|\bk|^2 \sigma_{l}^2/4}}{|\bk|^2},& \bk\ne (0,0,0)\\
&  \pi(\sigma_l^2-\sigma_{l+1}^2),
 & \bk=(0,0,0)\end{aligned}\right.
\ee
As above, the combination of smoothness and rapid decay implies that 
the trapezoidal-rule approximation to the Fourier integral 
converges at an exponentially rate.
\item
When proceeding from level $l$ to level $l+1$, $\sigma$ decreases by a
factor of two, according to \eqref{deltavalue}, so that
the range of integration in the corresponding Fourier integral
doubles, with
\be\label{fouriercutoff}
K_l \approx \frac{4}{\cR_l}\log\left(\frac{1}{\epsilon}\right)
\ee
to maintain the same level of accuracy.
However,
the distance between relevant sources and targets
shrinks by a factor of two at the same time, so that the number of oscillations,
{\em and the number of quadrature nodes in the range $[-K_l,K_l]^d$},
remains constant. %This is illustrated in \cref{fig:parascale}.
($K_l$ can actually be set slightly smaller than 
the value in \eqref{fouriercutoff} because of the extra factor $1/|\bk|^2$ in
\eqref{differencekernelft}.)
\item In the multilevel \acron algorithm, we use localized
Fourier convolution to compute the influence of the difference kernels
due to sources that lie within every box $B$ at level $l$ with targets that lie within $B$
or its colleagues. As a result, 
our trapezoidal quadrature must be accurate in a ball of radius
$2\cR_l = \cR_{l-1}$.
From the Poisson summation formula
\eqref{poissonsummation},
it can be shown \cite{trefethentrap} that
\be\label{fourierspacing}
\frac{2\pi}{h_l} \ge \frac{3}{2} \cR_l \Rightarrow h_l\le \frac{4\pi}{3 \cR_l}
\ee
is sufficient to achieve a precision $\epsilon$ for $|\x|\le 2\cR_l$,
so that we have
\be\label{differencekernelft2}
D_{l}(\x) \approx \frac{1}{(2\pi)^3} 
\sum_{\bk \in [-n_f,\dots,n_f]^d}
e^{ih_l\bk\cdot \x} \hat{D}_l(h_l\bk),
\ee
where 
\be
n_f = \frac{K_l}{h_l} = \frac{3}{\pi}\log\left(\frac{1}{\epsilon}\right).
\ee
(Similar, but less sharp, estimates can be derived from the Euler-Maclaurin formula.)
Using the same notation as for the classical Ewald approach, we will write
\be\label{differencekernelft2a}
D_{l}(\x)  \approx
\sum_{\bm \in [-n_f,\dots,n_f]^d} w_{l,\bm} e^{i\bk_{\bm} \cdot \x},
\ee
where $\bk_{\bm} = h_l (m_1,m_2,m_3)$ and
\be
\label{wlmdef}
w_{l,\bm} =\frac{1}{(2\pi)^3} \hat{D}_l(\bk_{\bm}).
\ee 
The total number of Fourier modes needed for the difference kernels is
\be\label{diffkernelftterms}
N_F=(2n_f+1)^3\approx \left(\frac{6}{\pi}\log\left(\frac{1}{\epsilon}\right)\right)^3,
\ee
independent of the level.
\item The residual kernel at level $L$, $R_L(|\x|)$, 
is compactly supported to precision $\epsilon$ 
within a ball of radius $\cR_L$, the boxsize at level $L$.
\end{enumerate}

\begin{figure}[!ht]
\centering
\includegraphics[height=36mm]{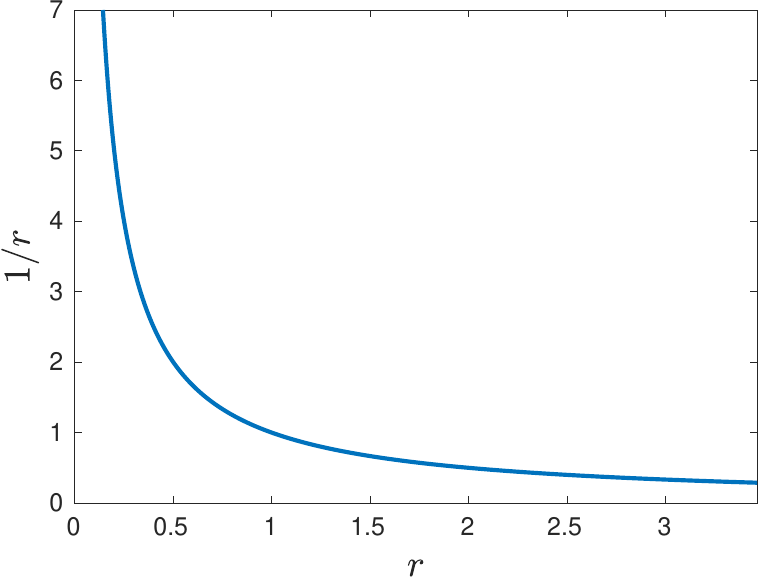}\hspace{4mm}
\includegraphics[height=36mm]{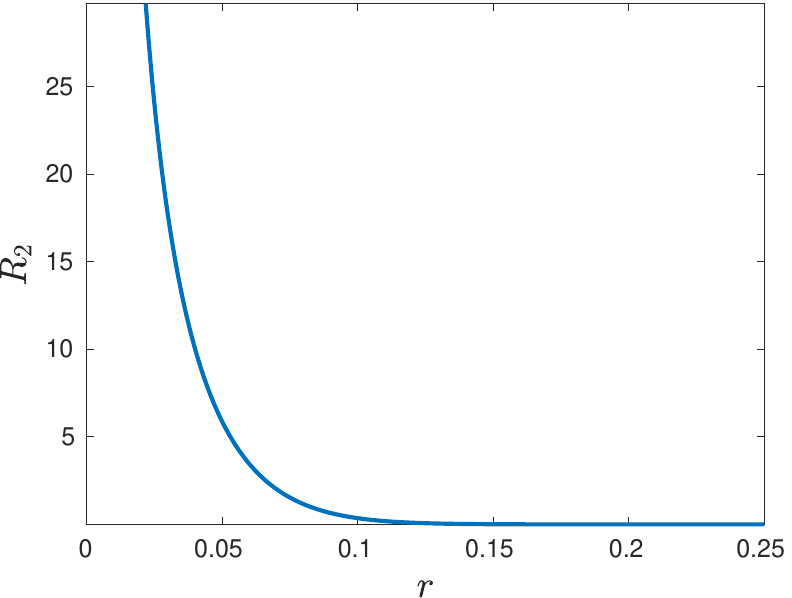}

\vspace{4mm}

\includegraphics[height=36mm]{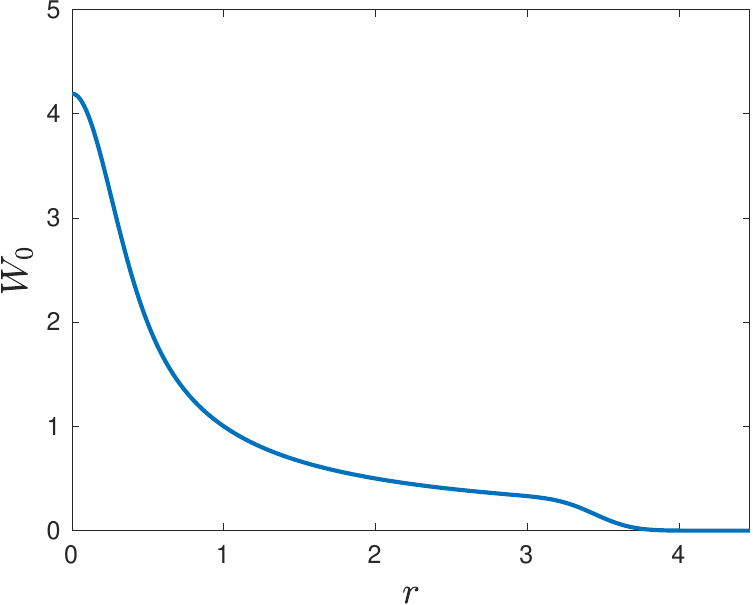}\hspace{4mm}
\includegraphics[height=36mm]{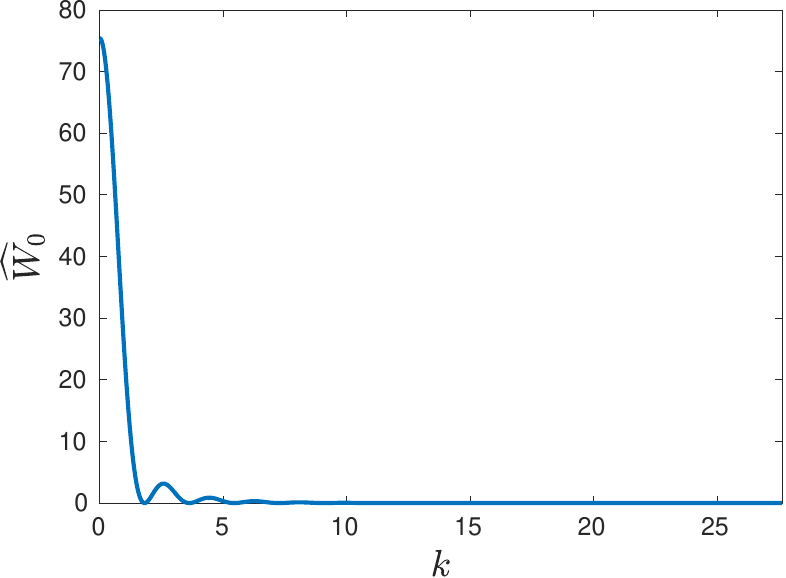}

\vspace{4mm}

\includegraphics[height=36mm]{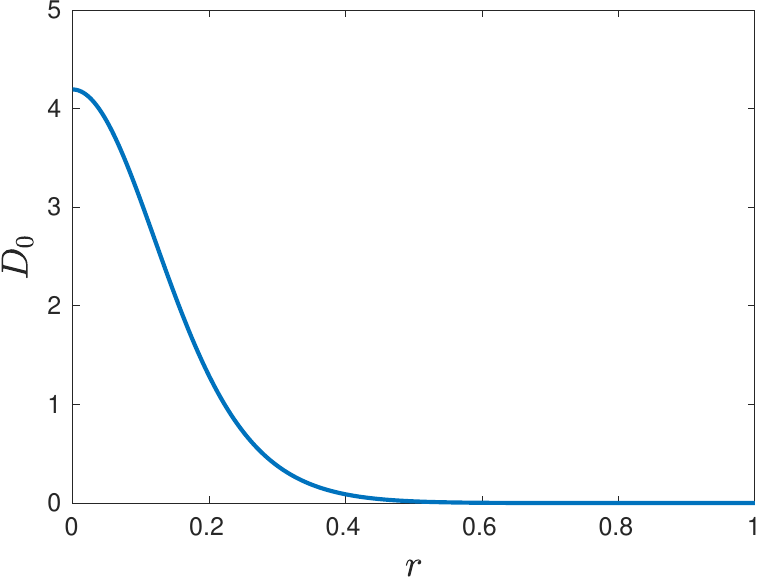}\hspace{4mm}
\includegraphics[height=36mm]{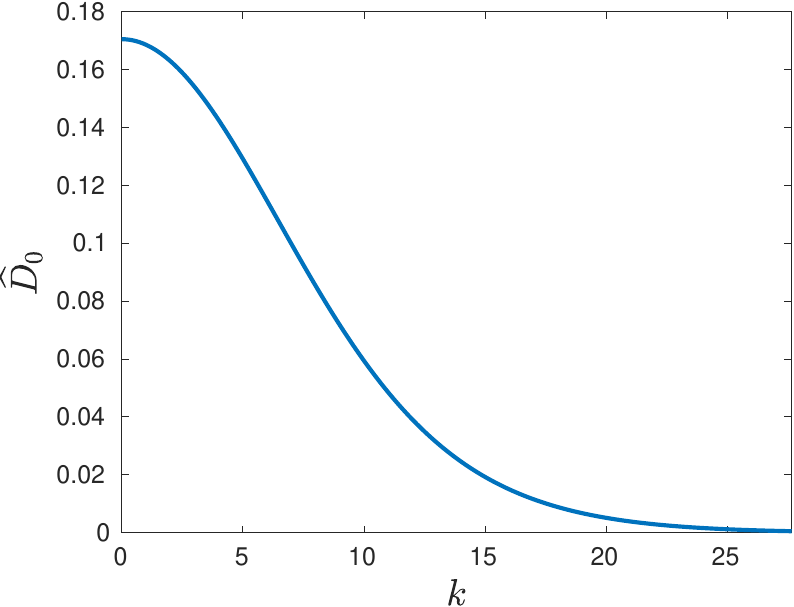}

\vspace{4mm}

\includegraphics[height=36mm]{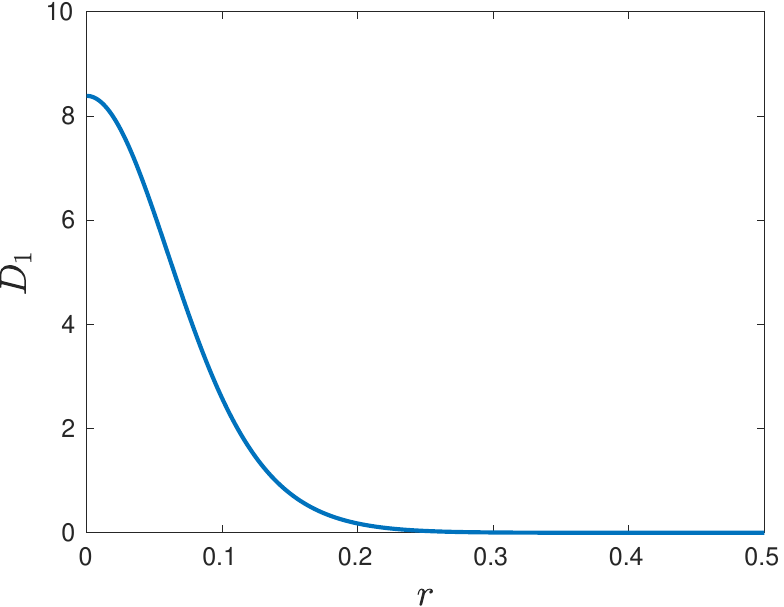}\hspace{4mm}
\includegraphics[height=36mm]{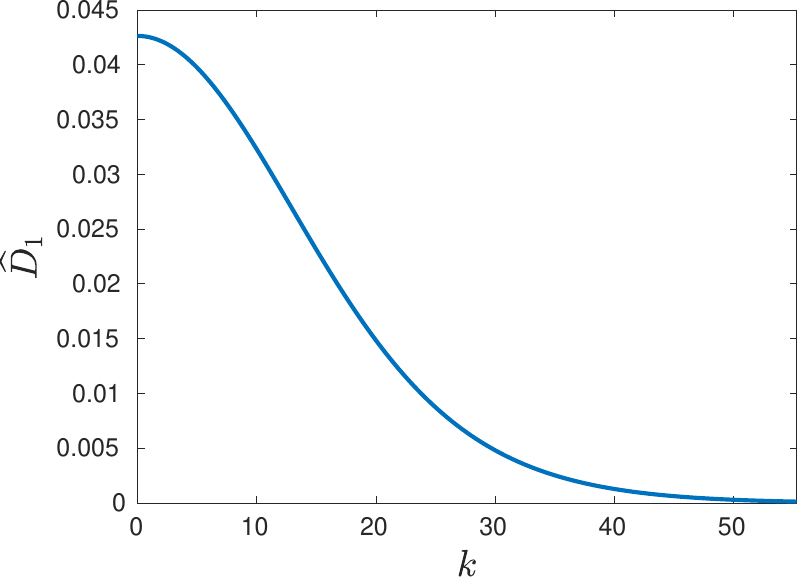}
\caption{\sf The dual-space splitting of the $1/r$ kernel 
  using Gaussians on a leaf box at level $2$
  with six digits of accuracy:
  $1/r=W_0(r)+D_0(r)+D_1(r)+R_2(r)$ for $r\le \sqrt{3}$.
  Top row: left - the original kernel $1/r$; right - the residual kernel $R_2(r)$
  that is numerically supported on $[0,1/4]$. Second row: left -
  the windowed kernel $W_0(r)$;
  right - its Fourier transform. Third row: left - the difference kernel $D_0(r)$ that
  is numerically supported on $[0, 1]$; right - its
  Fourier transform. Fourth row: left - the difference kernel $D_1(r)$ that is
  numerically 
  supported on $[0, 1/2]$; right - its Fourier
  transform. Note that the difference kernels $D_1$ is exactly a rescaled version
  of $D_1$, with one half the support and twice the frequency content. }
\label{l3dkernelsplit}
\end{figure}

Suppose now that the target point $\x_i$ (not coincident with any source)
lies in some leaf box $B_{L}$ at level $L$,
with $B_L \subset B_{L-1} ... \subset B_0$ denoting the hierarchy of parents at
coarser levels, where $B_0$ is the unit box at level 0.
Using \cref{mlks} and \cref{laplacesum}, we have
\be\label{laplacesum3}
\ba
u(\x_i)&= \qquad
\underbrace{u^{\rm far}(\x_i)}\qquad\quad +\qquad
\underbrace{\sum_{l=1}^{L-1}u_{l}^{\rm diff}(\x_i)}\qquad +\qquad 
\underbrace{u_{L}^{\rm local}(\x_i)}\\
&=
\sum_{j=1}^{N}W_0(|\x_i-\y_j|)\rho_j + \sum_{l=0}^{L-1}\sum_{j=1}^{N}D_l(|\x_i-\y_j|)\rho_j
+\sum_{j=1}^{N}R_L(|\x_i-\y_j|)\rho_j.
\ea
\ee
Before describing a fast algorithm based on \cref{laplacesum3}, we introduce
some notation.

\begin{definition}
Let $\bc(B)$ denote the center of box $B$ at level $l$. The outgoing expansion
$\Phi_l(B)$ for box $B$ has $N_F$ elements, indexed by 
\be\label{outgoingexpansion}
[\Phi_{l}(B)]_{\bm}=\sum_{\y_j\in B} e^{-ih_l\bk_{\bm}\cdot (\y_j-\bc(B))} \, \rho_j, 
\quad {\bm} \in [-n_f,\dots,n_f]^3.
\ee
\end{definition}

\begin{lemma} 
Suppose that $B$ is a box containing a target $\x_i$ at level $l$.
The {\em incoming expansion} $\Psi_{l}(B)$ has $N_F$ elements, indexed by 
\be\label{diagonaltranslation}
[\Psi_{l}(B)]_{\bm}=\sum_{S\in C(B)}w_{l,{\bm}}e^{ih_l\bk_{\bm}\cdot (\bc(B)-\bc(S))}
\Phi_{l}(S)[\bm],
\quad \bm \in [-n_f,\dots,n_f]^3,
\ee
where $w_{l,\bm}$ is given by \eqref{wlmdef}.
Then
\be\label{incomingexpansion}
u_{l}^{\rm diff}(\x_i)=
\sum_{\bm \in [-n_f,\dots,n_f]^3} 
e^{ih_l\bk_{\bm} \cdot (\x_i-\bc(B))}[\Psi_{l}(B)]_{\bm}. 
\ee
\end{lemma}

This result follows immediately from 
\eqref{differencekernelft2a} and the fact that translation of an expansion in
exponentials is in diagonal form. 

\begin{definition}
The colleagues of a box $B$ at level $l$ will be denote by 
$C(B)$ and the number of boxes in $C(B)$ will be denoted by
$N_C(B)$. ($N_C(B)$ is clearly bounded by $3^d$.) 
\end{definition}

\vspace{.2in}
\hrule
\vspace{.2in}

\begin{center}
{\bf Informal description of the $O(N \log N)$ algorithm}
\end{center}

\vspace{.1in}

\begin{enumerate}[label=(\roman*).]
%{\bf Step 1:}
\item
Sort sources and targets on an adaptive level-restricted tree.
Let $L_{\rm max}$ denoted the finest level.

\item
Initialize the field $u(\x_i)$ by computing
the first term $u^{\rm far}$ 
in \cref{laplacesum3} directly, as in 
\cref{farpart3}, except that $N_F$ is given by 
\cref{diffkernelftterms}, and independent of $N$.
[{\em This requires $O(N N_F)$ work.
Using the type-1 and type-2 NUFFTs, this can be reduced to
$O(N_F\log N_F + N \log^d(1/\epsilon))$ work, as in the discussion of the classical
Ewald method above.}]

%{\bf Step 2:}
\item  For every level $l=1,\dots,L_{\rm max}$,
compute the outgoing expansions $\Phi_{l}(B)$ for
  each {\em non-leaf} box $B$ containing sources at that level. [{\em Using direct methods,
this again requires $O(NN_F)$ work. With the type-1 NUFFT, the cost can be reduced to 
$O(N + N_F \log N_F)$ work with a precision-dependent constant.}]
 
\item  For every level $l=1,\dots,L_{\rm max}$,
for each {\em non-leaf} box $B$ at that level, convert the 
outgoing expansion to an incoming expansion in each of its colleagues $C$ that contains
targets and add to $\Psi_{l}(B)$. 
[{\em This requires at most $O(3^d N_F)$ work per box.}]

\item  For every level $l=1,\dots,L_{\rm max}$,
evaluate the sum \eqref{incomingexpansion}
for each target $\x_i$ if it lies in a box with an incoming expansion and add to 
field $u(\x_i)$.
[{\em Using direct methods, this again requires $O(NN_F)$ work. 
With the type-2 NUFFT, the cost can be reduced to 
$O(N + N_F \log N_F)$ work with a precision-dependent constant.}]

\item  For every level $l=1,\dots,L_{\rm max}$,
%{\bf Step 3:}
for every source $\y_j$, which is contained in a leaf box $B$ at level $l$, 
for each target $\x_i$ in a colleague $C(B)$ or coarse neighbor ${\cal N}(B)$, 
increment the potential with the local contribution:
\be
\ba
u(\x_i) = u(\x_i) +
R_l(|\x_i-\y_j|)\rho_j  \, .
\ea
\ee
\end{enumerate}

\vspace{.2in}
\hrule
\vspace{.2in}

In the adaptive setting, counting direct interactions requires
some care, but it can be shown that the total number of direct interactions
is of the order $O(N)$ 
(see \cite{carrier1988sisc,nabors1994sisc} for more details).
In fact, uniform distributions are typically 
the setting where the direct cost is the greatest.
It is worth noting that
additional savings can be had by observing that 
the residual kernel $R_l$ is negligible beyond a ball of radius $\cR_l$. 
Thus, in the simplest case where there are approximately $n_s$ 
targets in each of $3^d$ colleagues, the relevant targets for
each source $\y_j$ are only those contained in the sphere of radius $\cR_l$
centered at $\y_j$.
Thus, the average number of targets
with a nonzero contribution from a given source is approximately
$\frac{4\pi}{3} n_s$ rather than $27n_s$.
Finally, we recall that if a target is coincident with one of the sources,
it should be omitted in the direct calculation and the 
self-interaction correction
$u_{\rm self}(\x_i)$ should be added to the potential $u(\x_i)$.

Without carrying out a detailed operation count, assuming $L_{\rm max} = O(\log N)$,
this adaptive. multilevel extension of Ewald summation requires 
$O(N)$ work at each of $O(\log N)$ levels. 

\begin{remark} \label{kersplit}
Kernel-splitting by itself is not new. It is the basis for the multilevel
summation methods in  
\cite{brandt1990jcp,brandt1998,hardy2009pc,multilevel_summation_2015,
multilevel_summation_bspline,skeel2002}.
However, those methods rely
on the evaluation of the smooth, difference kernel interactions 
through low-rank approximations  in physical space.
(Of note is the recent paper 
\cite{tensor_multilevel_ewald} which uses a sum-of-Gaussian approximation of the 
difference kernel and direct convolution in physical space using
separation of variables for each individual Gaussian.) 
\acron is a {\em dual-space} method, like Ewald summation, making use of the 
Fourier transform to diagonalize these
interactions with a significant cost savings.
Diagonal translation plays a role
in some existing kernel-independent FMMs, such as \cite{fmm7}
and \cite{zhang2011jcp}, where Fourier-based convolution is used to account for
well-separated interactions.
\end{remark}

\subsection{The linear scaling full \acron algorithm}
\label{sec:afemmodifications}

In the scheme just outlined, we were obliged to carry out $O(N)$ operations at each of 
$\log N$ levels. In order to obtain an $O(N)$ scheme, we will need to develop a two-pass
version of the method (as in multilevel summation or the FMM). 
Instead of step (ii), we 
will begin at the level of leaf nodes and construct a set of {\em proxy} charges 
at all 
coarser levels that give rise to the same field as the original source distribution.
This will constitute an ``upward pass" (see \cref{uppassfig}). 
In the subsequent downward pass, rather than steps (iii) and (iv) above, where the difference
kernel contributions are evaluated at every level, we convert the 
incoming field expressed in terms of complex exponentials into 
an expansion in Chebyshev polynomials.
We then interpolate the Chebyshev expansion obtained in each parent box to its $2^d$
children and continue that process recursively until a leaf node is reached. Only then
is the Chebyshev expansion evaluated at the target locations (as in step (v) above).
In this manner, we access each source and target point only $O(1)$ times, 
independent of the depth of the tree.

\begin{figure}[!ht]
\centering
   \includegraphics[width=3.5in]{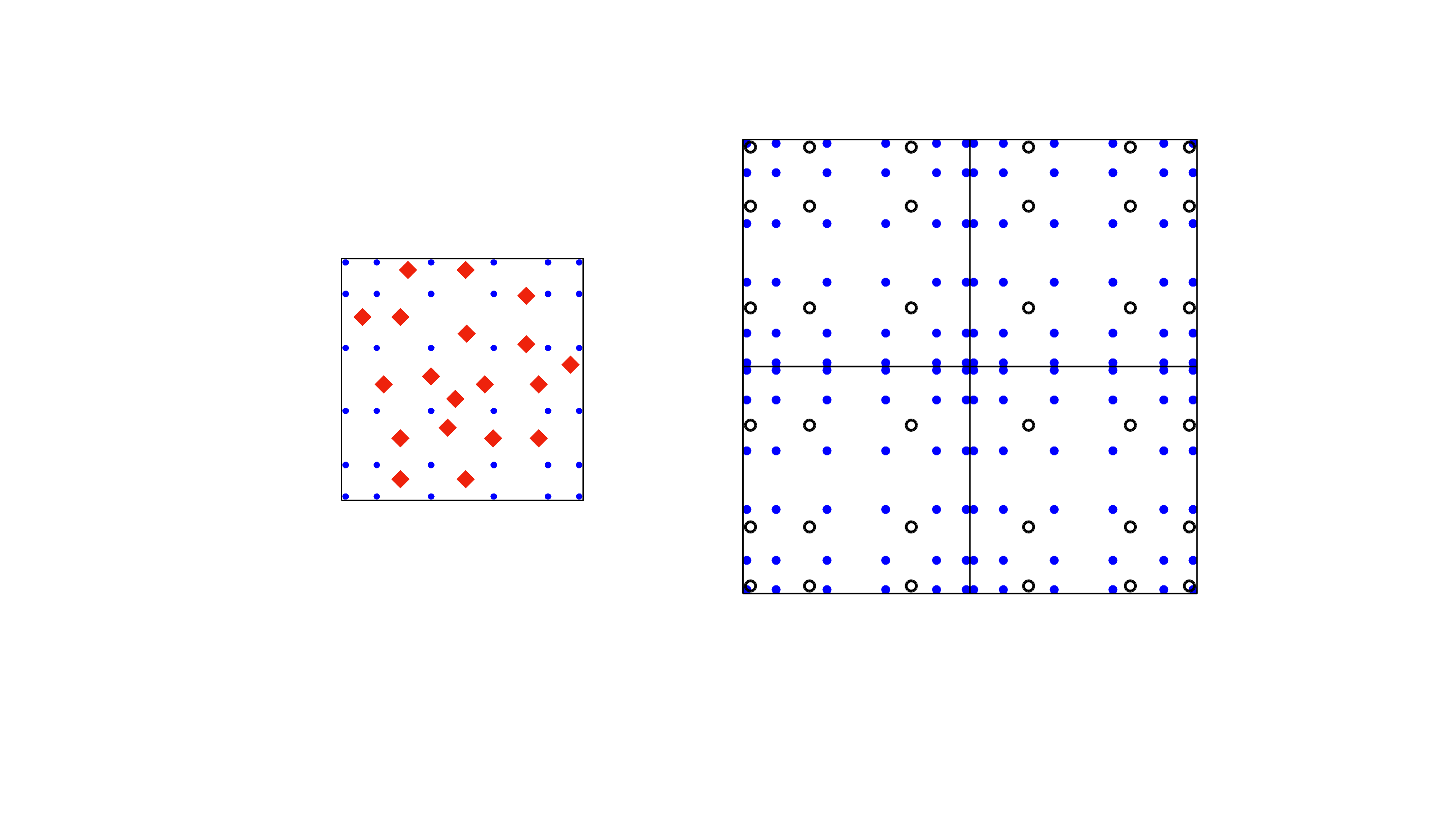}
   \caption{\sf 
In the $O(N)$ scheme, we construct a hierarchy of ``proxy"
charges to account for interactions at coarser and coarser scales.
(Left) At the level of leaf nodes, we first build a tensor product Chebyshev grid 
of proxy charges (solid blue circles) with fields equivalent to those induced by the 
original sources themselves (red diamonds), illustrated in two dimensions.
The cost of obtaining proxy charges from arbitrary charge locations 
requires $O(N p^d)$ work in $d$ dimensions, where $N$ is the number of sources.
(Right) In the upward pass, we translate the tensor product of proxy charges on each
child (solid blue circles) to a single tensor product of proxy charges at the level of 
the parent box (open black circles).
The parent proxy grid is responsible for approximating smoother interactions at its
own level. 
}
\label{uppassfig}
\end{figure}

Since the convolution kernels are smooth, we may use the anterpolation  approach 
described in \cref{sec:anterp},
illustrated in \cref{uppassfig}.
That is, in every leaf node, we replace the charges $\{ \rho_j\}$ with a 
tensor product grid of equivalent sources (``proxy charges") which give rise to the same smooth
potential over the range of interest.
At each coarser level, the $2^d$ tensor product grids of proxy charges 
are merged into a single tensor product grid, according to \cref{ctoplemma}.
This eliminates the $O(N \log N)$ cost incurred in the simpler scheme above, 
where some expansion is created from each source at each level.

\begin{definition}
The operator mapping the child proxy grid to the parent proxy grid will be 
denoted by $T_{\rm c2p}$. 
$T_{\rm c2p}$ is simply the (multi-index) anterpolation matrix $U^T$ where
$U$ is defined in \eqref{Udefnd}.
\end{definition}

Given the proxy charges on a tensor product grid, it is straightforward to
compute outgoing expansion for a box $B$ using a modified version of
\eqref{outgoingexpansion}:

\be\label{outgoingexpansion2}
[\Phi_{l}(B)]_{\bm}=\sum_{\bj \in [1,\dots,p]^d}
 e^{-ih_l\bk_{\bm}\cdot (\br_{\bj} - c(B))} \, \tilde \rho_{\bj}, 
\quad {\bm} \in [-n_f,\dots,n_f]^3,
\ee
where $\{ \br_{\bj} \}$ are the scaled tensor product Chebyshev nodes centered on $B$.
Using separation of variables, the cost is of the order
$O(p n_f^d)$.

We will also need to eliminate the $O(N \log N)$ cost incurred by evaluating the 
difference potential $u_{l,\rm diff}$ at each target at each level.
For this, once the incoming expansion has been obtained for a box $B$,
we compute the corresponding Chebyshev approximation $\Lambda_l(B)$:
\be
u_{l}^{\rm diff}(\x)=\sum_{\bm \in [1,\dots,p]^d}
\lambda_{l,\bm}T_{\bm-1}(\x)
\ee
with expansion coefficients 
\be\label{localexpansion}
\lambda_{l,\bm}=\sum_{\bn \in [1,\dots,p]^d} 
U(\bm,\bn) u_l(\tilde{\br}_{\bn}) \, .
\ee
Here, $\tilde{\br}_{\bn}$ are the shifted and scaled Chebyshev grid on the target box.
The values $u_l(\tilde{\br}_{\bn})$ are themselves obtained using separation of variables by 
evaluating
\be\label{incomingexpansion2}
u_{l}^{\rm diff}(\br_{\bn})=
\sum_{\bm \in [-n_f,\dots,n_f]^3} 
e^{ih_l\bk_{\bm} \cdot (\br_{\bn}-\bc(B))}[\Psi_{l}(B)]_{\bm}. 
\ee
The cost involved is of the order
$O( p n_f^d + p^2 n_f^{d-1} + .. p^d n_f)$. Since $n_f$ is generally larger than
$p$, we will refer to this cost as of the order $O(p n_f^d)$.

For non-leaf boxes, as in the FMM,
we do not evaluate the local expansions $\Lambda_l(B)$ directly.
Instead, we translate the local expansions from $B$ to its children.
The translation operator converting the local expansions from the
parent box to the child box is given by the formula
\be\label{p2ctranslation}
T_{\rm p2c}= V^{-1} E 
\ee
from \cref{ptoclemma} and \eqref{ptoc}.
It is easy to see that both $T_{\rm c2p}$ and $T_{\rm p2c}$
are independent of the level in the tree hierarchy and that there are 
$2^d$ such operators in $\mathbb{R}^d$.

For a target $\x$ in a leaf box $B$ at level $L$, we define the
{\em far-field potential} as the sum of the windowed kernel interactions and 
the difference kernel interactions up to that level: 
\be\label{farpotential}
u_{L}^{\rm far}(\x)=u^{\rm smooth}(\x)+\sum_{l=0}^{L-1}u_{l}^{\rm diff}(\x)
\ee
The far-field potential will be evaluated {\em once} for each
target via standard (Chebyshev) interpolation. 
The near-field correction is
\be\label{nearpotential}
\ba
u_{L}^{\rm local}(\x)&=
 \sum_{\y_j \in C(B)} R_L(|\x-\y_j|)\rho_j +
 \sum_{\y_j \in {\cal N}(B)} R_{L-1}(|\x-\y_j|)\rho_j\\
 &+
 \sum_{\y_j \in {\cal F}(B)} R_{L+1}(|\x-\y_j|)\rho_j,
\ea
\ee
where ${\cal N}(B)$ and ${\cal F}(B)$ are the coarse and fine neighbors of $B$,
respectively.

\begin{remark}
Note that the residual kernel index (in $R_L$, $R_{L-1}$ or $R_{L+1}$ above) is
tied to the level of refinement of the {\em source} box since the far field
contribution for that source may have terminated at a coarser level.
The union of $C(B)$, ${\cal N}(B)$, and ${\cal F}(B)$ is exactly the so-called
list 1 in the adaptive FMM (see, for example, \cite{cheng1999jcp}).
\end{remark}

\subsection{Reducing the cost of direct interactions} \label{localaccel}

In the \acron framework (as with FMMs and other linear scaling multilevel schemes),
the optimal tree depth is achieved when the cost of direct interactions, here with the 
residual kernel, balances the cost of far field interactions. Thus, accelerating
direct interactions has a significant effect on the overall performance.
As noted in the paragraph before \cref{kersplit}, we can reduce the number of 
targets that need to be considered by exploiting the fact that
the residual kernel $R_L$ is negligible beyond a ball of radius $\cR_L$ about any 
source at level $L$. As illustrated in the right-hand diagram
of \cref{boxdefs}, we can easily reduce the search for targets by refining 
every box {\em once}. 
Consider a source $\y_j$ that lies in the box $B$ in the 
left-hand diagram of \cref{boxdefs}. To see which targets are relevant, pairwise
distances would have to be computed throughout $B$'s colleagues and coarse neighbors.
With an additional level of refinement, however, suppose $\y_j$ is determined
to lie in the indicated child box $B_a$. We can then restrict the distance calculation 
to the colleagues and coarse neighbors of $B_a$ and only the subset of boxes
indicated in pale green. The remaining targets are all outside the relevant ball
of radius $\cR_L$ centered on $\y_j$.

For the case of a uniform distribution, note that
the number of boxes within the interaction range is $5^d$ at level $L+1$ rather
than $3^d$ boxes at level $l$. This reduces the number of distance calculations
per source from $3^d n_s$ to $5^d (n_s/2^d) = 2.5^d n_s$. 
Furthermore, the pale green boxes are well separated from $B_a$ and the residual 
kernel for targets in those boxes is smooth and easily approximated by a polynomial
of one variable in $r$ (or $r^2$). This avoids the need for evaluating special
functions (such as $\erf(r/\sigma)$) for these interactions.

\begin{figure}[!ht]
\centering
\includegraphics[width=4in]{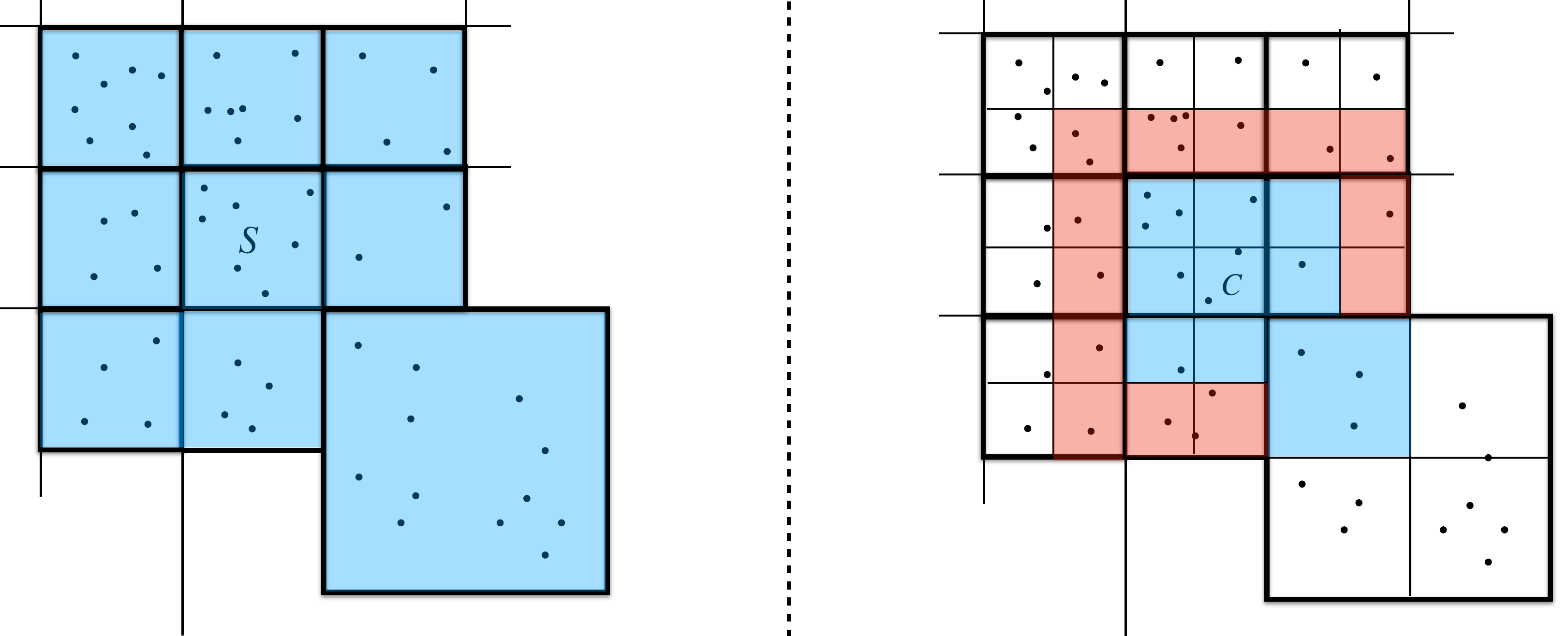}

\caption{\sf Interaction lists due to the residual kernel. Left: the original list $L_D$
  as in \cite{greengard2023arxiv}. The list $L_D$
  of a leaf box $S$ at level $l$ contains all target boxes that are its neighbors
  at level $l$ or leaf neighbors at level $l-1$. Right: the target boxes that
  are in the interaction
  range of a child box $C$ of $S$. Here, all leaf boxes are refined once more.
  Since the residual kernel $R_l$ is zero outside a ball of radius $|B_l|$,
  only blue and red boxes have nonzero interactions with the source box $C$ at level $l+1$.
  Furthermore, red boxes are well separated from $C$. Thus, the residual
  kernel is smooth for the interaction between red target boxes and the source box $C$.
  The residual kernel in this range can be well approximated by a low degree polynomial
  of $r^2$, avoiding the expensive calculation of special functions (including
  the square-root function). Finally, the self interaction is always carried out at
  level $l$ (i.e., the parent box $S$) to reduce cache misses in SIMD vectorization.}
\label{kernelevaluation}
\end{figure}

\subsection{The full \acron algorithm for discrete sources}

We assume we are given a level-restricted adaptive tree on input,
that has fewer than $n_s$ sources and targets in each leaf node.
Constructing such a tree requires $O(N \log N)$ work (assuming there are 
$O(\log N)$ levels in the tree hierarchy). See
\cite{biros2008sisc} for further details.
Each box $B$ in the data structure is then assigned some
logical flags. 

\begin{enumerate}
\item 
If $B$ is a leaf box, then $F_{\rm leaf}(B)=1$. Otherwise $F_{\rm leaf}(B)=0$. 
\item 
If $B$ contains more than $n_s$ sources, $F_{\rm out}(B)=1$.
Otherwise $F_{\rm out}(B)=0$. 
\item 
If any of $B$'s colleagues contain more than $n_s$ sources, $F_{\rm in}(B)=1$.
Otherwise $F_{\rm in}(B)=0$. 
\end{enumerate}

As noted above, we carry out one additional refinement for all boxes to accelerate
the direct interaction step. Thus, a box $B$ which was a leaf node in the given 
data structure has $F_{\rm leaf}(B)=1$. The children under this additional refinement
step are assigned $F_{\rm leaf}=0$.

\bigskip
\begin{center}
  \textbf{The discrete (\acron) algorithm}
\end{center}
\bigskip
{\em Comment} [On input, we are given a collection of $N$ sources and targets
  $\x_{i}$, $i=1,2,\ldots N$ and $\y_{j}$, $i=1,2,\ldots N$, 
  a prescribe precision $\epsilon$, and a parameter $n_s$ 
  (the maximum number of points in a leaf box).
  On output, we return the potential at all target locations.]
\bigskip

\begin{center}
  \textbf{Step 0: Initialization}
\end{center}
\bigskip

\begin{algorithmic}[1]
  \State{Build a level-restricted adaptive tree using the algorithm 
    (see, for example, \cite{biros2008sisc}).
    The root box is denoted by $B_0$, the maximum level is denoted 
   $L_{\rm max}$, and the total number of boxes is denoted $N_{\rm box}$.}
  \State{Initialize the potential $u(\x_i), i=1,\dots,N$ to zero.}
  \State{Calculate $\delta_1$ according to the prescribed precision $\epsilon$.}
  \State{Compute the Fourier expansion length $N_F^{(0)}$ for the windowed kernel $W_0$,
    the Fourier expansion length $N_F= (2n_f+1)^3$ for the difference kernels 
    $D_l, l=0,\ldots,L_{{\rm max}-1}$,
    and the polynomial
    approximation order $p$ in each dimension based on the prescribed precision $\epsilon$.}
%  \State{Allocate memory for four expansions and initialize them to zero.}
  \State{Compute the flags $F_{\rm leaf}(B)$, $F_{\rm out}(B)$, $F_{\rm in}(B)$,
    for each box $B$ in the tree.}
  \State{Carry out one additional refinement for each leaf node (assigning these
    additional boxes the flags $F_{\rm leaf}(B) = F_{\rm out}(B) = F_{\rm in}(B) = 0$).}
\end{algorithmic}

\bigskip
\begin{center}
  \textbf{Step 1: Precomputation}
\end{center}
\bigskip
{\em Comment} [Create transformation and translation matrices.] 

\begin{algorithmic}[1]
  \State{Compute  the $2^d$ translation matrices $T_{\rm c2p}$ and
    $T_{\rm p2c}$ according to \cref{ctoplemma} and
    \cref{ptoclemma}.}
  \State{Compute the transformation matrix $T_{\rm prox2pw}$ converting proxy charges
    into the outgoing plane wave expansions from \eqref{outgoingexpansion2}
    and the transformation matrix $T_{\rm pw2poly}$ converting incoming plane wave expansions
    into local expansions, from \eqref{incomingexpansion2},\eqref{localexpansion}. 
    By translation invariance,
    only one matrix is needed for each level. Using separation of variables, it is
    easy to see that one only needs to store one-dimensional matrices of
    size $(2n_f+1) \times p$.}
  \State{Compute the translation matrices $T_{\rm pwshift}$ converting outgoing plane wave
    expansions of a source box into incoming plane wave expansions of
    its colleagues from \eqref{diagonaltranslation}. 
    There are $3^d$ of them for each level.}
\end{algorithmic}

\bigskip
\begin{center}
  \textbf{Step 2: Upward pass}
\end{center}
\bigskip
{\em Comment} [For each leaf box, form the proxy charges from the original sources.]

\begin{algorithmic}
  \For{each box $B_i$, $i=1,\ldots,N_{\rm box}$}
  \If{$F_{\rm leaf}(B_i)=1$}
  \State{Form proxy charges $\tilde{\rho}_{\bm}(B)$ from sources via \cref{proxynd}.}
  \EndIf
  \EndFor
\end{algorithmic}
\bigskip
%{\em Comment} [For each parent box, form the proxy charges
%  by translating and merging proxy charges from its children.]

\begin{algorithmic}
  \For{level $l=L_{{\rm max}-1},\ldots,0$}
  \For{each box $B$ at level $l$}
  \If{$B$ is a parent box}
  \State{Form the proxy charges from proxy charges in children using $T_{\rm c2p}$.}
  \EndIf
  \EndFor
  \EndFor
\end{algorithmic}

\bigskip
\begin{center}
  \textbf{Step 3: Downward pass}
\end{center}
\bigskip

\begin{algorithmic}
  \For{level $l=0,\ldots,L_{{\rm max}-1}$}
  \For{each box $B$ at level $l$} \Comment{Form outgoing expansions}
  \If{$F_{\rm out}(B)=1$}
  \State{Form the outgoing expansion $\Phi_{l}(B)$ for the difference kernel $D_l$ from
    the proxy charge expansion coefficients using
    $T_{\rm prox2pw}$.}
  \EndIf
  \EndFor
  \bigskip

  \For{each box $B$ at level $l$} \Comment{Form incoming expansions} 
  \If{$F_{\rm in}(B)=1$} 
  \For{each box $B_n$ in $C(B)$}
  \If{$F_{\rm out}(B_n)=1$}
  \State{Translate the outgoing expansion $\Phi_l(B_n)$ to the center of $B$
    and add to the incoming plane wave expansion $\Psi_l(B)$ using 
    $T_{\rm pwshift}$.}
  \EndIf
  \EndFor
  \EndIf
  \EndFor
  \bigskip

  \For{each box $B$ at level $l$} \Comment{Form local expansions}
  \If{$F_{\rm in}(B)=1$}
  \State{Convert incoming plane wave expansion $\Psi_l(B)$ to the local
    expansion $\Lambda_l(B)$ using $T_{\rm pw2poly}$.}
  \EndIf
  \EndFor
  \bigskip

  \For{each box $B$ at level $l$} \Comment{Split local expansions}
  \If{$F_{\rm in}(B)=1$}
  \For{each child box $C$ of $B$}
  \State{Translate and add the local expansion of $B$ to the local expansion of $C$.}
  \EndFor
  \EndIf
  \EndFor
  \bigskip

  \For{each box $B$ at level $l$} \Comment{Evaluate local expansions}
  \If{$F_{\rm leaf}(B)=1$}
  \State{Evaluate the mollified potential $u_{L}^{\rm far}$ at each target $\x$ in $B$.}
  \EndIf
  \EndFor
  \EndFor
\end{algorithmic}

\bigskip
\begin{center}
  \textbf{Step 4: Compute direct interactions}
\end{center}
\bigskip

\begin{algorithmic}
  \For{level $l=0,\ldots,L_{{\rm max}}$}
  \For{each box $B$ at level $l$}
  \If{$B$ is a leaf box}
  \State{Compute the local contribution to the potential $u_{l}^{\rm local}$
    at all targets in the colleagues of $B$
    due to all sources in $B$ directly and add to the potential.
    The search for targets in this step is accelerated by carrying out the 
    processing using the refined leaf nodes (see \cref{localaccel}).}
  \EndIf
  \EndFor
  \EndFor
\end{algorithmic}

\bigskip
\begin{center}
  \textbf{Step 5: Self-interaction corrections}
\end{center}
\bigskip

\begin{algorithmic}
  \For{$l=0,\ldots,L_{\rm max}$} 
  \For{each leaf box $B$ at level $l$} 
  \For{each target $\y_j$ in $B$}
  \State{If target $\x_i$ is coincident with a source $\y_j$, compute the 
   self-interaction term $u^{\rm self}(\x_i)$ in \eqref{laplacesum2self}
   and add to $u(\x_i)$.}
  \EndFor
  \EndFor
  \EndFor
\end{algorithmic}
\bigskip

\subsubsection{Complexity analysis}

Let $n_f$ be the Fourier expansion length of the difference and windowed kernels in each
dimension, and let $p$ be the polynomial approximation order for these kernels 
in each dimension. (In general, $n_f>p$, so that terms of the form 
$p n_f^d + p^2 n_f^{d-1} + \dots p^d n_f$ are dominated by $p n_f^d$.)
It is straightforward to verify that
\begin{itemize}
\item The construction of the proxy charges on leaf boxes requires $O(p^d)$ work per source.
\item The translation of proxy charges from child to parent requires $O(d p^{d+1})$ work
per box.
\item The evaluation of outgoing expansions from proxy charges requires $O(d p n_f^d)$ 
work per box.
\item The translation from outgoing expansions to incoming expansions requires $O(3^dn_f^d)$
  work per box.
\item The computation of local expansions from incoming expansions requires 
$O(d p n_f^d)$ work per box.
\item The translation of local expansions from parent to child requires $O(d p^{d+1})$ 
work per box.
\item The evaluation of the mollified potential requires $O(p^d)$ work per target.
\item The evaluation of the local potential requires $O(3^d n_s)$ work per target.
\end{itemize}
Thus, the total cost $C$ is given by
\be\label{discretemkdscost}
C = O(3^d n_s N + 2p^d N + 2d p^{d+1} N_{\rm box} + 2d p n_f^d N_{\rm box} + 3^d n_f^d N_{\rm box}).
\ee
In the case of uniform distributions, $N_{\rm box}$ is bounded by $2N/n_s$, leading to $O(N)$
complexity. For general distributions, the proof of linear complexity (after sorting)
is similar to that in \cite{nabors1994sisc} for the FMM.
BLAS level 3 subroutines such as {\rm DGEMM} and {\rm ZGEMM}. 

\begin{remark}
The first term on the right hand side of \cref{discretemkdscost} 
involves the evaluation of special functions and has a large prefactor implicit in the
$O(N)$ notation.  All of the other terms have very small
associated constants and many  of the steps are compatible with optimized linear
algebra (BLAS) routines.
As observed above, the prefactor $3^d$ 
in the direct evaluation step with residual kernels can be accelerated
by looking only within a sphere of radius $\cR_L$ centered on a source at level $L$.
This reduces the volume to $\pi$ instead of $9$ in 2D and $4\pi/3$ instead of $27$ in 3D.
In practice, this makes it harder to use ``SIMD" vectorization on modern
architectures.
\end{remark}

\begin{remark}
The parameter $n_s$ should be chosen so that the cost of local interactions 
and the cost of all other steps are roughly balanced.
In our implementation, we optimize the selection of $n_s$ by experimentation with 
sources distributed on a surface. 
\end{remark}

\subsection{Modification of \acron for continuous sources}\label{sec:continuoussources}

We now consider the evaluation of the volume potential 
\eqref{volumepotential} for the specific case of the Laplace kernel:
\be\label{laplacevolumepotential}
u(\x)=\int_{B_0} \frac{1}{|\x-\y|}\rho(\y)d\y,
\ee
where $B_0$ is the unit cube $[-1/2,1/2]^3$, and the density $\rho$ is assumed to be smooth
on $B_0$. 
The first task is to build a level-restricted adaptive tree that resolves the density $\rho$
to precision $\epsilon$. That is, we keep refining the tree until
the density is approximated by a tensor product
polynomial of degree $\le q-1$ on each leaf box $B$ to the desired precision:
\be \label{rhoapprox}
\rho(\x)\approx \sum_{\bj \in [1,\dots,q]^d} c_{\bj} 
 P_{\bj-{\bf 1}}(\x), \quad \x\in B,
\ee
where we have switched from Chebyshev to Legendre polynomials 
for convenience when computing integrals (see \cref{sec:interp}).

\subsubsection{Creating proxy charges}

The changes to the code in terms of the far-field interactions are very minor.
We focus on the difference kernel $D_l$ here. 
The windowed kernel at level 0 is treated in the same way. 
For a leaf box $B$, we 
need to evaluate integrals of the form
\be\label{dkintegral}
u_{l}^{\rm diff}(\x) = \int_B D_l(|\x-\y|) \rho(\y) d\y.
\ee
We have seen above that $D_l$ is well-approximated by a polynomial of order
$p$. Assuming $q\le p$, the integral in \cref{dkintegral} can
be computed accurately using tensor product Gauss-Legendre quadrature of order $p$
since it is exact for polynomials of degree up to $2p$. That is,
\be\label{dkintegral2}
u_{l}^{\rm diff}(\x) \approx \sum_{\bn \in [1,\dots,p]^d}
\left(\frac{\cR_l}{2}\right)^d w_{\bn}
D_l(\x,\bq_{\bn})\rho(\bq_{\bn}),
\ee
where $\cR_l$ is the linear dimension of $B$, the points
$\{ \bq_{\bn} \}$ are the tensor product Legendre nodes (see \cref{sec:interp}),
and $w_{\bn}$ are the weights of the 
standard tensor-product Gauss-Legendre quadrature rule
scaled to $B$. Thus, the proxy charges for a leaf node are 
\[ \tilde{\rho}_{\bn} = 
\left(\frac{\cR_l}{2}\right)^d w_{\bn}
\rho(\bq_{\bn}).
\]
This requires evaluation of the Legendre expansion \eqref{rhoapprox}
at the $p^d$ tensor-product Legendre nodes at a cost of $O(q p^d)$ work using
separation of variables - much less work than in the discrete case.

Likewise, when evaluating the local expansions in Step 3 of the algorithm, 
the targets lie on a tensor-product
grid, so that the cost is of the order $O(q p^d)$ work rather than $O(q^d p^d)$.
Amortized over the $q^d$ discretization points needed to represent $\rho(\x)$, this
amounts to a cost of the order $O(N q (p/q)^d)$ work instead of $O(N p^d)$ - a significant
reduction.

\subsubsection{Approximation of the local interactions}

Recall that, in the discrete setting, the direct calculations involve sums of
residual kernel interactions $R$~\cref{kerneldef}. In the continuous setting,
these interactions involve integrals of singular functions and require some care.
We will make use of the integral representation
\be\label{localkernel}
R(r)=\frac{\erfc(r/\sigma)}{r}=\frac{2}{\sqrt{\pi}}\int_{1/\sigma}^\infty e^{-r^2t^2}dt
=\frac{1}{\sqrt{\pi}}\int_0^{\sigma^2} \frac{e^{-r^2/s}}{s^{3/2}} \, ds \, ,
\ee
which can be derived from \eqref{laplacesog} and the definition of 
$\erfc$.  We have dropped the level index $L$ in the residual kernel here,
since it is simply determined by the choice $\sigma = \sigma_L$.
The formulas in \eqref{localkernel} must be understood in a distributional sense,
which is justified here since what we actually need to compute is 
integrals of the following form:
\be\label{lkintegral}
u^{\rm local}(\x) = \int_B R(|\x-\y|)\rho(\y)d\y.
\ee
Without loss of generality, let $B$ be the unit cube centered at the origin with $\sigma$
chosen according to \eqref{deltavalue} as
\be\label{deltavalue2}
\sigma\approx \frac{1}{\sqrt{\log(1/\epsilon)}}.
\ee

For readers familiar with FMM-based volume integral codes, recall that the local interactions
are computed by tabulating integrals of the Green's function $1/r$ with each of $q^d$
polynomial basis functions at each of $q^d$ points within the box or its neighbors.
Thus, the work required is of the order $O(q^{2d})$. 
In the \acron framework, we seek to exploit the compact support of the 
residual kernel and the formulas in 
\cref{localkernel}, which we can interpret as
a continuous sum-of-Gaussians (SOG) representation for $R(r)$ involving more and more
sharply peaked Gaussians as $\sigma \rightarrow 0$.
We will actually truncate the integrals in \eqref{localkernel}, and write
\be\label{localkernel2}
R(r)
= \frac{2}{\sqrt{\pi}}\int_{1/\sigma}^{1/\sigma_{\rm min}} e^{-r^2t^2}dt
= \frac{1}{\sqrt{\pi}}\int_{\sigma_{\rm min}^2}^{\sigma^2} \frac{e^{-r^/s}}{s^{3/2}} \, ds \, ,
\quad r\in [\epsilon_0,1],
\ee
where $\epsilon_0$ is a free parameter and
\be
\sigma_{\rm min}=\sigma \epsilon_0 \approx \frac{\epsilon_0}{\sqrt{\log(1/\epsilon)}}.
\ee
It is easy to check that for $r\in [\epsilon_0,1]$ the resulting error in \eqref{localkernel2}
is of the order $O(\epsilon)$.
We then seek a discrete SOG approximation of $R(r)$ of the form
\be\label{sogapprox}
R(r) \approx \sum_{i=1}^{n_g} w_i e^{-r^2 t_i^2}, \quad r\in [\epsilon_0,1],
\ee
The choice of $\epsilon_0$ is somewhat involved and discussed below. Given $\epsilon_0$, however, 
there are many ways of constructing the SOG approximation in \cref{sogapprox}. See, for example,
\cite{beylkin2010acha} or the black-box algorithm in 
\cite{greengard2018sisc}. Combined with 
the application of generalized Gaussian quadrature 
\cite{ggq2,ggq3,ggq1}, one obtains an SOG approximation with a number
of terms of the order
\be
n_g=O(\log(1/\epsilon)).
\ee
Combining \cref{sogapprox} and \cref{lkintegral}, we obtain
\be
u^{\rm local}(\x) = u_1^{\rm local}(\x) + u^{\rm asymp}(\x),
\ee
where 
\be
\label{lkintegral2}
u_1^{\rm local}(\x) =
\sum_{i=1}^{n_g} w_i \int_B e^{-|\x-\y|^2t_i^2} \rho(\y)d\y
\ee
and
\be\label{asymptoticterm}
u^{\rm asymp}\approx \int_{|\x-\y|\le \epsilon_0} \left(\frac{\erfc(|\x-\y|/\sigma)}{|\x-\y|}-
\sum_{i=1}^{n_g} w_i e^{-|\x-\y|^2t_i^2}\right) \rho(\y) d\y.
\ee
For $u_1^{\rm local}(\x)$,
each term in \cref{lkintegral2} can be computed using separation of variables 
in $O(d q^{d+1})$ operations. 
For the contribution $u^{\rm asymp}(\x)$,
one can carry out the calculation explicitly using a Taylor expansion of the density 
$\rho(\y)$. For illustration, using a fourth order Taylor series and integrating term by
term, we obtain
\be\label{asymptoticterm2}
\ba
u^{\rm asymp}(\x) &= 4\pi \rho(\x) \left(I_0 - \sum w_i \int_0^{\epsilon_0}e^{-r^2\delta_i}r^2dr\right) \\
&+\frac{2\pi}{3} \Delta \rho(\x)
\left(I_2 - \sum w_i \int_0^{\epsilon_0}e^{-r^2/\delta_i}r^4dr\right) \\
&+\frac{\pi}{6} \Delta^2 \rho(\x)
\left(I_4 - \sum w_i \int_0^{\epsilon_0}e^{-r^2/\delta_i}r^6dr\right) +O(\epsilon_0^8),
\ea
\ee
where $I_n, n=0,2,4$ are given by the formula
\be
I_n=\int_0^{\epsilon_0} \erfc(x/\sigma)x^{n+1} dx,
\ee
and the error estimate $O(\epsilon_0^8)$ comes from the integration over the ball of radius $\epsilon_0$
of the error term in the Taylor series approximation (of order $O(\epsilon_0^6)$) multiplied
by $1/r$. 
The integrals in \eqref{asymptoticterm2} are easily computed in terms of the error function.
We are now able to choose $\epsilon_0$. We simply require (assuming the fourth order
approximation for $\rho$) that $\epsilon_0^8 = O(\epsilon)$ for
precision $\epsilon$.

To summarize, the cost of local interactions is of the order $O(\log(1/\epsilon) dq^{d+1})$
for each pair of boxes, leading to a cost of approximately
$O(3^d d q\log(1/\epsilon))$ per grid point.
In place of \cref{discretemkdscost}, the total
cost of the \acron algorithm for continuous sources is
\be\label{continuousmkdscost}
C=O(3^d dq N + 2(p/q)^{d-1} N + 2d p^{d+1} N_{\rm box} + 2d p n_f^d N_{\rm box} + 3^d n_f^d N_{\rm box}),
\ee
Here, $N$ is the total number of grid points:
\be\label{gridnumber}
N= N_{\rm leaf} q^d,
\ee
where $N_{\rm leaf}$ is the number of leaf boxes in the tree. In the tree construction
for a continuous density, it is easy to check that
the total number of boxes in the tree satisfies the bound
\be\label{boxnumber}
N_{\rm box}\le \frac{1}{1-2^{-d}}N_{\rm leaf}.
\ee
Combining \cref{continuousmkdscost}, \cref{gridnumber}, and \cref{boxnumber}, we obtain
\be\label{continuousmkdscost2}
C=O([3^d dq + 2(p/q)^{d-1} + 2d p (p/q)^{d} + 2d p (n_f/q)^d + 3^d (n_f/q)^d] \, N).
\ee

As noted above,
when compared with FMM-based volume integral approaches, the advantage of the 
present method is the expression of the residual kernel in terms of a modest
number of Gaussians,
permitting separation of variables acceleration in the local interactions and a dramatic
reduction in the storage of local tables (since all of the needed matrices are one-dimensional).

\subsection{Optimization of kernel splitting using prolate spheroidal wave functions}
\label{sec:betterkernelsplitting}

The last general improvement we introduce in the \acron framework
is a reduction in the number of quadrature nodes needed in \eqref{differencekernelft2a}.
For this, we will abandon the kernel splitting in
\cref{kerneldef}, which is closely related to the integral representation \cref{laplacesog}
which makes essential use of the fact that a Gaussian is rapidly decaying in both physical
and Fourier space. 

From the discussion thus far, the reader will note that the essential features
of kernel splitting aren't specific to the Gaussian, but assume that
\begin{enumerate}[label=(\alph*)] 
\item the windowed kernel is smooth and admits an efficient Fourier transform,
\item the difference kernels $D_l(r)$ are smooth, admit efficient Fourier transforms,
  and are compactly supported (to precision $\epsilon$)
  in the ball of radius $\cR_l$,
  where $\cR_l$ is the side length of the boxes at level $l$, and
\item the residual kernels are compactly supported (to precision $\epsilon$)
  in the ball of radius $\cR_l$.
\end{enumerate}

Fixing our domain to be the interval $[-1,1]$, note that 
in order for a Gaussian of the form $e^{-\delta^2 r^2}$
to decay to $\epsilon$ outside the interval $[-1,1]$, we need to choose
\be
\delta = \sqrt{\log(1/\epsilon)}.
\ee
By \cref{gaussiankernel}, the corresponding Fourier transform can be windowed at
\be\label{gaussiankmax}
K_{\rm max}=2\log(1/\epsilon).
\ee
Thus, it is worth considering other options, such as $\psic$, the first prolate
spheroidal wave function (PSWF) of order zero, discussed in \cref{sec:windownfunctions}.
For $\psic$ to decay to $\epsilon$ outside $[-1,1]$, we have that
\be\label{pswfcvalue}
c\approx \log(1/\epsilon).
\ee
By \cref{pswffouriertransform}, the Fourier transform of $\psic$ is
$\lambda_0 \psic(k/c)$, i.e., its Fourier transform can be windowed at
\be\label{pswfkmax}
K_{\rm max}=c\approx  \log(1/\epsilon).
\ee
Comparing \cref{gaussiankmax} with \cref{pswfkmax}, we find that
the number of Fourier modes needed to approximate $\psic$ is roughly
half that needed for the corresponding Gaussian. In $d$ dimensions, this
results in a reduction of the total number of Fourier modes by a factor of $2^d$, 
all else being equal. In \cref{pswfgaussian}, we plot the PSWF of order zero $\psic$ and
the corresponding Gaussian at six digits of accuracy on the interval $[-1,1]$,
as well as their Fourier transforms. 

\begin{figure}[!ht]
\centering
\includegraphics[height=40mm]{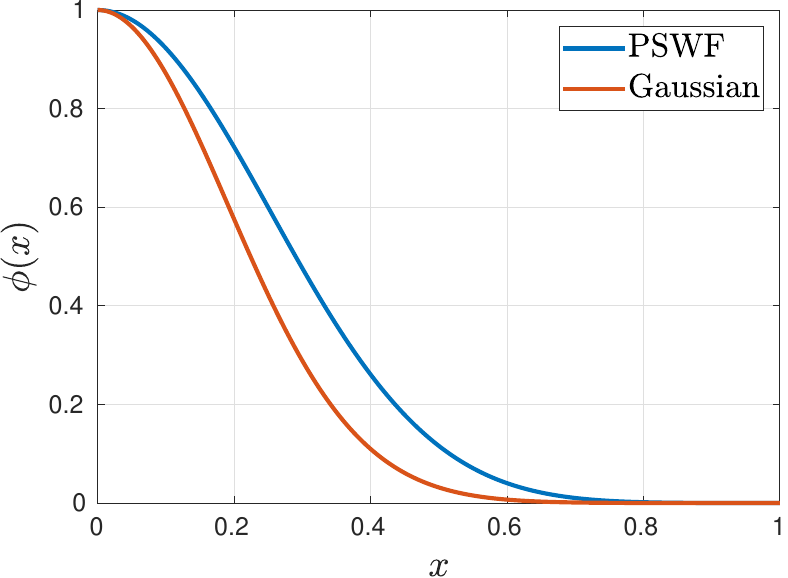}
\hspace{0.4in}
\includegraphics[height=40mm]{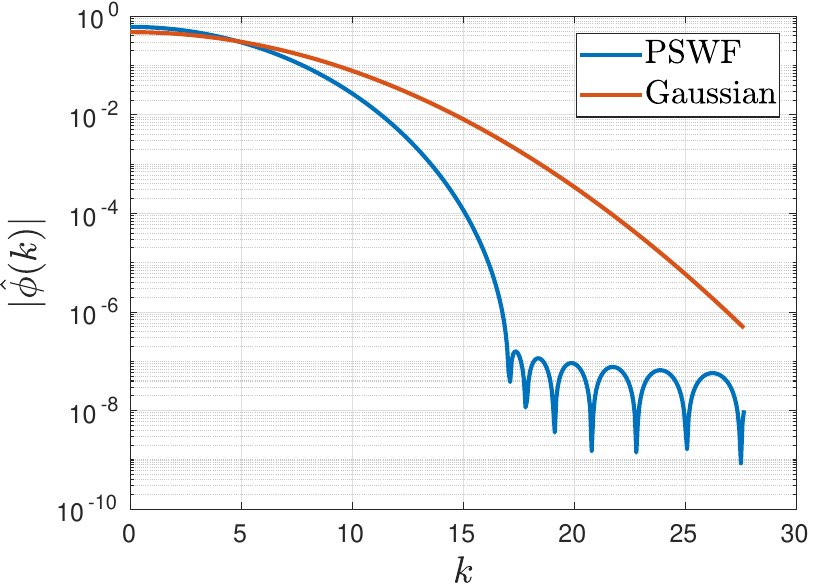}
\caption{\sf Comparison of the Gaussian with the PSWF of order zero.
  Left: the scaled PSWF of order zero $\psic(x)/\psic(0)$ with $c=16.893999099731445$
  and the Gaussian $g(x)=e^{-\delta^2 x^2}$ with $\delta=\sqrt{\log(10^6)}$.
  The parameters $c$ and $\delta$ are chosen such that $\psic(1)=g(1)=10^{-6}$.
  Right: the Fourier transforms of the Gaussian and $\psic$. 
  The Fourier transform of the latter decays much more rapidly
  than that of the former, requiring many fewer terms in the spectral approximation.}
\label{pswfgaussian}
\end{figure}

Using the function $\psic$, we propose an alternative kernel splitting. Letting
\be
c_0 = \int_{0}^1 \psic(x)dx = \frac{1}{2}\hpsic(0),
\ee
and
\be
\Phi^c_l(r) = \frac{1}{c_0} \int_0^{r/\cR_l}\psic(x)dx,
\ee
the new windowed, difference, and residual kernels are
defined by the formulas
\be\label{newkernels}
\ba
W_0(r) &= \frac{\Phi^c_0(r)}{r},\\
D_l(r) &=\frac{\Phi^c_{l+1}(r)-\Phi^c_{l}(r)}{r},\\
R_L(r) &=\frac{1-\Phi^c_L(r)}{r}.
\ea
\ee

The corresponding splitting of the $1/r$ kernel is (repeating \eqref{mlks} here
for convenience):
\[
\frac{1}{r} = W_0(r) + \sum_{l=0}^{L-1} D_l(r) + R_L(r), \quad L=0, 1, \ldots, L_{\rm max},
\]
which holds for any level $L$ and any target in the unit box.

Since $\Phi^c_l(x)=1$ for any $x \geq \cR_l$  by construction (to the desired
precision),
the difference and residual kernels in \cref{newkernels} are compactly supported
on balls centered at the origin of radius $\cR_l$ and $\cR_L$, respectively.
Since $\psic$ is an even function,
both the windowed and difference kernels are smooth at the origin.
Admittedly, the difference kernel $D_l(r)$ is not smooth at $r=\cR_{l+1}$ since
$\psic(r/\cR_{l+1})$ has decayed only to $\epsilon$ and is extended beyond that
range to zero. For the purpose of computation, however, this introduces an error
of the order $O(\epsilon)$ in both physical and Fourier space and can be safely ignored.
In short, the difference kernel
admits an efficient Fourier transform and the new kernel splitting satisfies the 
three properties in \cref{sec:betterkernelsplitting}.

From the formulas in \cref{sec:windownfunctions}, it is straightforward to show that
the Fourier transform of the difference kernel is
\be\label{differencekernelft3}
\ba
\hat{D}_l(k)
& = \frac{4\pi}{\psic(0)}\frac{\psic(|k|\cR_{l+1}/c)-\psic(|k|\cR_l/c)}{|k|^2},
\quad |k|\ne 0,\\
\hat{D}_l(0)&=\frac{2\pi c_2}{c_0}\left(\cR_l^2-\cR_{l+1}^2\right), \quad c_2=\int_0^1 x^2\psic(x)dx.
\ea
\ee
and that the Fourier transform of the windowed kernel is
\be\label{windowedkernelft2}
\widehat{W}_0(k) = 4\pi \left(\frac{\sin((1+C)\cR_0k/2)}{k}\right)^2 \hat{\phi}^c_{0}(|k|).
\ee

\cref{tablepswfparameters} lists the actual values of $c$ in $\psic$, 
the number of quadrature nodes $N_1 = (2n_f+1)$ needed to discretize the Fourier
transform for the difference kernels, and
the value $p$ used to define the order of the polynomial representing
the far field potential as well as the number of proxy
charges in each linear dimension.
We also show the grid spacing needed for the 
trapezoidal rule discretization of the difference kernels in the Fourier domain.

The cost of computing the 
residual kernels is more or less identical, since 
the evaluation of either $\Phi^c_l$ or $\erfc$ is most efficiently done 
by polynomial (or piecewise polynomial) approximation of the kernel from 
a pre-computed, one-dimensional table of values
using Estrin's method~\cite{estrin1960}. The number of terms required in the
polynomial approximation of $\Phi_l$ and $\erfc$ are more or less identical.

\begin{table}[t]
  \caption{\sf Parameters used in the PSWF kernel splitting for the 3D Laplace
    kernel. $\epsilon$ is the desired precision, $c$ is the computed PSWF parameter
    for the indicated value of $\epsilon$, $N_1$ is the total number of Fourier modes
    in each dimension for $\psic$, $p$ is the polynomial approximation order needed,
    and $h_0$ is the grid spacing in Fourier space for the difference
    kernel $D_0$, respectively.
    At level $l$, $h_l = 2^l h_0$ and $K_{max} = 2^l \, c$.
    Note that $h_0$ is very close to the analytic formula \cref{fourierspacing}, 
    namely $4 \pi/3$. For comparison, we also list $N_1^{\rm G}$ (the total number
    of Fourier modes in each dimension for the original splitting using Gaussians) and 
    $p^{\rm G}$ (the associated polynomial approximation order for the difference kernels)
    in the last two columns.}
\centering
\begin{tabular}{c|c|c|c|c|c|c|c}
\toprule
$\epsilon$ & $\log(1/\epsilon)$ & $c$ & $N_1$ & $p$ & $h_0$ & $N_1^{\rm G}$ & $p^{\rm G}$\\
\midrule
$10^{-3}$  & $6.9$  & $7.2462000846862793$ & 13 & 9  & $1.3240\pi$ & 22 & 16 \\
$10^{-6}$  & $13.8$ & $13.739999771118164$ & 25 & 18 & $1.3372\pi$ & 44 & 30 \\
$10^{-9}$  & $20.7$ & $20.736000061035156$ & 39 & 28 & $1.3250\pi$ & 66 & 46 \\
$10^{-12}$ & $27.6$ & $27.870000839233398$ & 53 & 38 & $1.3354\pi$ & 88 & 62 \\
\bottomrule
\end{tabular}
\label{tablepswfparameters}
\end{table}

\section{The \acron framework for general kernels}

Most of the steps in the \acron algorithm are both dimension-independent and kernel-independent,
since the mollified interactions are dealt with using the 
Fourier transform and polynomial approximation. 
For each kernel, however, the user must develop a telescoping series the requisite 
properties in \cref{sec:betterkernelsplitting}.
For each such kernel-splitting formalism, its Fourier transform must be obtained and
an accurate estimate of the quadrature for the spectral representation.
While we do not attempt to provide a complete theory here, we show by example that
the \acron framework is applicable and effective for a broad class of kernels. 
Note that, once the kernel is expressed as the
sum of a windowed kernel and a residual kernel based on a parameter $\sigma$,
the difference kernels can simply be define as the
difference of windowed kernels at two successive scales (as in \cref{windowedkernelL}).

\subsection{Smooth kernels}

If the kernel function $K(r)$ is smooth, then kernel splitting is particularly simple.
Letting $w(x)$ be a smooth bump function supported on $[-1,1]$ to the desired precision
(such as a suitably scaled Gaussian or PSWF), we can simply write 
\be\label{smoothkernelsplitting}
K(r) = M_l(r)+R_l(r) \coloneqq K(r) (1-w(r/\sigma_l)) + K(r) w(r/\sigma_l).
\ee
Here, 
$M_l$ is the mollified kernel and $R_l$ is the residual kernel, which is compactly supported
in the ball of radius $\sigma_l$.
$\sigma_l$, as in the \acron method described in detail above, is a parameter that
depends on the precision
$\epsilon$ and level $l$. The efficiency of the splitting depends largely on two factors:
how efficiently can we compute the residual kernel $R_l$ and
how many points are needed in the spectral representation of the 
difference kernel $D_l=M_{l+1}-M_{l}$. 
These are kernel-specific issues for which it is difficult to develop a general theory,
although both theory and numerical experiments suggest that 
the PSWF $\psic$ is a better choice of a window function than a Gaussian (see the
discussion in \cref{sec:betterkernelsplitting}).

\subsection{Singular, slowly decaying kernels}

Let us now consider long range, translation invariant kernels $K(r)$ which have a 
singularity at the origin. By duality, their
Fourier transforms are also singular at or near the origin and
slowly decaying. 
The essential requirement in constructing a hierarchical splitting scheme with the
properties enumerated in \cref{sec:betterkernelsplitting} is to find a sequence of
mollified kernels $S_l(r)$ with two properties:
\begin{enumerate}
\item $S_l(r)$ is smooth 
\item $S_l(r) = K(r)$ for $r\ge \cR_l$ to the requested
precision $\epsilon$. 
\end{enumerate}

There are a number of methods one can use for this purpose, illustrated
with specific examples in the following sections, where we exploit
known integral representations, sum-of-Gaussian approximations, etc.
If the Fourier transform $\widehat{K}(k)$ is known, 
one can also multiply by a suitable sequence of ``filters" in the Fourier domain, yielding
a fairly general approach.  We present an example in \cref{fsplitpwsf}.

In this section, we point out that there are purely numerical methods which can be used
for this construction as well, without knowledge of the Fourier transform.
One of these is to define $S_l = K(r)$ for $r\ge \cR_l$, and to use a smooth extension
algorithm to fill in function values on $[0,\cR_l]$ to enforce some degree of smoothness,
such as with the function extension method in 
\cite{epstein2022arxiv} which can enforce $\mathcal{C}^n$ smoothness. For $n$ sufficiently
large, the Fourier transform can be computed numerically and 
decays rapidly enough that the trapezoidal rule is highly accurate.

One could also construct a sequence of smooth kernels $S_l(r)$ by convolution 
in physical space with a hierarchy of carefully chosen Gaussians or other 
{\em approximations to the identity} \cite{stein1971}. We do not investigate these
options in the present paper, focusing on the kernels for which other approaches are
easily available.

Once we have constructed $S_l$, the kernel splitting takes the standard form
\be\label{generalkernelsplitting2}
K(r) = W_0(r) + \sum_{l=0}^{L-1}D_l(r) + R_L(r), \quad L=0,1,\ldots,L_{\rm max},
\ee
where
\be\label{generalkernels}
\ba
W_0(r) &= S_{0}(r) w(r),\\
D_l(r) &= S_{l+1}(r)-S_l(r),\\
R_L(r) &= K(r)-S_L(r).
\ea
\ee

Since the difference kernels $D_l$ are compactly supported in physical space, they  
are smooth in Fourier space (by the Paley-Wiener theorem)
and the trapezoidal rule is spectrally accurate.
For $W_0(r)$, we also need a smooth window function
$w(r)$ such that $w(r)=1$ for $|r|\le \sqrt{d}\cR_0$
and $w(r)=0$ for, say, $r>2\cR_0$. Such a window function enforces rapid decay of the Fourier
transform of $W_0$. In short, kernel splitting in the general case involves
hierarchical smoothing coupled with windowing at the coarsest level.

\subsection{Kernels with known integral representations}

In \cref{sec:integralrepresentation}, we have listed a collection of
kernels for which there are known integral
representations involving Gaussians. Each of these integral representations provides a
natural formalism for kernel-splitting.
In general, suppose that the kernel $K(r)$ admits the following integral representation:
\be \label{kintrep}
K(r)=\int_0^\infty s(t)e^{-r^2t^2}dt
\ee
for some weight function $s(t)$.
(The integral representations
listed in \cref{sec:integralrepresentation} are all of this form.) 
Then the multilevel kernel splitting of $K(r)$ can be accomplished using Gaussians:
\be\label{generalkernelsplitting}
K(r)=M_0(r)+\sum_{l=0}^{L-1}D_l(r)+R_L(r), \quad L=0,1,\ldots,L_{\rm max},
\ee
where
\be
\ba
M_0(r)&=\int_0^{\delta_0}w(t)e^{-r^2t^2}dt,\\
D_l(r)&=\int_{\delta_{l}}^{\delta_{l+1}} w(t)e^{-r^2t^2}dt,\\
R_L(r)&=\int_{\delta_{L}}^\infty w(t)e^{-r^2t^2}dt,
\ea
\ee
with $\delta_l = 1/\sigma_l$ and $\sigma_l$ is given by \eqref{deltavalue}.

\begin{remark} \label{delta_vs_sigma}
The parameters $\delta_l$ get {\em larger} as one moves to finer levels,
whereas the parameters $\sigma_l$ get smaller. 
More precisely, 
$\sigma_l$ is the length scale of a box and, in some respects, more natural. 
Many integral representations, however,
are more easily written in the form \eqref{kintrep}, and we will use the inverse
length scale $\delta_j$ for many kernels below.
\end{remark}

Care must be taken when constructing an efficient discrete
Fourier approximation of $M_0$. While $M_0$ is smooth in  
physical space, it may still have a far field which is not rapidly decaying, so that
its Fourier transform could have some singular structure (typically at the origin for 
non-oscillatory kernels).
One approach to finding a windowed function with the same value in the computational
domain (say the unit box) is to calculate the Fourier representation of a truncated version
of $M_0$, with the truncation beyond the domain of interest.
This is a topic discussed extensively in \cite{vico2016jcp} and used 
above for the $1/r$ kernel (see \cref{vicoremark}). Other specific examples will be discussed
shortly.

\subsection{Sum-of-Gaussians approximation}

When an integral representation for a kernel is available in terms of Gaussians,
as in the preceding section, finding a telescoping approximation is simply a matter of
quadrature. When no such representation is available, however, one can begin
by constructing an approximation of the form 
\be\label{sogapproximation}
K(r)=
\begin{cases}
\sum_{i=1}^{n_g} w_i e^{-r^2t^2_i}, & r\in [\epsilon_0,\sqrt{d}r_0], \\
\sum_{i=1}^{n_g} w_i e^{-r^2t^2_i} + K_{loc}(r) & 0 < r < \epsilon_0. 
\end{cases}
\ee
where $r_0$ is the side length of the box $B_0$. The cut-off $\epsilon_0$ is generally 
needed since many kernels of interest are singular at the origin.

Constructing such a sum of Gaussians (SOG) approximation on an interval bounded away from the
origin is a well-studied (but nonlinear) task and
closely related to approximation
by sums of exponentials (SOE) and rational approximation and we refer the reader to 
the literature for further details (see, for example,
\cite{beylkin2010acha,glover1984ijc,jiang2013jsc}).
Recently, attempts have been made to permit the ``black-box" construction of SOG approximations
(see, for example, \cite{gao2022jsc,greengard2018sisc}).

Given the SOG approximation valid on $[\epsilon_0,\sqrt{d} r_0]$,
kernel splitting can be achieved by defining the
windowed, difference, and residual kernels as follows:
\be\label{sogkernels}
\ba
W_0(r) &= \sum_{t_i\le \delta_0} w_i e^{-r^2t^2_i},\\ 
D_l(r) &= \sum_{\delta_l\le t_i\le \delta_{l+1}} w_i e^{-r^2t^2_i},\\ 
R_L(r) &= \sum_{t_i\ge \delta_L} w_i e^{-r^2t^2_i}. 
\ea
\ee
Clearly, the difference and residual kernels are compactly supported with the
correct support size (since $\delta_l = 1/\sigma_l$). 
Moreover, the windowed and difference kernels are smooth and admit
efficient Fourier approximations. 

More importantly, for continuous sources, 
we can exploit the representation of the residual kernel 
as a sum of Gaussians, as we did in \cref{sec:continuoussources}, and accelerate
the computation of local interactions using separation of variables.
When $t_i$ is very large, however, we can do even better.
and compute the local interactions asymptotically but with controlled precision.
For this, suppose we have expanded the density $\rho$ as a Taylor
series. We may then replace the finite domain of integration 
(over a leaf node and its colleagues)
by integration over $\mathbb{R}^d$ for any
individual Gaussian $e^{-r^2t^2_i}$ since the Gaussian has decayed
to $\epsilon$ outside the colleagues at level $L$ by construction (that is,
$t_i \geq \delta_L$).
To sixth order in $\sigma$, it is easy to check that
\be
\int_{\mathbb{R}^d}e^{-|\x-\y|^2/\sigma}\rho(\y)d\y =(\pi \sigma)^{d}
\left(\rho(\x) + \frac{\sigma^2}{4}\Delta \rho(\x) + \frac{\sigma^4}{32} \Delta^2 \rho(\x)
+ O(\sigma^6)\right).
\ee
Estimating the error term,
the asymptotic expansion is accurate to the desired precision as soon as 
$\left(4\sigma^2/\cR_l^2\right)^3 < \epsilon$, where $\sigma=1/t_i$.
Thus, for values $t_i$ in \eqref{sogkernels} that satisfy this criterion, we can bypass
the table-based, separation of variables method to convolve with the Gaussian kernel
and replace it with a calculation
involving only a few floating point operations per target (once the Laplacian and bi-Laplacian have
been applied to $\rho$).

We also need to account for the contribution of the local correction kernel $K_{loc}$ in 
\cref{sogapproximation}, which is supported
on $[0,\epsilon_0]$:
\be
\ba
&\int_{B_0} K_{loc}(\x-\y) \sigma(\y)d\y  \\
&= \int_{\x+B_{\epsilon_0}} K_{loc}(r) \sigma(\x+\y)d\y.
\ea
\ee
It is easy to see that the above integral can be reduced to a set of one-dimensional integrals
on $[0,\epsilon_0]$ by expanding the density as a Taylor series in spherical coordinates
and having access to integrals of the form
\be
\int_0^{\epsilon_0} K_{loc}(r) r^{k} \, dr
\ee
These integrals are easily computed on the fly at negligible cost.

For discrete sources, \cref{sogapproximation} can still be used for kernel splitting,
but the local interactions become expensive (without significant precomputation work)
because of the large number of terms needed to evaluate the residual kernel.

\subsubsection{Sum-of-Gaussians approximation for the Yukawa kernel}

In general, one may obtain nearly optimal SOG approximation for a given kernel
whose range is restricted to the computational domain by using generalized 
Gaussian quadrature~\cite{ggq1,ggq2,ggq3}. This is particularly effective if
this can be done in a precomputation step so that it does not dominate the total
computational cost. For any given power function, this is easy to do (storing
only the values $\{ w_i, t_i \}$ in \eqref{sogapproximation}).
For the Yukawa kernel, this is not so straightforward because the tables are distinct
for each value of the parameter $\lambda$ in the kernel. 
Thus, we seek an SOG approximation that can be computed 
on the fly at negligible cost. For this, we use the
change of variable $u=e^t$ in \cref{yukawasog}, which leads to
\be\label{yukawasog2}
\ba
\frac{1}{2\pi} K_0(\lambda r)
& = \frac{1}{2\pi}\int_{-\infty}^\infty e^{-r^2e^{2u}-\frac{\lambda^2e^{-2u}}{4}}du,
\quad \x\in\mathbb{R}^2,\\
\frac{1}{4\pi} \frac{e^{-\lambda r}}{r}
& = \frac{1}{2\pi^{3/2}}\int_{-\infty}^\infty e^{-r^2e^{2u}-\frac{\lambda^2e^{-2u}}{4}+u}du,
\quad \x\in\mathbb{R}^3.
\ea
\ee
Note that the integrands in \cref{yukawasog2} are exponentially decaying at
$\pm\infty$ so that the windowed trapezoidal rule converges with spectral accuracy.
When the parameter $\lambda$ is very small, a more efficient optimization procedure related
to the modified Prony method in \cite{beylkin2010acha} can be used to reduced the 
number of Gaussians for $u \in (-\infty,0]$.

\subsection{Kernel splitting of the general power function using PSWFs} 

Since, as we saw for the case $1/r$, it is generally more efficient to carry out kernel
splitting with prolate spheroidal wave functions rather than Gaussians, we consider
how to do so for the general power function. Other kernels can be treated in a similar fashion.
For the general power function, its kernel splitting using $\psic$ can be constructed
as follows:
\be\label{powerfunctionpswfsplitting}
\frac{1}{r^\alpha}  = M_0(r) + \sum_{l=0}^{L-1} D_l(r) + R_L(r), \quad L=0, 1, \ldots, L_{\rm max},
\ee
where
\be\label{powerfunctionkernels}
\ba
M_0(r) &= \frac{\Phi_0(r)}{r^\alpha},\\
D_l(r) &=\frac{\Phi_{l+1}(r)-\Phi_{l}(r)}{r^\alpha},\\
R_L(r) &=\frac{1-\Phi_L(r)}{r^\alpha},
\ea
\ee

\be\label{powerphil}
\Phi_l(r) = \int_0^{r} x^{\alpha-1}\phi_l(x)dx = \frac{1}{c_0} \int_0^{r/\cR_l}x^{\alpha-1}\psic(x)dx,
\ee  
\be
\phi_l(x) = \frac{1}{\cR_l^{\alpha} c_0}\psic\left(\frac{x}{\cR_l}\right),
\quad
c_0 = \int_{0}^1 x^{\alpha-1}\psic(x)dx.
\ee
In order to verify that the above construction satisfies the properties (a)-(c) in
\cref{sec:betterkernelsplitting}, we note that $\psic$ is a smooth, even function.
Substituting the Taylor expansion of $\psic$ into \cref{powerphil}, we obtain
\be
\Phi_l(r)=\int_0^{r/\cR_l}x^{\alpha-1}\sum_{i=0}^\infty C_i x^idx
=\sum_{i=0}^\infty \frac{C_i}{i+\alpha} \left(\frac{r}{\cR_l}\right)^{\alpha+i}.
\ee
Thus, the windowed and difference kernels in \cref{powerfunctionkernels}
are smooth at the origin. The compactness of the residual kernel follows
from the fact that $\Phi_L(r) = 1$ for $r\ge \cR_L$.

\begin{remark}
  The above construction works well for $\alpha\in (0,2]$ in the sense that
    the Fourier expansion length of the windowed and difference kernels defined in
    \cref{powerfunctionkernels} does not differ much from that for the $1/r$ kernel.
    When $\alpha$ increases, it is better to use PSWFs with values of the 
    parameter $c$ different from those listed in \cref{tablepswfparameters}. 
    This is because the parameter values there are chosen specifically for the 
    $1/k^2$ kernel in Fourier space. 
    As $\alpha$ increases, the Fourier transform of $1/r^\alpha$ decreases more slowly.
    We have not carried out any detailed analysis on this.
\end{remark}

\subsubsection{An alternative kernel splitting for the $1/r^2$ kernel in three dimensions}
\label{sl3dkernelsplitting}

The kernel $K(r) = \frac{1}{r^2}$ is the Green's function for the square root of the Laplacian
in three dimensions, with Fourier transform $\widehat{K}(\bk) = 2\pi^2/k$ (see
\cref{powerfunction} with $\alpha=2$ and $d=3$). In this section, we provide an alternative
route to kernel splitting that relies on knowing $\widehat{K}$ rather than
using \cref{powerfunctionpswfsplitting} with $\alpha=2$. The resulting scheme has more or
less the same performance - we introduce it here to show that there are many effective routes
compatible with the \acron framework.
The only subtle issue when manipulating the kernel in the Fourier domain is to
ensure that the difference and residual kernels have the appropriate
compact support in physical space. To be more precise, we write  
\be\label{sl3dpswfsplitting}
\frac{1}{r^2}  = M_0(r) + \sum_{l=0}^{L-1} D_l(r) + R_L(r), \quad L=0, 1, \ldots, L_{\rm max},
\ee
where
\be\label{sl3dpswfkernels}
\ba
M_0(r) &= \frac{1-\psic(r/\cR_0)/\psic(c)}{r^2},\\
D_l(r) &=\frac{\psic(r/\cR_{l})-\psic(r/\cR_{l+1})}{\psic(0)r^2},\\
R_L(r) &=\frac{\psic(r/\cR_L)}{\psic(0)r^2}.
\ea
\ee
The corresponding kernels in Fourier space are similar to those in \cref{newkernels}.
For example,
\be
\hat{D}_l(k)=2\pi^2\frac{\int_0^{k \cR_l/c}\psic(x)dx-\int_0^{kc\cR_{l+1}/c}\psic(x)dx}{c_0 k},
\ee
with $c_0 = \int_{0}^1 \psic(x)dx$.

\subsection{Kernel splitting using Gaussians or PSWFs in Fourier space} \label{fsplitpwsf}

As just noted, it can be convenient to carry out kernel splitting 
in Fourier space rather than physical space, 
especially when the singularity of the kernel is difficult
to account for in physical space, but straightforward to handle in the Fourier representation.
For the constant coefficient Green's functions of classic physics, 
kernel splitting in Fourier space is often very convenient since $\widehat{K}$ is typically
dimension-independent when expressed as a function of $k= |\bk|$.

We illustrate this idea by carrying out kernel splitting for the Yukawa kernel 
$G_{\rm Y}(r)$ using PSWFs. From \cref{yukawakernelfouriertransform}, the Fourier transform
of the Yukawa kernel is $\hat{G}_{\rm Y}(\bk)=\frac{1}{k^2+\lambda^2}$ in any dimension.
From  \cref{generalkernelsplitting} and \cref{gaussiankernel}, we obtain
the following kernel splitting using Gaussians in Fourier space:
\be\label{yukawagaussiankernelsplit}
\hat{G}_{\rm Y}(k)=\frac{1}{k^2+\lambda^2}
=\hat{M}_{0}(k)+\sum_{l=0}^{L-1}\hat{D}_l(k)+\hat{R}_L(k), \quad L=0,1,\ldots,L_{\rm max},
\ee
with
\be\label{yukawagaussiankernels}
\ba
\hat{M}_{0}(k)&=\frac{e^{-(k^2+\lambda^2)\sigma_0^2/4}}{k^2+\lambda^2},\\
\hat{D}_l(k)&=\frac{e^{-(k^2+\lambda^2)\sigma_{l+1}^2/4}-e^{-(k^2+\lambda^2)\sigma_l^2/4}}
    {k^2+\lambda^2},\\
\hat{R}_L(k)&=\frac{1-e^{-(k^2+\lambda^2)\sigma_L^2/4}}
    {k^2+\lambda^2}.
\ea
\ee
where the $\sigma_l$ are defined in \eqref{deltavalue}.
The fact that the difference and residual kernels are compactly supported to precision 
$\epsilon$ follows from the observation that
the Fourier transforms of the difference and residual kernels
can be extended to entire functions of the complexified argument $k$ and the duality
of the Fourier transform~\cite[Theorem 1 on pp. 30]{trefethen2000}.

Using PSWFs instead, we have
\be\label{yukawaPSWFkernelsplit}
\hat{G}_{\rm Y}(k)=\frac{1}{k^2+\lambda^2}
=\hat{M}^c_{0}(k)+\sum_{l=0}^{L-1}\hat{D}^c_l(k)+\hat{R}^c_L(k), \quad L=0,1,\ldots,L_{\rm max},
\ee
with
\be\label{yukawaPSWFkernels}
\ba
\hat{M}^c_{0}(k)&=\frac{\psic\left(\sqrt{k^2+\lambda^2}\cR_0/c\right)}{\psic(0)(k^2+\lambda^2)},\\
\hat{D}^c_l(k)&=\frac{\psic\left(\sqrt{k^2+\lambda^2}\cR_{l+1}/c\right)
  -\psic\left(\sqrt{k^2+\lambda^2}\cR_l/c\right)}
    {\psic(0)(k^2+\lambda^2)},\\
\hat{R}^c_L(k)&=\frac{1-\psic\left(\sqrt{k^2+\lambda^2}\cR_L/c\right)/\psic(0)}
    {(k^2+\lambda^2)}.
\ea
\ee

Since $\psic$ is an even, entire function, the Fourier transforms of the difference and residual
kernels can be extended as entire functions. Thus, these two kernels are 
compactly supported to precision $\epsilon$. 
The kernel splitting in \cref{yukawaPSWFkernelsplit} is independent of
dimension and works for the case $\lambda=0$ as well. In three dimensions, the expression of
$\hat{D}^c_l(k)$ in \cref{yukawaPSWFkernels} reduces to \cref{differencekernelft3}
for the $1/r$ kernel when $\lambda=0$. In two dimensions, it provides an efficient
kernel splitting for $K_0(\lambda r)$ and $\log(r)$ using PSWFs. Note that
it is not so obvious how to account for the logarithmic singularity
if one tries to carry out kernel splitting in physical space directly.

\subsection{Handling singularities in Fourier space}

When we use either Gaussians or PSWFs for kernel splitting, it is easy to 
enforce rapid and controlled decay of the Fourier transform
of the mollified and difference kernels. 
Since the difference kernels are compactly supported in physical space, they  
are smooth in Fourier space (by the Paley-Wiener theorem)
and the trapezoidal rule is spectrally accurate.
That is not the case for the mollified kernel, which may be slowly decaying. In the
nonoscillatory case, this leads to a singularity at the origin.
In the case of the Yukawa kernel, one might expect that the exponential decay leads
to the same result - and rapid convergence using the trapezoidal rule. Even
though both $\hat{M}_{0}^c(k)$ and $\hat{R}^c_L(k)$ in \cref{yukawaPSWFkernels} 
are formally smooth, however, they are nearly singular at the origin when $\lambda$ is small.
The same issue arises with the residual kernel as well, noting that
$\hat{R}^c_L(k)$ is simply the difference between the Fourier transform
of the original kernel and that of the mollified kernel at level $L$. 
Thus, as discussed in the setting of the $1/r$ kernel, it is convenient to construct
a {\em windowed} kernel which matches the mollified kernel over the domain of interest but 
is rapidly attenuated to zero so that a smooth quadrature can be applied in Fourier space. 

We have done this explicitly for the Laplace kernel in three dimensions in
\cref{sec:l3dgaussianks}, where the Fourier transform~\cref{windowedkernelft} of the windowed
kernel is simply the product of a Gaussian and the Fourier transform of the truncated $1/r$
kernel. The approached used there works for all kernel splittings we have discussed so far.
To see this, suppose that
the Fourier transform of the original kernel is $\hat{K}(k)$, and that the
Fourier transform of the mollified kernel is $\hat{G}_l(k) \hat{K}(k)$, where
the corresponding function $G_l(r)$ in physical space is compactly supported
to precision $\epsilon$ in a ball of radius $\cR_l$. Let the windowed kernel be defined by
\be
\widehat{W}_l(k) = \hat{G}_l(k) \hat{T}(k),
\ee
where
\be
\hat{T}(k) = 
\int_{\mathbb{R}^d}e^{-i\bk\cdot\x}K(\x)\chi \left( \frac{|\x|}{((1+\sqrt{d})\cR_l)} \right) \, d\x
=\int_{|\x|\le(1+\sqrt{d})\cR_l}e^{-i\bk\cdot\x}K(\x)d\x.
\ee
Here, $\chi(r)=1$ for $0 \leq r \leq 1$, $\chi(r)=0$ for $r > 1$, and thus, 
$\hat{T}(k)$ is simply the Fourier transform of the truncated kernel introduced in
in \cite{vico2016jcp}.
Then
\be\label{kernelequivalence}
W_l(r) = M_l(r), \quad r\le \sqrt{d}\cR_l.
\ee
In order to show this, we make use of the convolution theorem:
\be
\ba
M_l(\x)&=\int_{\mathbb{R}^d} K(\y) G_l(\x-\y)d\y\\
&=\int_{|\y|\le (1+\sqrt{d})\cR_l}K(\y) G_l(\x-\y)d\y,
\ea
\ee
where we have used the fact that $G_l(\x-\y)=0$ for $\y$ outside the ball of radius
$(1+\sqrt{d})\cR_l$ when $|\x|\le \sqrt{d}\cR_l$, since $G_l(r)$ is compactly supported in a ball
of radius $\cR_l$ to the desired precision. Likewise,
\be
\ba
W_l(\x)&=\int_{\mathbb{R}^d} T(\y) G_l(\x-\y)d\y\\
&=\int_{\mathbb{R}^d} K(\y)\chi(|\y|/((1+\sqrt{d})\cR_l)) G_l(\x-\y)d\y\\
&=\int_{|\y|\le (1+\sqrt{d})\cR_l} K(\y) G_l(\x-\y)d\x.
\ea
\ee
Since $W_l(r) = M_l(r)$ over the range of interest but compactly supported,
simple trapezoidal quadrature can be applied in Fourier space with spectral accuracy.

\section{Numerical results}

We have implemented the algorithms described above in Fortran. 
The software was compiled using the Intel compiler and linked with the Intel MKL library,
with all experiments run in single-threaded mode on a 3.30GHz Intel(R) Xeon(R) Gold 6234 CPU.
We first consider the \acron approach in the discrete setting, for \eqref{discretesum}.

\subsection{The ``point code": fast transforms using the discrete \acron framework}

We begin with the Laplace kernel $1/r$ in three dimensions, since it is such a well-studied
problem, and there are numerous open source libraries for this task, including the FMM3D and 
PVFMM libraries.

\subsubsection{The 3D Laplace kernel}

In our implementation, we set the subdivision parameter $n_s$ (the maximum number of points 
in a leaf node) to 
$280$ for precisions $\epsilon=10^{-3}, 10^{-6}$ and to $800$ for
$\epsilon=10^{-9}, 10^{-12}$. For the PVFMM,
the multipole expansion order is set to $4, 7, 11, 15$ for $\epsilon=10^{-3},
10^{-6}, 10^{-9}, 10^{-12}$, respectively.
\cref{l3dpunif} shows the total time in seconds when the points are uniformly
distributed in the unit box with the number of points
$N=m \cdot 10^6$ for $m=1,\ldots,8$,
and \cref{l3dpadap} shows the total time in second when the points are distributed
on a sphere of radius $0.45$ centered at the origin.
For low accuracy ($\epsilon=10^{-3}$), the time
for building the tree and sorting the points to create the adaptive data structure
takes roughly one-third of the total time. 

In \cref{l3dpointpps}, we show the average throughput of these schemes. Note that 
the \acron performance is very close to that of the PVFMM, 
while FMM3D is slower by a modest factor. We conjecture that for the $1/r$ kernel, 
a fully optimized scheme that uses features from each of FMM3D, PVFMM and \acron 
may be able to do significantly better.

\begin{remark}
  For direct interactions (as in PVFMM and FMM3D), we make use of 
  SIMD accelerated kernel evaluation for $1/r$ (on a single core),
  using the Scientific computing template library (SCTL)~\cite{sctl}.
\end{remark}

\begin{figure}[!ht]
\centering
\includegraphics[height=40mm]{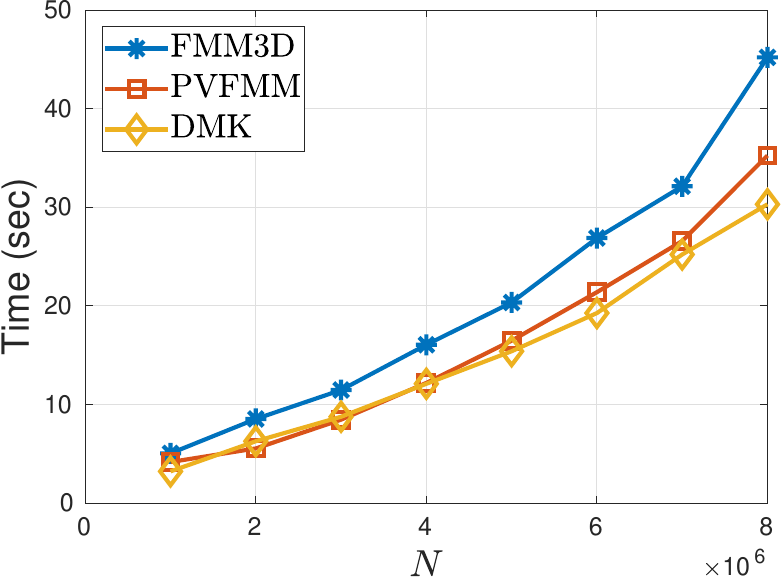}
\hspace{0.4in}
\includegraphics[height=40mm]{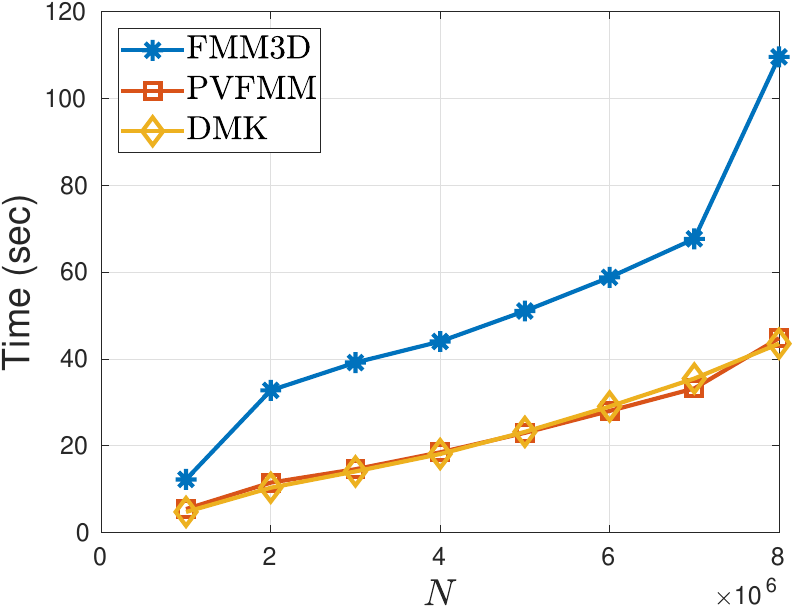}

\vspace{5mm}

\includegraphics[height=40mm]{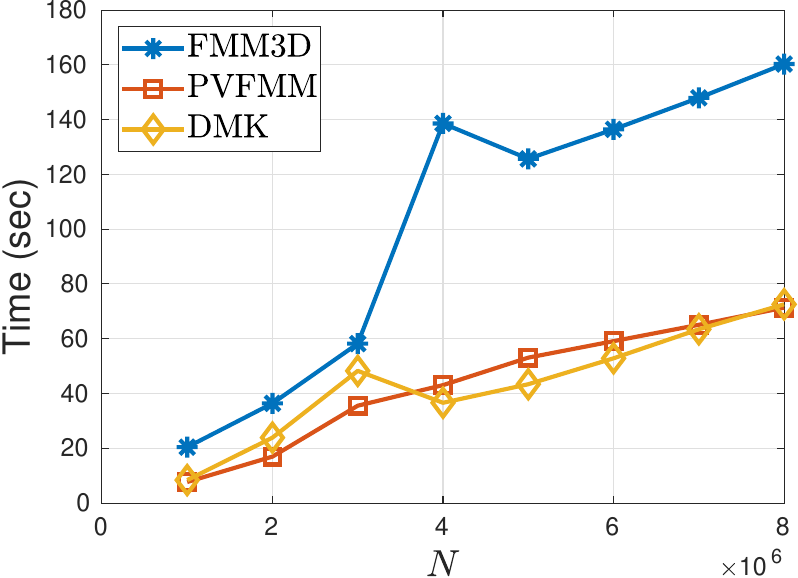}
\hspace{0.4in}
\includegraphics[height=40mm]{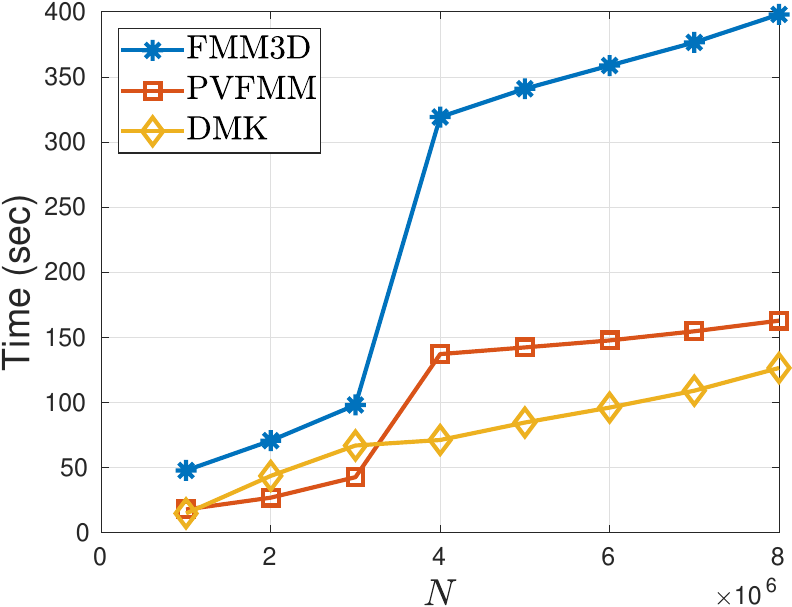}
\caption{\sf Linear scaling of \acron for the 3D Laplace kernel
  and comparison with the FMM3D and PVFMM libraries for various 
  precisions $\epsilon$. In these figures, the points are uniformly distributed
  in the unit box. The $x$-axis indicates the total number of points and
  the $y$-axis the total time in seconds. 
  Top left: $\epsilon=10^{-3}$;  top right: $\epsilon=10^{-6}$.
  Bottom left: $\epsilon=10^{-9}$;  bottom right: $\epsilon=10^{-12}$.}
\label{l3dpunif}
\end{figure}

\begin{figure}[!ht]
\centering
\includegraphics[height=40mm]{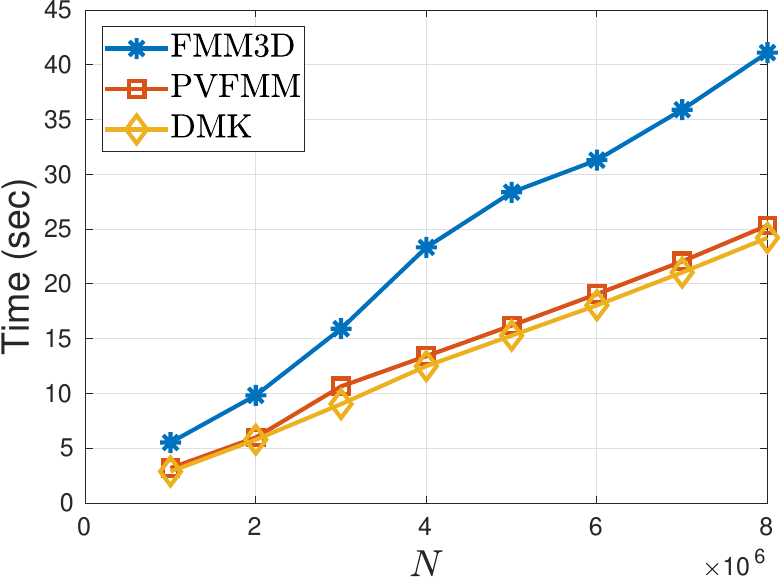}
\hspace{0.4in}
\includegraphics[height=40mm]{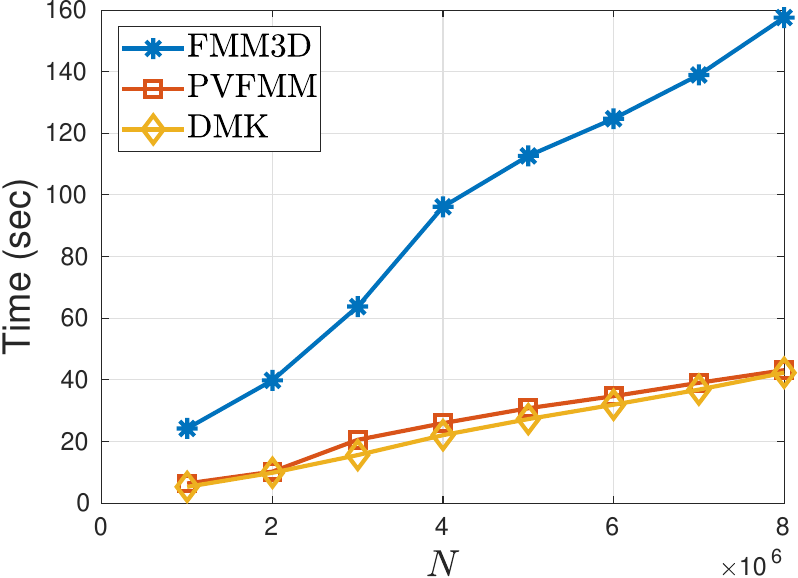}

\vspace{5mm}

\includegraphics[height=40mm]{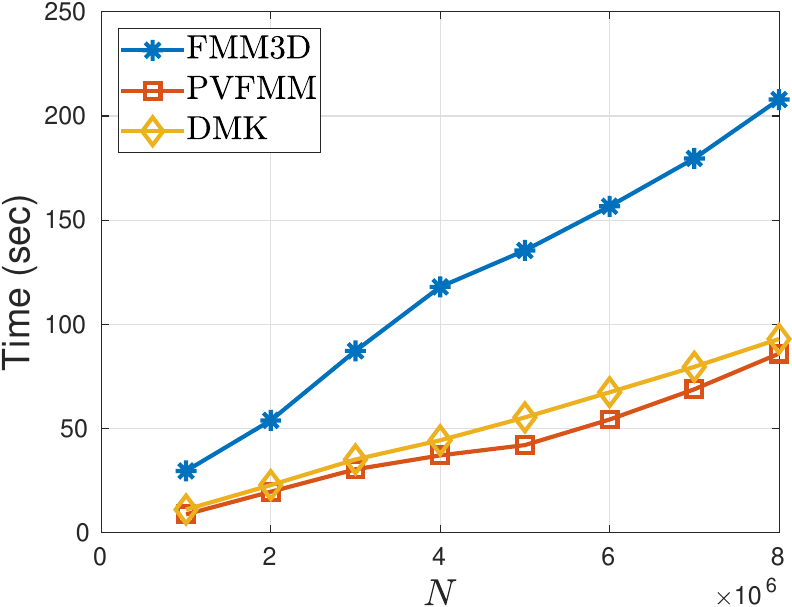}
\hspace{0.4in}
\includegraphics[height=40mm]{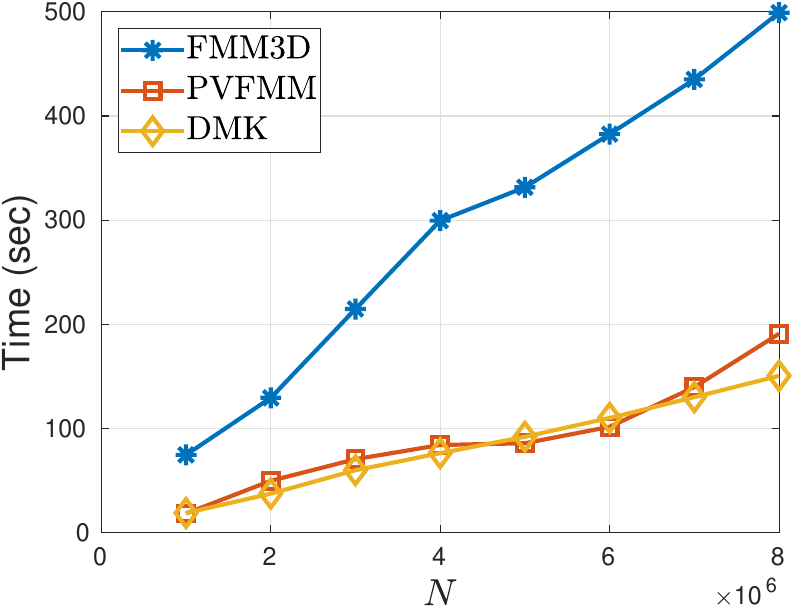}
\caption{\sf Linear scaling of \acron for the 3D Laplace kernel
  and comparison with the FMM3D and PVFMM libraries for various 
  precisions $\epsilon$. In these figures, the points are uniformly distributed
  on a sphere of radius $0.45$. The $x$-axis is the total number of points and
  the $y$-axis the total time in seconds. 
  Top left: $\epsilon=10^{-3}$;  top right: $\epsilon=10^{-6}$.
  Bottom left: $\epsilon=10^{-9}$;  bottom right: $\epsilon=10^{-12}$.}
\label{l3dpadap}
\end{figure}

\begin{figure}[!ht]
\centering
\includegraphics[height=50mm]{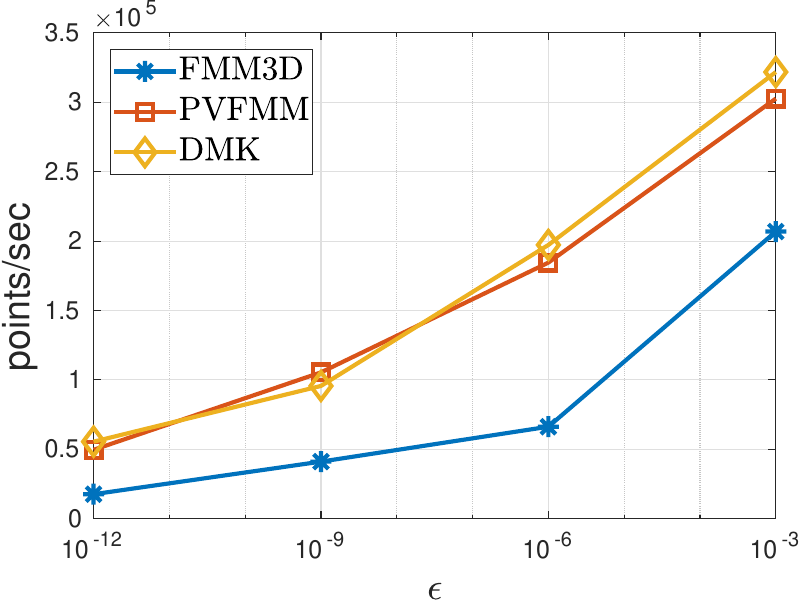}
\caption{\sf Average throughput of \acron for the 3D Laplace kernel
  and comparison with the FMM3D and PVFMM libraries. 
  The $x$-axis indicates the prescribed precision $\epsilon$, and
  the $y$-axis indicates the throughput measured in $10^5$ points per second.} 
\label{l3dpointpps}
\end{figure}

\subsubsection{The 3D kernels for the square-root of the Laplacian and the Yukawa operator}

In three dimensions,
the kernel for the square-root of the Laplacian is $1/r^2$ and for the Yukawa operator
is $e^{-\lambda r}/r$. 
Since SIMD-accelerated evaluation for $1/r^2$ is straightforward, we set
$n_s$ to have the same value as for the 
Laplace kernel. Since the Yukawa kernel is not scale-invariant,
acceleration is more difficult, we reduce the magnitude of $n_s$ to adjust the 
balance between far field and local work. 
We set $n_s = 40, 80, 300, 600$ for $\epsilon = 10^{-3}, 10^{-6}, 10^{-9}, 10^{-12}$, 
respectively. In \Cref{sly3dtimingresults}, we see the linear scaling performance 
of \acron for these two kernels. In \cref{sly3dthroughput}, we show the throughput
in the \acron framework. Note that the throughput for the square-root
of the Laplacian is similar to that for the Laplacian, while the throughput
of the 3D Yukawa kernel is significantly worse -  due almost entirely to the increased
expense of direct kernel evaluation.

\begin{figure}[!ht]
\centering
\includegraphics[height=40mm]{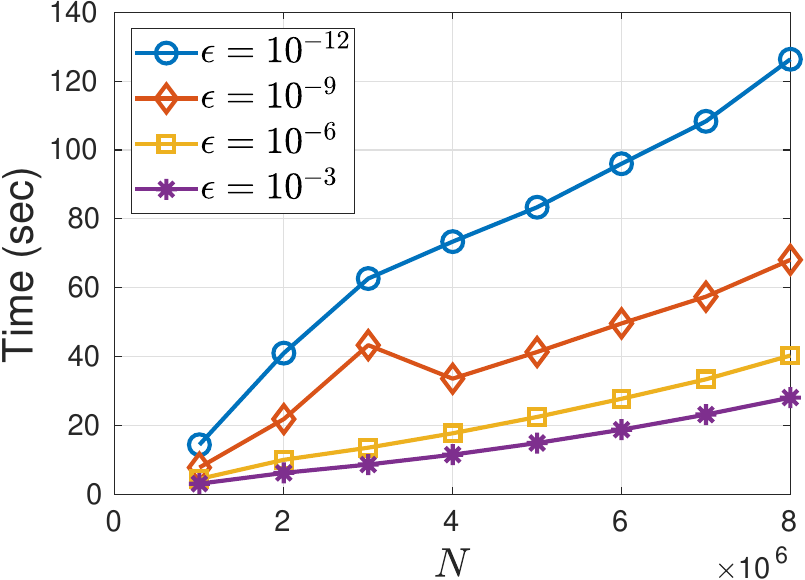}
\hspace{0.4in}
\includegraphics[height=40mm]{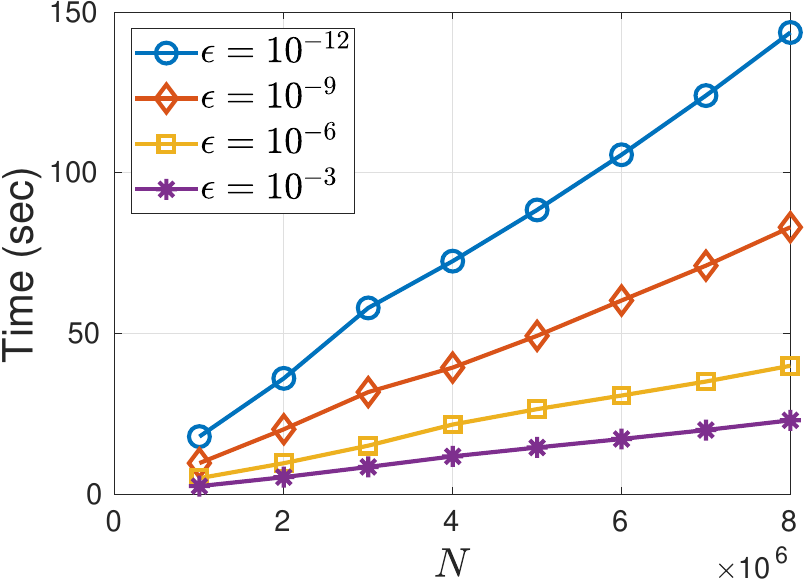}

\vspace{5mm}

\includegraphics[height=40mm]{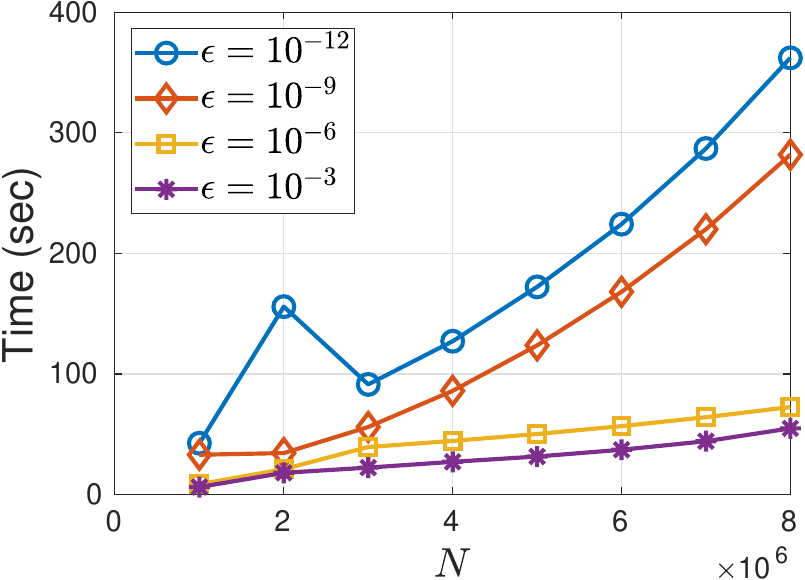}
\hspace{0.4in}
\includegraphics[height=40mm]{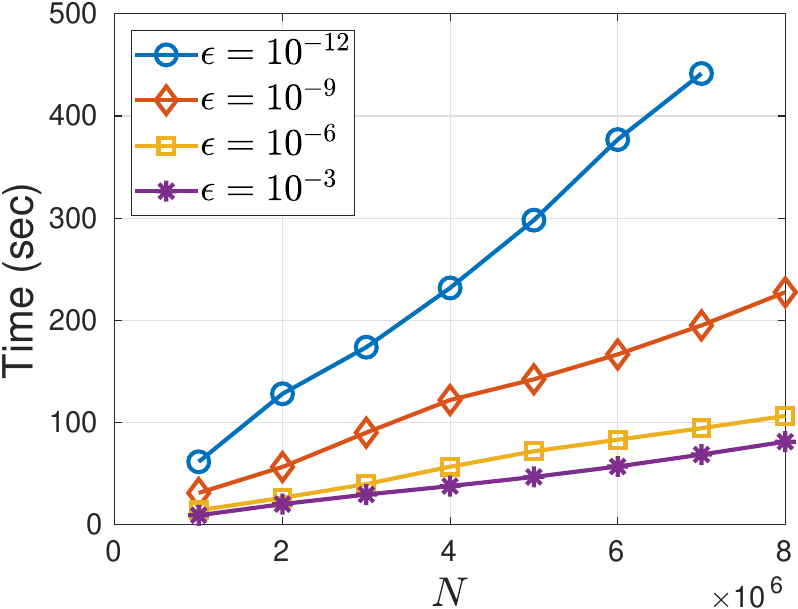}
\caption{\sf Timing results of \acron. Top: the kernel for
  the square-root of the Laplacian in 3D; bottom: the kernel for the Yukawa operator in 3D
  with $\lambda=6$. Left: data for uniform distribution of points; 
  right: data points distributed on a sphere.}
\label{sly3dtimingresults}
\end{figure}

\begin{figure}[!ht]
\centering
\includegraphics[height=40mm]{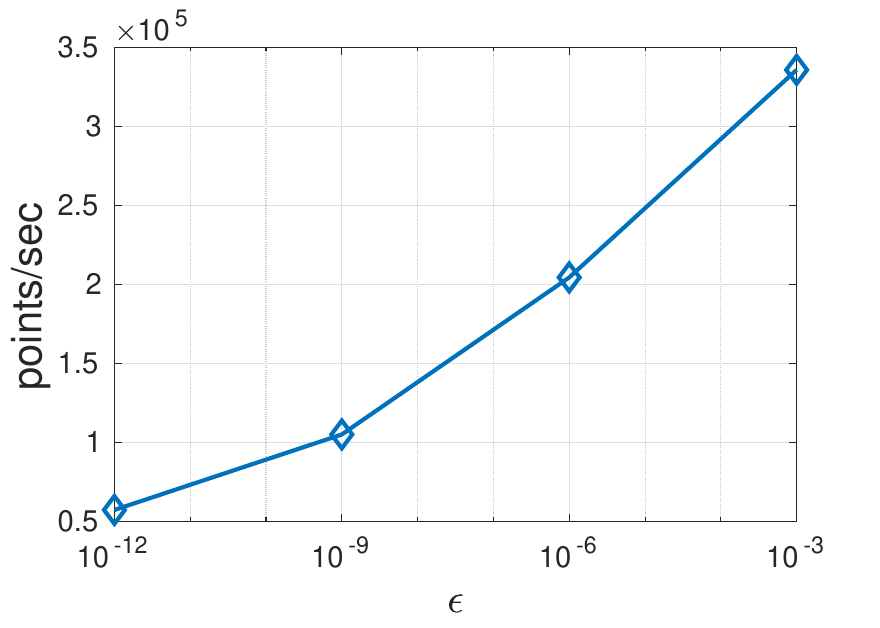}
\hspace{0.4in}
\includegraphics[height=40mm]{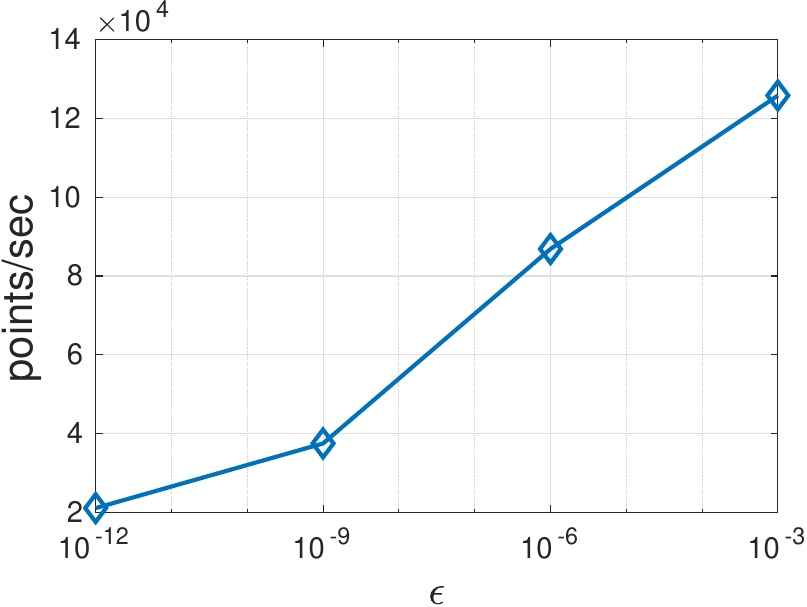}
\caption{\sf Average throughput of \acron. Left: the kernel for
  the square-root of the Laplacian in 3D; right: the kernel for the Yukawa operator in 3D 
  with $\lambda=6$.} 
\label{sly3dthroughput}
\end{figure}

\subsubsection{Kernels for the Laplace operator, 
the square-root of the Laplace operator and the Yukawa operator in 2D}

In two dimensions, the kernels for
Laplacian, the square-root of the Laplacian and the Yukawa operator are $\log(r)$, 
$1/r$, and $K_0(\lambda r)$, respectively. Both $\log(r)$ and $1/r$
easy to rescale and admit straightforward SIMD-accelerated fast kernel evaluation.
In our implementation, we set
$n_s = 120, 120, 160, 160$ for $\epsilon = 10^{-3}, 10^{-6}, 10^{-9}, 10^{-12}$, respectively.
Since the 2D Yukawa kernel is not scale-invariant, we have not yet developed a
SIMD-accelerated kernel routine. Thus, we set $n_s = 30, 30, 45, 45$ for
$\epsilon = 10^{-3}, 10^{-6}, 10^{-9}, 10^{-12}$, respectively.
In \Cref{2dtimingresults}, we illustrate the linear scaling of \acron
for these three kernels. \cref{2dthroughput} shows the average throughput.

\begin{remark}
  We make use of the fast kernel evaluation routine for $\log(r)$ from the C++ vector
  class library~\cite{vcl}. 
\end{remark}

In both two and three dimensions, the principal reason the performance
for the Yukawa kernel is slower than that for the Laplacian or its square root is that
we lack a SIMD-accelerated direct kernel evaluation routine.
In \cref{timingbreakup}, we show a breakdown of the timing for various
components of the algorithm.
Note that the cost of translating plane-wave representations is basically the same
for all kernels, so long as the tree has the same depth, and that the increase
in the total time required for the Yukawa kernel is due to the ten-fold increase
in the cost of kernel evaluation ($t_{\rm direct}$),
when $n_s$ is set to the same number as for the other two kernels.
By shrinking $n_s$, which cause one more level of refinement,
the total cost for the Yukawa kernel is reduced slightly.

\begin{table}[t]
  \caption{Timing results for various stages of the \acron algorithm for the Laplace kernel, 
    the kernel for the square-root of the Laplacian and the Yukawa kernel
    in two and three dimensions. The number of source points is four million on either a circle
    (2D) or a sphere (3D).
    The requested precision is $\epsilon=10^{-6}$. $n_s$ is the maximum number of particles in a
    leaf box. $t_{\rm tree}$ is the time for constructing
    the level-restricted tree from the particle distribution. $t_{\rm Fourier}$ is the time
    for translating plane-wave expansions. $t_{\rm direct}$ is the time for computing
    direct interactions using the residual kernel. $t_{\rm total}$ is the total computational time.
}
\sisetup{
  tight-spacing=false
}
\centering
\begin{tabular}{lccccc}
\toprule
{Kernel}  &  ${n_s}$ & ${t_{\rm tree}}$ & ${t_{\rm Fourier}}$ & ${t_{\rm direct}}$ & ${t_{\rm total}}$  \\
\midrule
\text{Two dimensions}\\
\midrule
$\log(r)$        & 120   &    4.98  &    0.91    &    2.11  &    8.58  \\
$1/r$            & 120   &    4.99  &    0.91    &    1.48  &    7.96  \\
$K_0(\lambda r)$ & 120   &    5.03  &    0.90    &    22.3  &    28.8  \\
$K_0(\lambda r)$ &  30   &    8.66  &    2.91    &    5.68  &    19.1  \\
\midrule
\text{Three dimensions}\\
\midrule
$1/r$            & 280   &    4.81  &    7.06    &    6.99  &    21.3  \\
$1/r^2$          & 280   &    4.82  &    8.04    &    6.23  &    21.8  \\
$e^{-\lambda r}/r$ &  280   &    5.33  &    7.17    &    63.4  &    78.3  \\
$e^{-\lambda r}/r$ &   80   &    11.0  &    18.3    &    19.6  &    56.4  \\
\bottomrule
\end{tabular}
\label{timingbreakup}
\end{table}

\begin{figure}[!ht]
\centering
\includegraphics[height=40mm]{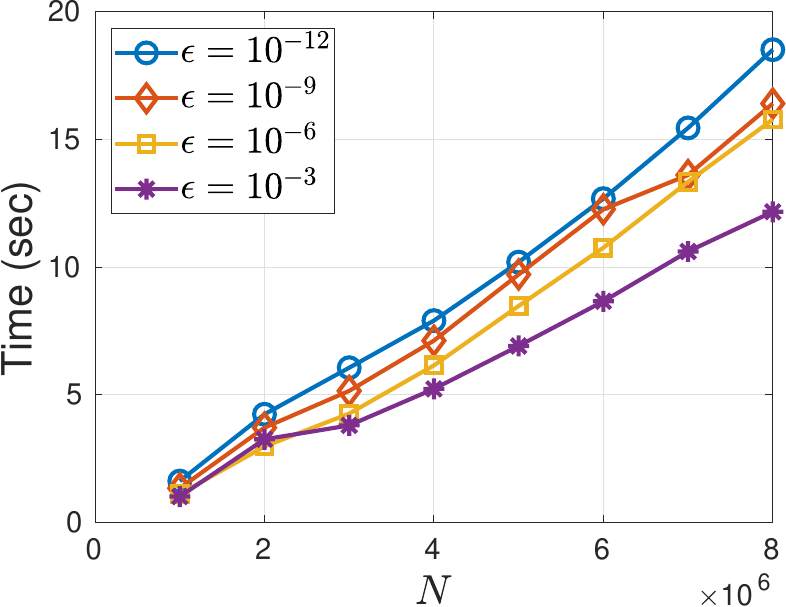}
\hspace{0.4in}
\includegraphics[height=40mm]{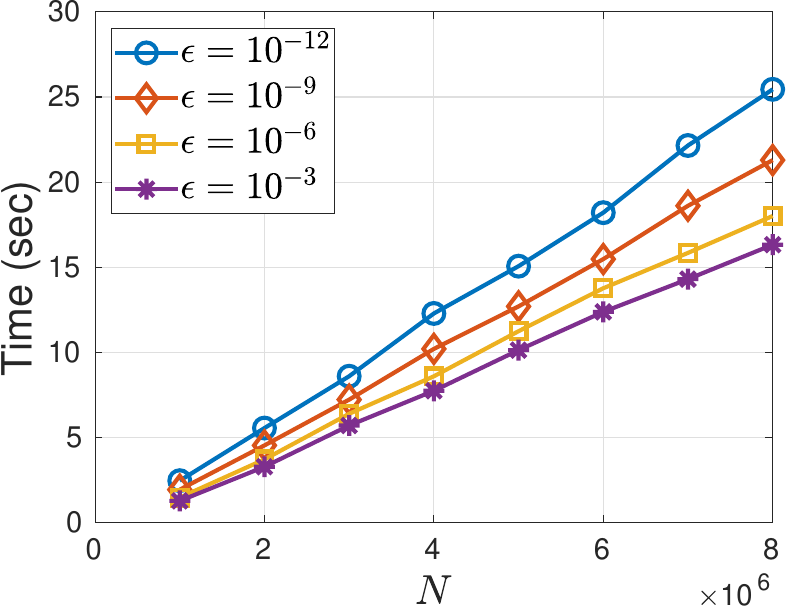}

\vspace{5mm}

\includegraphics[height=40mm]{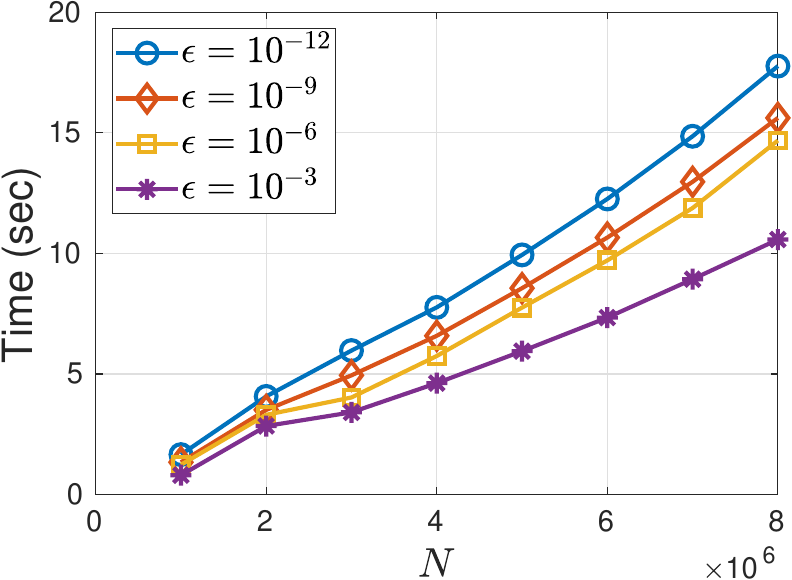}
\hspace{0.4in}
\includegraphics[height=40mm]{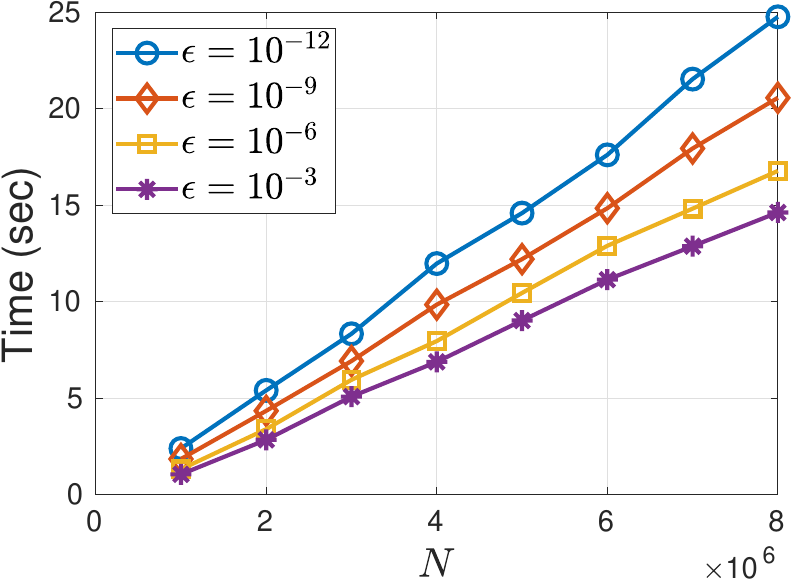}

\vspace{5mm}

\includegraphics[height=40mm]{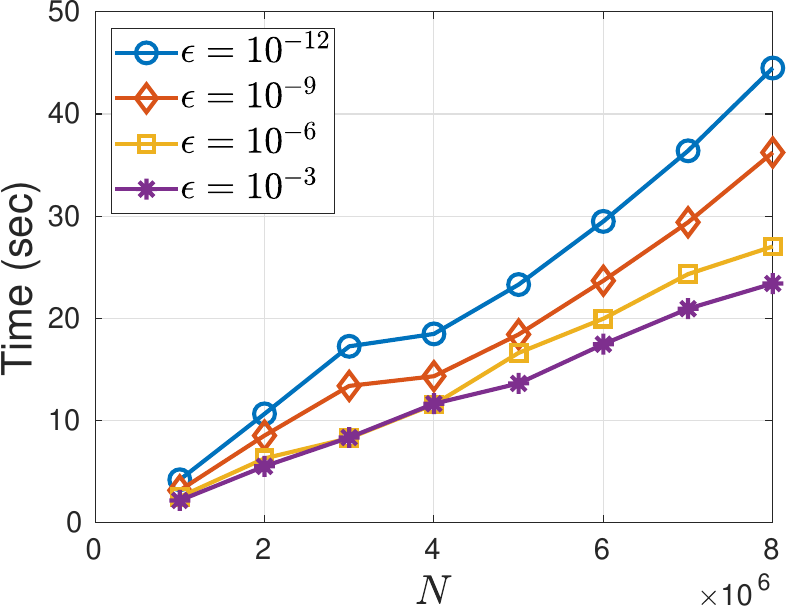}
\hspace{0.4in}
\includegraphics[height=40mm]{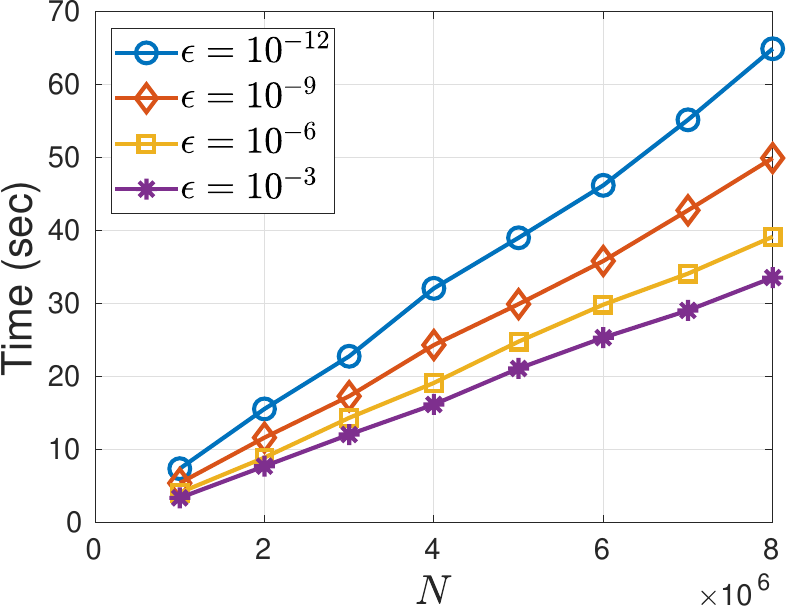}

\caption{\sf Timing results of \acron for various 2D kernels.
  Top: the Laplace kernel; middle: the kernel for
  the square-root of the Laplacian; bottom: the Yukawa kernel with
  $\lambda=6$. Left: uniform distribution; right: 
  points distributed on a circle.}
\label{2dtimingresults}
\end{figure}

\begin{figure}[!ht]
\centering
\includegraphics[height=30mm]{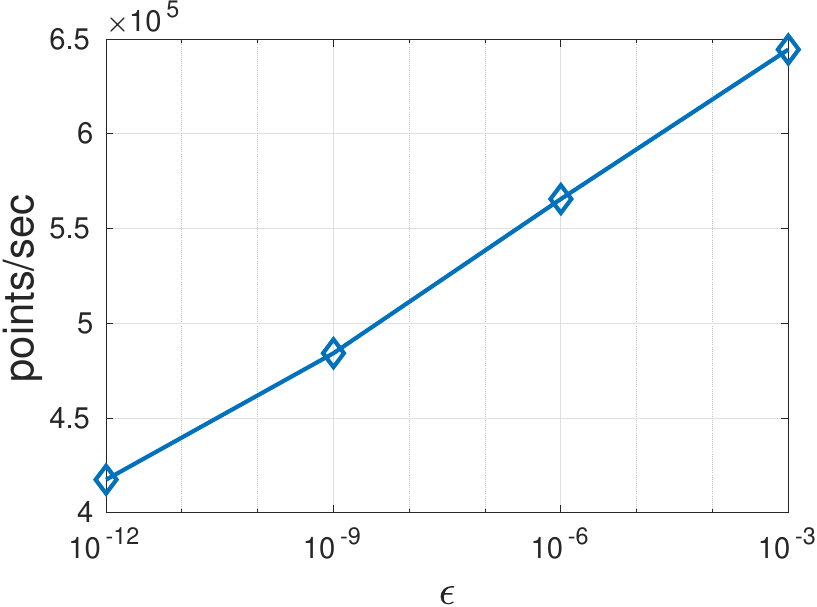}
\hspace{0.1in}
\includegraphics[height=30mm]{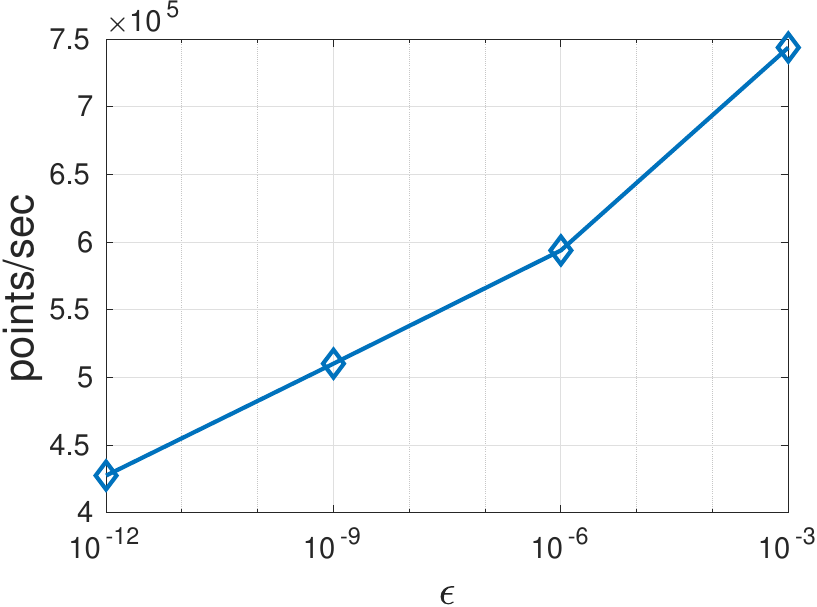}
\hspace{0.1in}
\includegraphics[height=30mm]{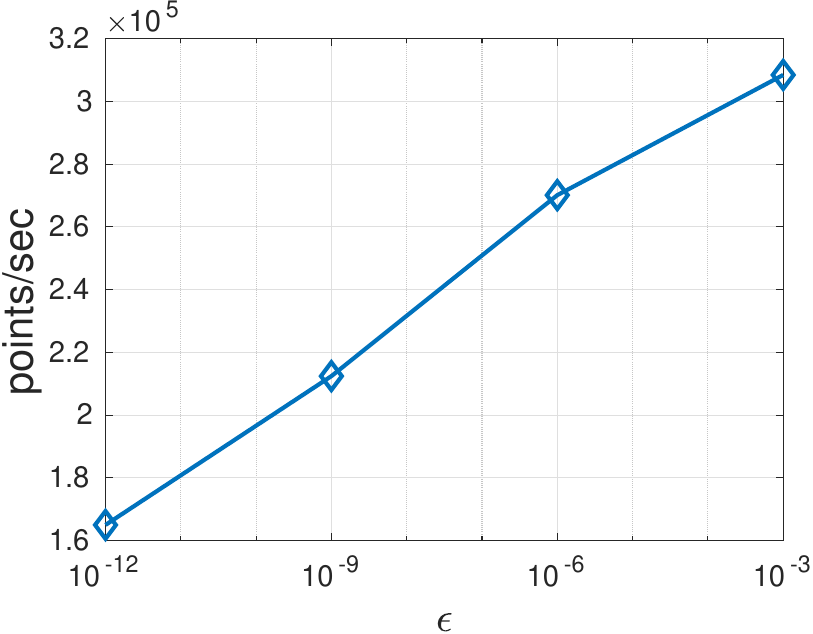}
\caption{\sf Average throughput of \acron for 2D kernels. Left: the
  Laplace kernel; middle: the kernel for the square-root of the Laplacian;
  right: the Yukawa kernel with $\lambda=6$.} 
\label{2dthroughput}
\end{figure}

\subsection{The ``box code": fast transforms using the continuous \acron framework}

The performance of \acron for computing volume integrals of the form 
\eqref{volumepotential} is much less dependent on the specific kernel, once the 
requested precision and polynomial approximation order $k$ for the input density
are fixed.
\Cref{boxcodepps} shows the average throughput
of the algorithm for six kernels: the Laplace kernels in two and three dimensions 
($\log(r)$ and $1/r$), the Yukawa kernels in two the three dimensions
($K_0(\lambda r)$ and $e^{-\lambda r}/r$), and the kernels for the square-root 
of the Laplacian in two and three dimensions
($1/r$ and $1/r^2$). In the two-dimensional setting, we fixed 
$k$ to be $16$. For three dimensions, we set $k=16$ for three and six digits
of accuracy and $k=20$ for nine and twelve digits of accuracy. The throughput of PVFMM
for the 3D Laplace kernel is generated with multipole expansion orders $p=4, 8, 12, 15$
and polynomial expansion orders $k=6, 10, 14, 16$ for $3, 6, 9, 12$ digits of accuracy,
respectively.

For our experiments involving the Laplace and Yukawa kernels in three dimensions, 
we used an analytic solution given as the sum of two Gaussians:
\be\label{uexact3d}
u_{\rm exact}(\x)=\frac{1}{\pi\delta^{d/2}}e^{-|\x-\x_1|^2/\delta}-\frac{1}{2\pi\delta^{d/2}}e^{-|\x-\x_2|^2/\delta}
\ee
with 
$\x_1=(0.1,0.02,0.04)$ and $\x_2=(0.03,-0.1,0.05)$.
We calculate the average throughput with
$\delta=4\cdot 10^{-3}, 10^{-3}, 10^{-4}, 10^{-5}$.
In the two-dimensional case, we used an analytic solution given by
\be\label{uexact2d}
u_{\rm exact}(\x) = e^{-(|\x|/r_0)^\alpha}
\ee
with $r_0=0.25$. The solution drops sharply beyond the circle of radius $r_0$
(see, for example, \cite{biros2015cicp,pvfmm}).
We calculate the average throughput over the parameter values $\alpha=60, 96, 110, 180$.
The input density is computed explicitly as $\rho(\x)=\Delta u_{\rm exact}(\x)$
and $\rho(\x)=(\Delta+\lambda^2) u_{\rm exact}(\x)$ for the Laplace and
Yukawa kernels, respectively. The adaptive tree is constructed by querying for values of
$\rho(\x)$ until it is resolved to the desired precision.

For the kernels corresponding to the square-root of the Laplacian, the input density
is assumed to be the sum of $40$ Gaussians with centers at $40$ equispaced points
on the circle of radius $0.15$ in two dimensions and variance $\delta$.
The average throughput is calculated over parameter values
$\delta=10^{-5}, 10^{-5}/3, 10^{-5}/9, 10^{-5}/27$.
In three dimensions, the input density is given by \cref{uexact3d}.
A reference solution is then calculated via the integral representation
\be
\frac{1}{r^\alpha}=\frac{1}{\Gamma(\alpha/2)}\int_{-\infty}^{\infty} e^{-r^2 e^t+\alpha t/2} dt
\ee
and the fact that the convolution of two Gaussians is another Gaussian
\be
\int_{\mathbb{R}^d} e^{-a|\x-\y|^2}e^{-b|\y|^2}d\y = \left(\frac{\pi}{a+b}\right)^{d/2}e^{-ab|\x|^2/(a+b)}.
\ee
This leads to
\be\label{uexactsl}
\int_{\mathbb{R}^d} \frac{e^{-|\y-\y_0|^2/\delta}}{|\x-\y|^\alpha}
d\y =\frac{1}{\Gamma(\alpha/2)}
\int_{-\infty}^{\infty} \left(\frac{\pi}{e^t+1/\delta}\right)^{d/2}
e^{-|\x-\y_0|^2/(e^{-t}+\delta)+\alpha t/2} dt.
\ee
Finally, \cref{uexactsl} can be evaluated numerically via the truncated trapezoidal rule,
since the integrand decays exponentially to zero at $\pm \infty$.

\begin{figure}[!ht]
\centering
\includegraphics[height=40mm]{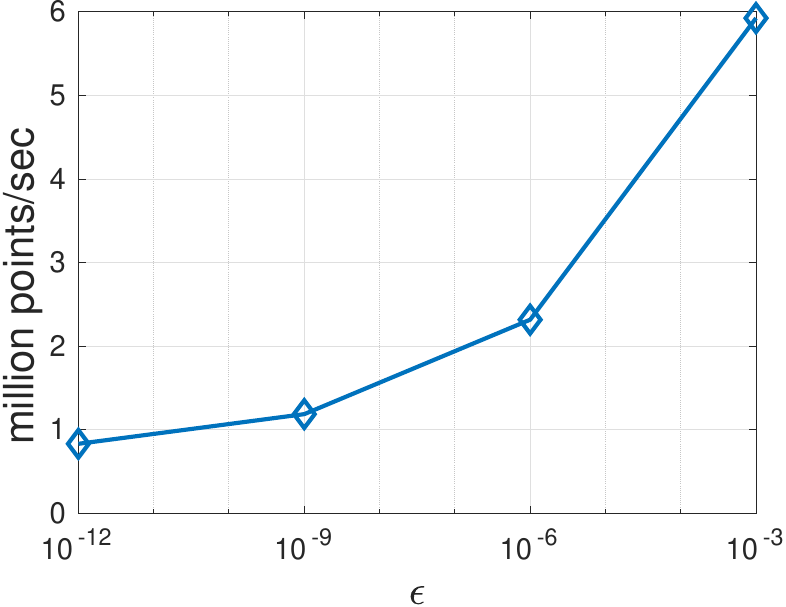}
\hspace{0.4in}
\includegraphics[height=40mm]{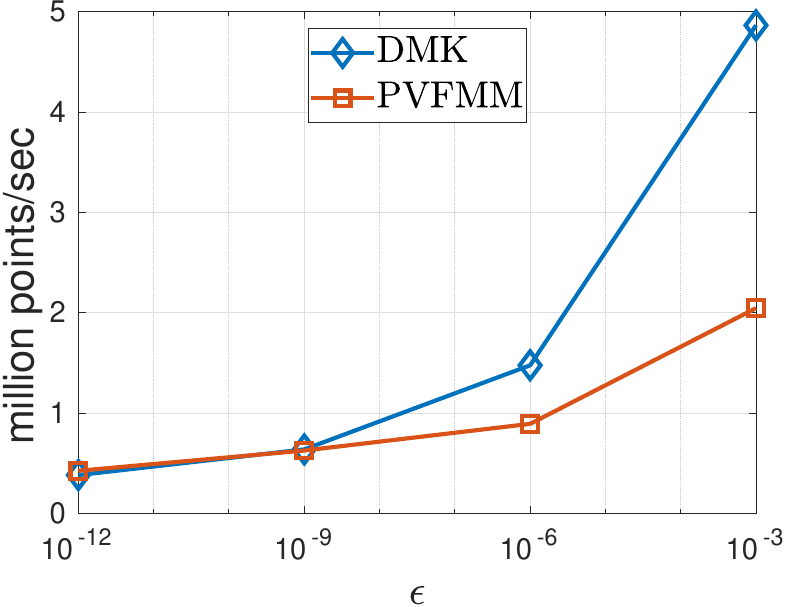}

\vspace{5mm}

\includegraphics[height=40mm]{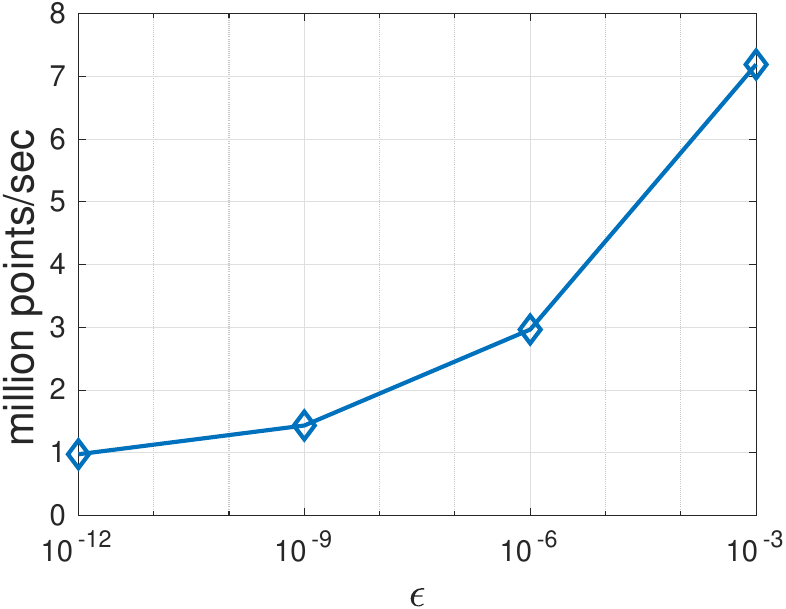}
\hspace{0.4in}
\includegraphics[height=40mm]{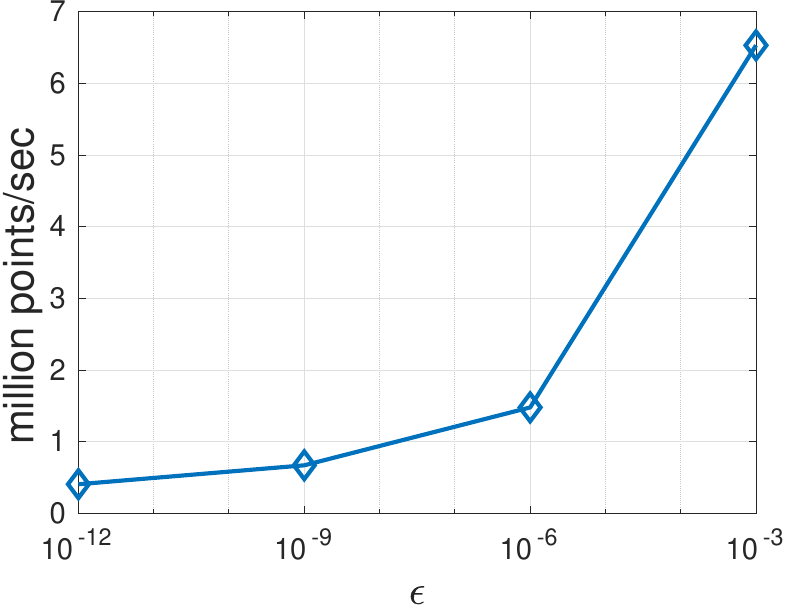}

\vspace{5mm}

\includegraphics[height=40mm]{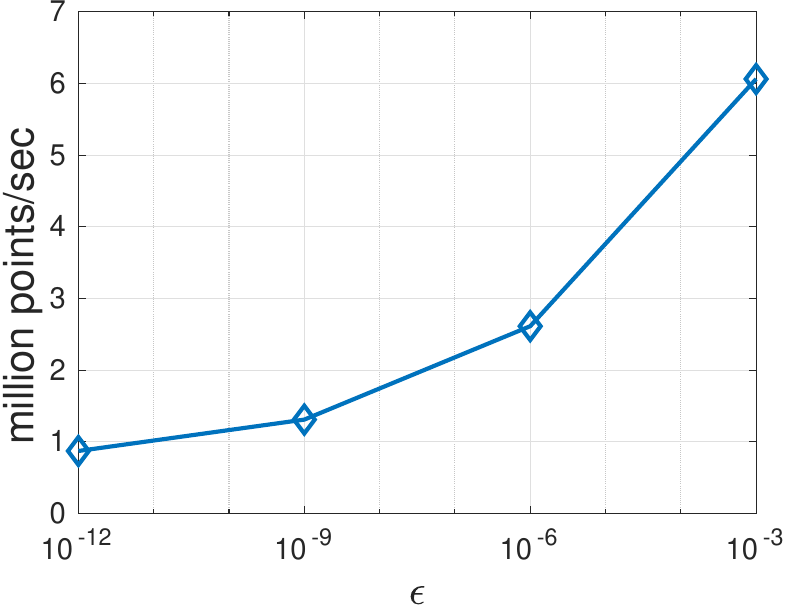}
\hspace{0.4in}
\includegraphics[height=40mm]{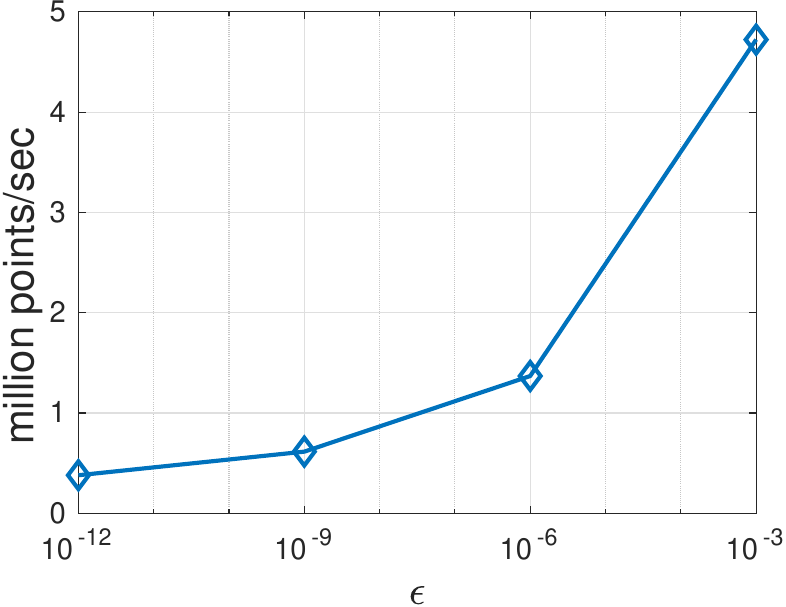}

\caption{\sf Throughput of \acron for various kernels with continuous source
  distributions
  in two and three dimensions. The $x$-axis is the prescribed precision, and
  the $y$-axis is the throughput measured in million points per second. Top: the Laplace
  kernel; middle: the Yukawa kernel with $\lambda=6$; bottom: the square-root Laplace
  kernel. Timings for 2D are on the left. Timings for 3D are on the right.}
\label{boxcodepps}
\end{figure}

The reason that the performance of the box code has such a weak dependence
on the kernel itself is that the direct interactions are accelerated by 
the sum-of-Gaussians approximation of the kernel as well as asymptotic analysis.
Thus the cost of near-field interactions does not depend 
on the cost of evaluating the kernel itself. Furthermore, the cost of direct interactions
goes down dramatically because of separation of variables. (It is also much more suitable
for low-level code optimization across heterogeneous computer architectures.)

\begin{remark}
We have used the sum-of-Gaussians approximation in our
current implementation of the box code, due to the fact that the SOG approximation 
is easily computed on the fly. Further speedup could be obtained
if we carried out kernel-splitting using PSWFs, and used the SOG approximation only 
for the residual kernel. 
We estimate that the throughput would increase by a factor of $1.5-2$.
\end{remark}

\section{Conclusions}

We have presented a new class of methods for the development of discrete and continuous
fast transforms with translation invariant kernels.
It draws from a wide variety of existing schemes: Ewald summation (which exploits
Fourier analysis and diagonalization), multilevel summation (which exploits
kernel splitting and hierarchical localization), and fast multipole methods (which
exploit the use of translation operators). 

The \acron framework starts with a hierarchical splitting of 
the given kernel as a sum of windowed, difference,
and residual kernels. In the execution of the method, all relevant interactions
take place within a box and its nearest neighbors at every level of the grid hierarchy.
This turns out to be much simpler than manipulating the ``interaction lists" that are
essential in the FMM to ensure the well-separateness criterion.
While the rank of the far-field
interaction is increased compared to the FMM, the ability to use
diagonal translation via localized Fourier transforms and acceleration of near-field
interactions using tensor product transforms makes the 
\acron framework competitive with the FMM, even for the Laplace kernel in
three dimensions. More striking is the fact that the performance of DMK-based
transforms is more or less independent of the kernel, so long as it is 
smooth away from the origin in Fourier space, especially for continuous sources. 
In that setting, the near neighbor interactions are computed using either
tensor product transforms or asymptotics, with only one-dimensional
transformation matrices needed (which can be computed on the fly).
For discrete transforms, the cost is more variable, because of the need to
compute near field interactions using the residual kernel. We believe that this is an 
area where significant improvements can be made, including better single core SIMD parallelism.

We have focused in this paper on a broad class of  non-oscillatory kernels, only some of
which were Green's functions for an underlying PDE.
For mildly oscillatory kernels, \acron can be applied without significant
modification. In the truly high-frequency regime, however, a fast algorithm
will require some additional tools:
either 
coupling to the high-frequency FMM (as in \cite{wideband2d,wideband3d}) 
or to a butterfly factorization \cite{butterfly_direct,li-2015,oneil-2010}, or possibly by 
a modification of the \acron framework using adaptive quadrature in the Fourier domain,
as in \cite{beylkin2009jcp}.
We considered only free-space interactions here, but it should be clear that,
in the \acron framework singly, doubly or triply periodic conditions are easy to apply 
since they only involve modification of the Fourier representation at the root of the tree
where the windowed kernel is applied (as in the fast Ewald method of \cite{shamshirgar2021jcp}).
In the continuous case, we have restricted ourselves to functions defined on the unit box $B_0$,
rather than a complex domain $\Omega \subset B_0$. Extensions to this case are being 
developed and will be reported at a later date.

Finally, it is worth noting that, using \acron, we represent the dense kernel matrix as
the sum of matrices, which are low rank but dense at coarse levels but high-rank and sparse
at finer levels. This is in contrast with the hierarchical compression of well-separated
blocks of the matrix, which underlies the FMM,
${\mathcal H}$-matrices and skeletonization-based schemes. The latter provide a route, not just
to fast transforms, but to fast matrix inversion.
This suggests the investigation of \acron as an
alternative route to the design of fast solvers as well.

\section*{Acknowledgments}
The authors would like to thank Alex Barnett, Charles Epstein, and Manas Rachh at the
Flatiron Institute for useful discussions. They would particularly like to
thank Libin Lu and Dhairya Malhotra at the Flatiron Institute for their help
in the implementation of the SIMD vectorization of local interactions. Finally,
they would like to thank Philip Greengard at Columbia University
for providing software for the evaluation of PSWFs.

\bibliographystyle{siam}
\bibliography{journalnames,fmm}

\appendix

\section{Mathematical tools}

\subsection{Hankel transform}
The Hankel transform is defined by the formula
\be
\hat{f}_\nu(k) = \int_0^\infty J_\nu(kr)f(r)rdr,
\ee
where
\be
J_\nu(x)=\frac{\left(\frac{1}{2}x\right)^\nu}{\pi^{\frac{1}{2}}\Gamma\left(\nu+\frac{1}{2}\right)}
\int_0^\pi\cos(x\cos\theta)(\sin\theta)^{2\nu}d\theta
\ee
is the Bessel function of the first kind of order $\nu$~\cite[\S10.9.4]{nisthandbook}.
In particular,
\be
J_0(x)=\frac{1}{\pi}\int_0^\pi e^{ix\cos\theta}d\theta, \quad
  J_{\frac{1}{2}}(x)=\sqrt{\frac{2x}{\pi}}\frac{\sin(x)}{x}.
\ee

\subsection{Radially symmetric functions}
Suppose that $f$ is radially symmetric, i.e., $f(\x)=f(r)$.
Then its Fourier transform $\hat{f}$ is also radially symmetric. Indeed, we have
\be
k^{\frac{d-2}{2}}\hat{f}(k) = (2\pi)^{\frac{d}{2}}\int_0^\infty J_{\frac{d-2}{2}}(kr)r^{\frac{d-2}{2}}f(r)rdr.
\ee
That is, the Fourier transform of a radially symmetric function in $\mathbb{R}^d$ can be
computed via the Hankel transform along the radial direction. In two dimensions,
\be
\hat{f}(k) = 2\pi\int_0^\infty J_{0}(kr)f(r)rdr.
\ee
In three dimensions,
\be\label{rbfft3d}
\hat{f}(k) = 4\pi\int_0^\infty \frac{\sin(kr)}{kr}f(r)r^2dr.
\ee

Some relevant Fourier transform pairs are listed below.
\begin{enumerate}[label=(\alph*)]
\item Gaussian in $\mathbb{R}^d$.
  \be\label{gaussiankernel}
  G(\x)=e^{-\delta^2 r^2}, \qquad \hat{G}(\bk) = \left(\frac{\sqrt{\pi}}{\delta}\right)^{d}
  e^{-k^2/(4\delta^2)}.
  \ee
\item Power function in $\mathbb{R}^d$.
  \be\label{powerfunction}
  f(\x)=\frac{1}{r^\alpha}, \qquad \hat{f}(\bk) = (2\pi)^{\frac{d}{2}}
  \frac{2^{\frac{d-\alpha}{2}}\Gamma\left(\frac{d-\alpha}{2}\right)}
       {2^{\frac{\alpha}{2}}\Gamma\left(\frac{\alpha}{2}\right)}
       \frac{1}{k^{d-\alpha}}.
  \ee
\item Green's function for 
  the Laplace operator ($-\Delta$):
  \be
\label{laplaceker}
  \ba
  G_{\rm L}(\x) &= -\frac{1}{2\pi} \log r,
  & \hat{G}_{\rm L}(\bk) = \frac{1}{k^2}, \quad \x\in\mathbb{R}^2, \\
  G_{\rm L}(\x) &= \frac{1}{4\pi} \frac{1}{r},
  & \hat{G}_{\rm L}(\bk) = \frac{1}{k^2}, \quad \x\in\mathbb{R}^3.
  \ea
  \ee
\item Green's function for 
  the Helmholtz operator ($-\Delta-\omega^2$): 
  \be\label{helmholtzker}
  \ba
  G_{\rm H}(\x) &= \frac{i}{4} H_0^{(1)}(\omega r),
  & \hat{G}_{\rm H}(\bk) = \frac{1}{k^2-\omega^2}, \quad \x\in\mathbb{R}^2, \\
  G_{\rm H}(\x) &= \frac{1}{4\pi} \frac{e^{i\omega r}}{r},
  & \hat{G}_{\rm H}(\bk) = \frac{1}{k^2-\omega^2}, \quad \x\in\mathbb{R}^3.
  \ea
  \ee
  Here $H_0^{(1)}$ is the Hankel function of the first kind of order zero.
\item Green's function for 
  the Yukawa operator ($-\Delta+\lambda^2$):
  \be\label{yukawakernelfouriertransform}
  \ba
  G_{\rm Y}(\x) &= \frac{1}{2\pi} K_0(\lambda r),
  & \hat{G}_{\rm Y}(\bk) = \frac{1}{k^2+\lambda^2}, \quad \x\in\mathbb{R}^2, \\
  G_{\rm Y}(\x) &= \frac{1}{4\pi} \frac{e^{-\lambda r}}{r},
  & \hat{G}_{\rm Y}(\bk) = \frac{1}{k^2+\lambda^2}, \quad \x\in\mathbb{R}^3.
  \ea
  \ee
  Here $K_0$ is the modified Bessel function of the second kind of order zero.
\end{enumerate}

\subsection{Fourier transform of radially symmetric functions in three dimensions}
Consider
\be
F(\x) = \frac{\int_0^{r} f(u)du}{r}, \quad \x\in \mathbb{R}^3,
\ee
where $f$ is an even nonnegative function. We also assume that
\be
\lim_{r\rightarrow \infty} \int_0^{r} f(u)du = 0,
\ee
which is true for either the difference kernel or
the truncated mollified kernel discussed above.
Then its Fourier transform $\widehat{F}$ can be calculated using spherical coordinates:
\be
\ba
\widehat{F}(k) &= 
\int_0^\infty \int_0^\pi \int_0^{2\pi} e^{-i|k| r \cos\theta}F(r) r^2 \sin\theta
drd\theta d\phi\\
&=2\pi \int_0^\infty \int_0^\pi e^{-i|k| r\cos\theta} F(r) r^2\sin\theta drd\theta\\
&=\frac{2\pi}{i|k|}\int_0^\infty \left(e^{i|k| r}-e^{-i|k| r}\right)
F(r) r dr\\
&= \frac{2\pi}{i|k|}\int_0^\infty \left(e^{i|k| r}-e^{-i|k| r}\right)
\left(\int_0^r f(u)du\right) dr\\
&= \frac{2\pi}{|k|^2}\int_0^\infty \left(e^{i|k| r}+e^{-i|k| r}\right)
f(r) dr\\
&= \frac{2\pi\hat{f}(k)}{k^2},
\ea
\ee
where the first equality uses the expression of the Fourier transform in spherical
coordinates, the fifth equality follows from integration by parts, and the last
equality follows from the assumption that $f$ is even.

\subsection{Integral representations of kernels}\label{sec:integralrepresentation}

A unified approach to developing a telescoping sum for the kernel of 
interest is to begin with an 
integral representation in terms of exponential functions or
Gaussians. Many kernels are equipped with such classical representations, suitable
for the \acron framework.
Some are listed below.
\begin{enumerate}[label=(\alph*)]
\item The power function \eqref{powerfunction} in $\mathbb{R}^d$
\cite{beylkin2010acha}.
  \be\label{powersog}
  \ba
  \frac{1}{r^\alpha}
  &=\frac{1}{\Gamma\left(\alpha\right)}
  \int_0^\infty e^{-r t} t^{\alpha-1}dt,\\
  &=\frac{1}{\Gamma\left(\alpha/2\right)}
  \int_0^\infty e^{-r^2 t} t^{\alpha/2-1}dt,\\
  &=\frac{2}{\Gamma\left(\alpha/2\right)}
  \int_0^\infty e^{-r^2 t^2} t^{\alpha-1}dt.
  \ea
  \ee
\item Green's function for the Laplace operator \eqref{laplaceker}:
  \be\label{laplacesog}
  \ba
  -\frac{1}{2\pi}\log r&=\frac{1}{4\pi}\int_0^\infty \frac{1}{t}e^{-r^2/(4t)}dt
  = \frac{1}{2\pi}\int_0^\infty e^{-r^2t^2}\frac{dt}{t},\\
  \frac{1}{4\pi r}&=\frac{1}{8\pi^{3/2}}\int_0^\infty \frac{1}{t^{3/2}}e^{-r^2/(4t)}dt
  = \frac{1}{2\pi^{3/2}}\int_0^\infty e^{-r^2t^2}dt.
  \ea
  \ee
Some of these formulas must be interpreted in a distributional sense. The first 
expression in each line can be understood to describe the connection 
between the Laplace kernel and the heat kernel, viewing the 
solution to the Poisson equation as the solution
to the initial value problem for the heat equation with the same forcing term,
as the solution to the heat flow problem reaches its equilibrium state as 
$t \rightarrow \infty$.
\item Green's function for 
  the Yukawa operator \eqref{yukawakernelfouriertransform}, (see
  \cite{bertoglio2012cpc,greengard2018sisc}):
  \be\label{yukawasog}
  \ba
  \frac{1}{2\pi} K_0(\lambda r)
  & = \frac{1}{2\pi}\int_0^\infty e^{-r^2t^2-\frac{\lambda^2}{4t^2}}\frac{dt}{t},
  \quad \x\in\mathbb{R}^2,\\
  \frac{1}{4\pi} \frac{e^{-\lambda r}}{r}
  & = \frac{1}{2\pi^{3/2}}\int_0^\infty e^{-r^2t^2-\frac{\lambda^2}{4t^2}}dt,
  \quad \x\in\mathbb{R}^3.
  \ea
  \ee
\item Green's function for the Helmholtz operator \eqref{helmholtzker}, (see
  \cite{beylkin2009jcp}):
  \be\label{helmholtzsog}
  \ba
  \frac{i}{4} H_0^{(1)}((\kappa+i\lambda) r)
  & = \frac{1}{2\pi}\int_0^\infty e^{-r^2t^2+\frac{(\kappa+i\lambda)^2}{4t^2}}\frac{dt}{t}, \quad \x\in\mathbb{R}^2,\\
  \frac{1}{4\pi} \frac{e^{i(\kappa+i\lambda) r}}{r}
  & = \frac{1}{2\pi^{3/2}}\int_0^\infty e^{-r^2t^2+\frac{(\kappa+i\lambda)^2}{4t^2}}dt, \quad \x\in\mathbb{R}^3.
  \ea
  \ee

\item Inverse multiquadric:
  \be\label{reglarizedkernelsog}
  \frac{1}{\sqrt{r^2+a^2}}
  = \frac{2}{\sqrt{\pi}}\int_0^\infty e^{-(r^2+a^2)t^2}dt.
  \ee
\item Inverse quadratic:
  \be\label{invquadratic}
  \frac{1}{r^2+a^2}=\int_0^\infty e^{-(r^2+a^2)t}dt.
  \ee
\end{enumerate}

\subsection{Prolate spheroidal wave functions of order zero}

Suppose that $c>0$ is a real number.
Prolate spheroidal wave functions (PSWFs) of order zero are eigenfunctions of 
the integral operator
$F_c: L^2[-1,1] \rightarrow L^2[-1,1]$, defined by the formula
\be\label{pswfintegraloperator}
F_c[\phi](x) = \int_{-1}^1 \phi(t)e^{icxt}dt.
\ee
It is known~\cite{osipov2013}
that the eigenfunctions $\psi_0^c$, $\psi_1^c$, $\ldots$ of $F_c$ are purely
real, orthonormal, and complete in $L^2[-1, 1]$. The even--numbered functions are
even; the odd--numbered ones are odd. Each function $\psi_n$ has exactly $n$ simple
roots in $(-1, 1)$. All eigenvalues $\lambda_n$ of $F_c$ are nonzero and simple; the
even--numbered ones are purely real, and the odd--numbered ones are purely imaginary;
in particular, $\lambda_n = i^n |\lambda_n|$, for every integer $n\ge 0$.

The PSWFs provide a natural tool for the study of band-limited functions on an interval.
Indeed, a function $f:\mathbb{R}\rightarrow \mathbb{C}$ is band-limited with band limit $c$
if for all real $x$,
\be
f(x)=\int_{-1}^1 \sigma(t)e^{icxt}dt,
\ee
for some $\sigma\in L^2[-1,1]$. It is known that if
\be
\int_{-1}^1 |\sigma(t)|^2dt=1,
\ee
then
\be\label{pswfoptimalproperty}
\int_{-1}^1 |f(x)|^2dx\le |\lambda_0(c)|^2.
\ee
The equality in \eqref{pswfoptimalproperty} occurs only if $\sigma=\psi_0^c$.

\subsection{Window functions}\label{sec:windownfunctions}
Window functions are widely used in signal processing and numerical analysis. 
In this paper, we focus on two window functions - 
the Gaussian~\eqref{gaussiankernel} and the PSWF
of order zero $\psic$. 
The Gaussian window function has a number of compelling properties. First, it
appear naturally in the integral representations of many other kernels (as seen above).
Second, Gaussian and its Fourier
transform are both smooth and have explicit expressions via elementary functions. 
Third, Gaussians
are the only class of functions whose $d$-dimensional version is simply the tensor product
of its one-dimensional version. 
That is, if $f(\x)=\prod_{i=1}^d f(x_i)$ for $\x\in \mathbb{R}^d$ with
$d\ge 2$, then $f(\x)=e^{-|\x|^2/\sigma}$ for some $\sigma$. Fourth, 
Gaussians satisfy several optimality properties, including minimizing 
the Heisenberg uncertainty for $L^2$ functions (the product of the second moment of 
$f$ and the second moment of $\widehat{f}$ divide by 
the product $\|f\|_2 \, \| \widehat{f}\|_2$)~\cite{dym1975,folland1997jfaa}.
While they work extremely well as window functions in finite precision (ignoring exponentially
small tails), they are well-known to be sub-optimal in terms of computational efficiency.
In particular, the Heisenberg property does not imply that a Gaussian is an optimal
window for band-limited functions. For that, the best choice, in some sense, is
the PSWF of order zero $\psic$.
Indeed, as shown in
\eqref{pswfoptimalproperty}, $\psic$ is the function with support in $[-1,1]$
with minimal $L^2$-norm (energy) outside the frequency interval $[-c,c]$. Moreover, if we
truncate $\psic$ at $\pm c$ and its Fourier transform at $\pm c$, then
\be\label{pswffouriertransform}
\hpsic(k)=\lambda_0 \psic(k/c),
\ee
where $\lambda_0$ is the associated eigenvalue of the operator defined in
\eqref{pswfintegraloperator}. Thus, the truncated $\psic$ is the analog of the Gaussian 
on a finite interval in the sense that the Fourier transform is the 
original function up to suitable scaling.

\begin{remark}
  It has been observed that the Kaiser-Bessel function and the 
  ``exponential of semicircle'' (ES) function are very close
  to $\psic$ in terms of the Fourier expansion length~\cite{barnett2021acha},
  and hence in terms of efficiency of approximation.
  The Kaiser-Bessel function and the ES function have the advantage of having a
  closed form expression, which is not the case for $\psic$. 
  However, the evaluation of all such special functions can be accelerated
  through the use of piecewise polynomial approximation, after which 
  the PSWF is as easy to evaluate as any other window function.
  In practice, we have found that the PSWF is slightly better than the Kaiser-Bessel
  or ES kernels in terms of the discretized Fourier expansion length,
  especially when the required accuracy is not very high. 
  Thus, we use either the Gaussian or the PSWF as window functions in our 
  numerical experiments.
\end{remark}

\end{document}